\documentclass[a4paper,12pt]{amsart}

\usepackage[pagebackref,linktocpage=true,colorlinks=true,linkcolor=Blue,citecolor=BrickRed,urlcolor=RoyalBlue]{hyperref}
\usepackage[alphabetic,msc-links,abbrev]{amsrefs} 

\usepackage{amsmath}
\usepackage{amssymb}
\usepackage{amsthm}
\usepackage{mathtools}
\usepackage{mathrsfs} 
\usepackage{graphicx}
\usepackage{enumitem}

\usepackage[usenames,dvipsnames,svgnames,table]{xcolor}
\usepackage{skak} 

\usepackage[colorinlistoftodos,prependcaption,textsize=tiny,
textwidth=0.75in]{todonotes}
\newtheorem{theorem}{Theorem}[section]
\newtheorem{lemma}[theorem]{Lemma}
\newtheorem{proposition}[theorem]{Proposition}
\newtheorem{definition}[theorem]{Definition}
\newtheorem{corollary}[theorem]{Corollary}

\theoremstyle{remark}
\newtheorem{remark}[theorem]{Remark}

\theoremstyle{definition}

\numberwithin{equation}{section}

\def\R{\mathbb{R}}

\def\d{\partial}
\renewcommand{\div}{\textnormal{\textrm{div}}}

\def\dif{{\mathrm d}}
\def\ep{\varepsilon}

\def\supp{\textnormal{\textrm{supp}}}
\def\qed{\hfill$\square$}
\def\ath{\theta_{\ast}}
\def\vv{\vec{\mathrm{v}}}
\def\x{\vec{x}}

\def\X{\vec{X}}

\def\normal{\vec{N}}
\def\un{\hat{n}}
\def\hx{\hat{x}}
\def\hy{\hat{y}}
\def\hz{\hat{z}}
\def\hr{\hat{r}}

\def\hrho{\hat{\rho}}
\def\hth{\hat{\theta}}

\def\mB{\mathcal{B}}
\def\mC{\mathcal{C}}

\def\mG{\mathcal{G}}

\def\mL{\mathcal{L}}
\def\mQ{\mathcal{Q}}

\def\mS{\mathcal{S}}

\def\fk{\mathfrak{k}}
\def\fl{\mathfrak{l}}

\def\fG{\mathfrak{G}}
\def\fU{\mathfrak{U}}
\def\fS{\mathfrak{S}}

\def\tA{\tilde{A}}
\def\teta{\tilde{\eta}}

\DeclareMathOperator{\sinc}{sinc}

\renewcommand{\div}{\textnormal{\textrm{div}}}
\newcommand\absm[1]{\langle #1\rangle}

\begin{document}

\title{
The Dirichlet-Neumann Operator for Taylor's Cone
	%
    %
}
\author{
	Yucong {\sc Huang} and
	Aram {\sc Karakhanyan} 
}
\thanks{The research was partially supported by EPSRC grant   
EP/S03157X/1 "Mean curvature measure of free boundary".}
\date{}

\begin{abstract}
The aim of this paper 
is to analyze the Dirichlet-Neumann operator 
in axially symmetric conical domains. 
We provide a constructive treatment of the generic singularity at the vertex by using 
a new coordinate system that maps the 
conical domain to a strip. Building upon the paradifferential theory, 
we then establish our main Sobolev estimates. 
We also find the shape derivative, the linearization formula, and the cancellation property 
for  the Dirichlet-Neumann operator. 
Our results can be viewed as the first step towards establishing the 
mathematical framework for the perturbations of 
Taylor's cone which appears in the jet break-up control.

\end{abstract}

\maketitle

{\hypersetup{linkcolor=blue}
\setcounter{tocdepth}{1}
	\tableofcontents
}

\section{Introduction}
Let us consider a conical solid of revolution $\Omega=\{\x\in\R^3 \vert F(\x)<0 \}$
given by $F(\x) \vcentcolon \R^3 \to \R$. The Dirichlet-Neumann operator (DN operator for short), associated with $F$, is a functional defined by
\begin{gather}\label{DNOp}
	G[F](\psi) \vcentcolon = \nabla F \cdot \nabla \Psi \big\vert_{\d \Omega}, \qquad \text{for functions } \ \psi\vcentcolon \d\Omega \to \R,
	\end{gather}
where $\Psi$ is the solution to
	\begin{gather}\quad  \left\{ \begin{aligned}	\label{harmonic} 
		&\Delta \Psi = 0 && \text{in } \Omega,\\
		& \Psi = \psi && \text{on } \d\Omega.
	\end{aligned}\right.
\end{gather}
In this paper, we examine the properties of $G[F](\cdot)$ in Sobolev spaces. We also derive its functional derivative with respect to the perturbation of surface $F$.
 
The main motivation for studying conical DN operator stems from the mathematical analysis of Taylor's cone. This phenomenon occurs when electric field is applied to a droplet of conducting fluid, which leads to the formation of free boundary surface with conical cusp. The inviscid fluid occupying the time-dependent bulk domain $\Omega(t)\subset \R^3$ is governed by the equations:   
\begin{equation}\label{euler}
\rho(\d_t u +u \cdot \nabla u)=-\nabla p+\div\, \tau_E , \qquad \div u = 0 \qquad \text{ in } \ \Omega(t),
\end{equation}
where $u$ is the velocity, $p$ is the pressure, $\rho>0$ is the constant density, and $\div\,\tau_E$ is
the electric force acting on the fluid, with the tensor $\tau_E$ given by 
\begin{equation*}
\tau_E=\epsilon E\otimes E-\frac\epsilon2|E|^2\mathbb{I}. 
\end{equation*}
Here, $E$ is the electric field in vacuum domain, $\epsilon>0$ is the vacuum permittivity, and $\mathbb{I}$ is the $3\times 3$ identity matrix. Let $\varphi$ be the electric potential, i.e. $E=-\nabla \varphi$ with
\begin{equation*}
	\Delta \varphi=0 \quad \text{ in } \ \R^3\backslash \Omega(t), \qquad  \varphi \vert_{\d\Omega(t)} = 0. 
\end{equation*}
Note that the boundary condition $\varphi\vert_{\d\Omega(t)}=0$ means that the tangential component of the electric field $E$ at the fluid-vacuum boundary vanishes. This is a natural assumption since almost all fluids in question have some, if small, conductivity, so the appropriate boundary condition for an equilibrium situation is that of a conductor.
Then the tensor $\tau_E$ can be written in terms of $\varphi$ as:
\begin{equation*}
\tau_E=
\epsilon\begin{pmatrix}
-\frac{1}{2}|\nabla \varphi|^2+ |\d_x\varphi|^2 & \d_x\varphi \d_y\varphi & \d_x \varphi \d_z\varphi\\
\d_y\varphi \d_x\varphi& -\frac{1}{2}|\nabla \varphi|^2+|\d_y\varphi|^2 & \d_y\varphi \d_z\varphi\\
\d_z\varphi \d_x\varphi & \d_z\varphi \d_y\varphi & -\frac{1}{2}|\nabla \varphi|^2+ |\d_z\varphi|^2 \\
\end{pmatrix}
\end{equation*}
and consequently the divergence of above tensor can be written as the gradient:
\begin{equation*}
\div\, \tau_E=\frac\epsilon2\nabla(|\nabla\varphi|^2).
\end{equation*}
In addition, Young-Laplace law asserts that the pressure at vacuum-fluid interface is: 
\begin{equation*}
	p\vert_{\d\Omega(t)} = -\kappa \mathcal{H},
\end{equation*} 
where $\mathcal{H}$ is the mean curvature of $\d\Omega(t)$ and $\kappa>0$ is the coefficient of surface tension.
 
Suppose that the fluid is irrotational, then there is a velocity potential $\Psi$ such that
\begin{equation*}
	u=\nabla \Psi \qquad \text{with} \qquad  \Delta \Psi = 0 \ \text{ in } \ \Omega(t). 
\end{equation*}
If the fluid-vacuum boundary surface is described by $F(t,\x)=0$, then $F$ evolves along the particle path. Thus it satisfies the evolution equation: 
\begin{equation}\label{kine}
	\d_t F + \nabla F \cdot \nabla \Psi=0, \quad \mbox{on } \  \d\Omega(t)\vcentcolon=\big\{\x\in\R^3 \ \big\vert \ F(t,\x)=0 \big\}.
\end{equation} 
Moreover, integrating the first equation of (\ref{euler}) over space gives the Bernoulli's law, which states that on $\d\Omega(t)$ one has 
\begin{equation}\label{eq:bbblll}
\d_t \Psi +\frac12|\nabla\Psi|^2=\frac\epsilon{2\rho}|\nabla\varphi|^2
+
\frac\kappa\rho \mathcal{H}, \quad \text{where } \ \mathcal{H} \vcentcolon= -\dfrac{1}{2} \div \Big(\dfrac{\nabla F}{|\nabla F|}\Big).
\end{equation}
To complete the statement of the problem we have to add the boundary conditions:
\begin{equation}\label{bdry}
 \varphi|_{\d\Omega(t)}=0,\quad \lim_{|\x|\to \infty}|\nabla \Psi|=0, \quad \lim_{|\x|\to \infty} |\nabla\varphi| = 0. 
\end{equation}
The second and third conditions mean that the fluid at infinity is at rest, and there is no electric field at infinity.

Let $(r,\theta,\alpha)$ be the spherical coordinate where $r$ is the radial variable; $\theta$ is the polar angle; $\alpha$ is the azimuthal angle. For a conical domain symmetric about $\hat{z}$-axis, the surface function $F$ can be expressed in the spherical coordinates as
\begin{equation*}
	F(t,r,\theta) = \theta - \Theta(t,r) \quad \text{with} \quad 0<\Theta<\pi \quad \text{in} \quad (t,r)\in (0,\infty)^2.
\end{equation*}
The mean curvature $\mathcal{H}(\Theta)= -\frac{1}{2}\div(\frac{\nabla F}{|\nabla F|})$ for this surface is given by:
\begin{equation}\label{HTheta}
	\mathcal{H}(\Theta) = \d_r\Big( \dfrac{r \d_r \Theta}{2\sqrt{1+r^2 |\d_r\Theta|^2}} \Big) + \dfrac{\d_r \Theta}{\sqrt{1+r^2 |\d_r\Theta|^2}} - \dfrac{\cot\Theta}{2r \sqrt{1+r^2 |\d_r\Theta|^2}}.
\end{equation}
Since function $\Theta(t,r)$ determines the surface, we can redefine $G[\Theta](\cdot)$ as the DN operator associated with $\d\Omega=\{\theta-\Theta=0\}$. Then from (\ref{DNOp}), one can verify that 
\begin{equation}\label{DN-Theta}
G[\Theta](\psi) = \Big\{ \dfrac{\d_\theta\Psi}{r^2} - \d_r \Theta \d_r \Psi \Big\}\Big\vert_{\theta=\Theta} \  \text{where } \left\{ \begin{aligned}
	&\Delta \Psi = 0 \quad \text{ in } \ \{0<\theta<\Theta\},\\
	&\psi(t,r)\vcentcolon= \Psi\vert_{\d\Omega} = \Psi\big(t,r,\Theta(t,r)\big).
\end{aligned}\right. 
\end{equation}
Furthermore, one observes that $\varphi$ solves the elliptic boundary value problem:
\begin{equation*}
	\Delta \varphi =0 \ \text{ in } \ \{\Theta\le \theta<\pi\}, \quad \varphi =0 \ \text{ on }\  \{\theta=\Theta\}, \quad |\nabla \varphi| \to 0 \ \text{ as } \ r\to \infty.
\end{equation*}
Thus the term $|\nabla \varphi|^2\vert_{\theta=\Theta}$ is solely determined by the surface function $\Theta$, and it can be considered as a functional of $\Theta$. More specifically, according to Appendix \ref{append:STEF},
\begin{gather}
	\text{if we set } \ \xi(r) \vcentcolon= -C \sqrt{r} P_{\frac{1}{2}}\big(\!\!-\!\cos\Theta(r) \big) \ \text{ for arbitrary constant } C\neq 0 \text{ then} \nonumber\\
	|\nabla\varphi|^2 \big\vert_{\theta=\Theta} =  \dfrac{\xi^2}{4r^2} + \bigg| r G[\Theta](\xi) + \dfrac{C}{\sqrt{r}} P_{\frac{1}{2}}^1\big(\!\!-\!\cos\Theta\big) \bigg|^2 - \dfrac{\big| \d_r \xi - r^2 \d_r \Theta G[\Theta](\xi) \big|^2}{1+ r^2 |\d_r\Theta|^2}, \label{nablaVarphi}
\end{gather}
where $P_{\nu}^{\mu}(x)$ denotes the associated Legendre function. Using (\ref{HTheta})--(\ref{nablaVarphi}), the original problem (\ref{kine})--(\ref{bdry}) can be reformulated as
\begin{equation}\label{Taylor-Zakharov}
	\left\{
		\begin{aligned}
			&\d_t \Theta - G[\Theta](\psi) = 0,\\
			&\d_t\psi + \dfrac{\big|\d_r\psi|^2}{2}-\dfrac{| r\d_r \Theta \d_r \psi+ r G[\Theta](\psi)\big|^2}{2 (1+r^2|\d_r \Theta|^2)} - \dfrac{\kappa}{\rho} \mathcal{H}(\Theta) - \dfrac{\epsilon |\nabla \varphi|^2}{2\rho}\Big\vert_{\theta=\Theta} = 0.
		\end{aligned}\right.
\end{equation}
The derivations for (\ref{HTheta})--(\ref{Taylor-Zakharov}) are included in Appendix \ref{append:AS-Zakh}. Note that (\ref{Taylor-Zakharov}) resembles Zakharov's system for water wave problem (See \cite{Zak1968}), except for two main distinctions. The first is that DN operator $G[\Theta](\cdot)$ and mean curvature $\mathcal{H}(\Theta)$ in (\ref{Taylor-Zakharov}) are constructed in a conical domain while the ones in Zakharov's system are formulated in half space $\R^3_{+}$. The second difference is that the electric force term $|\nabla \varphi|^2\vert_{\theta=\Theta}$ is present in the equation for Bernoulli's law.   

The famous Taylor's solution \cite{Taylor} given by $\varphi \propto \sqrt{r} P_{\frac12}(-\cos\theta)$ is the equilibrium solution where the free boundary surface is a cone, i.e. 
$\Theta=\ath$ for some constant $\ath\in(0,\pi)$. For this to hold, one must match the orders in $r$ between surface tension term and electric force term, under the assumption that $\Theta=\ath$. The surface tension term $\mathcal{H}(\Theta)$, in the Bernoulli law, 
is of order $\frac1r$, and to match it, $\varphi$
must be a harmonic function with the order: $|\nabla\varphi|^2 \sim \frac{1}{r}$. This leads to the spherical harmonic function $\varphi=C_{\ast} \sqrt{r} P_{\frac12}(-\cos\theta)$, with some renormalization constant $C_{\ast}$. To elaborate further, we let $\ath\in (0,\pi)$ be the unique root of equation $P_{\frac{1}{2}}(-\cos\ath)=0$. Note that $\ath\approx 0.2738\,\pi$. If we set
\begin{equation*}
	\bar{\Theta}\vcentcolon=\ath,\quad \bar{\Psi}\vcentcolon=0, \quad \bar{\varphi}\vcentcolon= C_{\ast}\sqrt{r} P_{\frac12}(-\cos\theta), \ \text{ where } \ C_{\ast}\vcentcolon= -\dfrac{\sqrt{\kappa\cot\ath}}{\sqrt{\epsilon}P_{1/2}^1(-\cos\ath)},
\end{equation*}
then it can be verified that $(\bar{\Theta},\bar{\Psi},\bar{\varphi})$ is a stationary solution to (\ref{Taylor-Zakharov}).
Clearly, the term $|\nabla\bar{\varphi}|^2\vert_{\theta=\ath}$ or $\mathcal{H}(\bar{\Theta})$ individually blows up at the origin with order $\frac{1}{r}$. However, as we have seen above, this singularity gets mutually cancelled in (\ref{eq:bbblll}). For a general perturbed solution, one may think of 
$(\bar{\Theta},\bar{\Psi},\bar{\varphi})$ as the first order approximation of the solution solving (\ref{Taylor-Zakharov}). From here, one can then consider the second order approximation given by the form $\tilde{\Theta}\vcentcolon= \Theta -\ath$ and $\tilde{\varphi}:=\varphi-C_{\ast}\sqrt{r} P_{\frac12}(-\cos\theta)$ such that $\tilde{\Theta}, \, \Psi, \, \tilde{\varphi} \to 0$ uniformly as $r\to \infty$, $r\to 0^+$, and $\tilde{\varphi}$ solves the Dirichlet boundary value problem
\begin{equation}
\Delta \tilde{\varphi} = 0 \ \text{ in } \ \{\Theta \le \theta< \pi \}, \qquad \tilde{\varphi} = -C_{\ast} \sqrt{r} P_{\frac{1}{2}}\big(-\cos\Theta(r)\big) \ \text{ on } \ \{\theta=\Theta\}. 
\end{equation}
As the first step toward the existence and uniqueness of such solution to (\ref{Taylor-Zakharov}), one must study the behaviour of operator $G[\Theta](\psi)$ when $\Theta$, $\psi$ belongs to some Sobolev space. Moreover, the derivative of $G[\Theta](\psi)$ with respect to the functional perturbation over $\Theta$ will also be necessary. These are the main aims of this paper.

In the formulation (\ref{DN-Theta}), we immediately encounter the problem of geometric singularity at the origin for conical domains. This also suggests that the functional space for which solution belongs to must be a weighted Sobolev space in spatial variable $r$. In order to resolve this issue, we introduce the new variable $\sigma =- \ln r \in \R$ and new set of unknowns:
\begin{equation*}
	\Phi(\sigma,\theta)\vcentcolon= \sqrt{e^{-\sigma}}\, \Psi\big(e^{-\sigma},\theta\big), \quad \phi(\sigma) \vcentcolon= \sqrt{e^{-\sigma}}\, \psi\big(e^{-\sigma}\big), \quad \eta(\sigma)\vcentcolon= \Theta\big(e^{-\sigma}\big).
\end{equation*} 
Thus the conical domain $\{(r,\theta)\, \vert \, r\in(0,\infty), \ 0<\theta\le\Theta(r) \}$ is transformed into a curved strip given by $\{(\sigma,\theta)\, \vert \, \sigma\in\R, \ 0<\theta\le\eta(\sigma) \}$, and the problem (\ref{harmonic}) becomes:
\begin{equation}\label{PhiDBVP}
	\left\{\begin{aligned}
		&\d_\sigma(\d_\sigma \Phi \sin\theta) + \d_\theta(\d_\theta \Phi \sin\theta) - \dfrac{\sin\theta}{4} \Phi = 0 && \text{for } \ \sigma\in\R \ \text{ and } \ \theta\in(0,\eta(\sigma)],\\
		& \Phi\big(\sigma,\eta(\sigma)\big) = \phi(\sigma), \qquad \d_{\theta}\Phi(\sigma,0)=0 && \text{for } \ \sigma\in\R.
	\end{aligned}\right.
\end{equation}
Moreover, under this change of variables and unknowns, the DN operator becomes
\begin{equation*}
	G[\Theta](\psi)\vert_{r = e^{-\sigma}} = e^{5\sigma/2} \Big\{ \mG[\eta](\phi) - \dfrac{1}{2} \phi \d_\sigma \eta \Big\},
\end{equation*}
where $\mG[\eta](\phi)$ is the auxiliary DN operator defined as:
\begin{equation}\label{auxmG}
	\mG[\eta](\phi) \vcentcolon= \big( \d_\theta \Phi -\d_\sigma \eta \d_\sigma \Phi \big)\big\vert_{\theta=\eta(\sigma)}, \quad \text{where $\Phi$ solves (\ref{PhiDBVP}).}  
\end{equation}
Therefore the analysis of $G[\Theta](\psi)$ amounts to the analysis of $\mG[\eta](\phi)$ in this new formulation, and our main results below will solely focus on the operator $\mG[\eta](\phi)$ instead. As we shall see, this formulation will not only map the singularity at $r \to 0$ to the far field $\sigma \to \infty$, thus making the analysis manageable, it also allows us to utilize the theory of paradifferential calculus since the spatial domain under consideration is now the real line $\sigma\in\R$, as opposed to the half line $r\in(0,\infty)$ in the original spherical coordinate formulation. In particular, when the boundary is a cone, meaning $\eta(\sigma)\equiv \ath$ for constant angle $\ath\in(0,\pi)$, the solution to (\ref{PhiDBVP}) is given by the convolution:
\begin{equation*}
	\Phi(\sigma,\theta) = \int_{\R}\!\! \phi(s) K(\sigma-s,\theta)\, \dif s \quad \text{ with } \quad K(\sigma,\theta)\vcentcolon= \int_{\R}\! \dfrac{\fk(\zeta,\theta)}{\fk(\zeta,\ath)} \dfrac{e^{i\sigma \zeta}}{2\pi} \, \dif \zeta, 
\end{equation*}
where $\fk(\zeta,\theta)\vcentcolon= P_{-\frac{1}{2}+i\zeta}(\cos\theta)$ is the associated Legendre's function with complex index. Note that $\fk(\zeta,\theta)$ is real valued, and it is also known as the conical function (See \cite{ZK66}). By (\ref{auxmG}), $\mG[\ath](\phi)$ can be expressed as a psuedodifferential operator with symbol:
\begin{equation*}
	\mG[\ath](\phi) = \d_\theta \Phi(\sigma,\theta)\vert_{\theta=\ath} = \dfrac{\fk_{1}(\d_\sigma,\ath)}{\fk(\d_\sigma,\ath)} \phi \qquad \text{ where } \  \fk_1(\zeta,\theta)\vcentcolon= P_{-\frac{1}{2}+i\zeta}^1(\cos\theta).
\end{equation*}
This is in contrast to the first order symbol $|D| \tanh(h|D|)$ for 3D water wave problem with a finite depth $h>0$ (See \cite{Craig}). It will be shown in Section \ref{ssec:kbounds} that the kernel $K(\sigma,\theta)$ satisfies certain integral bounds which allow one to obtain the Sobolev estimates for $\Phi=K(\cdot,\theta)\star \phi$ and $\mG[\ath](\phi)$. These solutions will also be used to construct $\mG[\eta](\cdot)$ for general non-flat conical boundaries $\theta=\eta(\sigma)\neq \ath$.
   
The above heuristic observations suggest that there must be a complete regularity theory of $\mG[\eta](\phi)$ in Sobolev spaces $H^s(\R)$ in $\sigma$ coordinate, for some $s>2$, which also allows one to derive the linearization of $\mG[\eta](\phi)$ with respect to perturbation in $\eta$. Finally, we remark that the $H^s(\R)$ space in $\sigma$ coordinate is equivalent to a certain weighted Sobolev space in $r$ coordinate, and we refer readers to Proposition \ref{prop:normEquiv} in the appendix for more details. 





\section{Main results}
The main results of this paper are the following $3$ theorems for $\mG[\eta](\phi)$ in (\ref{auxmG}):  
\subsection{Behaviour in Sobolev spaces} The first result states that $\phi\mapsto \mG[\eta](\phi)$ can be constructed as a bounded linear map in Sobolev spaces. It also provides an explicit formula for estimate constant in terms of surface function $\eta$.  
\begin{theorem}\label{thm:DNSob}
	Denote $\teta\vcentcolon=\eta-\ath$ with $\ath\in(0,\pi)$. Suppose $\teta\in H^{s+\frac{1}{2}}(\R)$ with $s>\frac{5}{2}$. Then $\phi\mapsto \mG[\eta](\phi)$ exists as a bounded linear map from $H^{m}(\R)$ to $H^{m-1}(\R)$ for $\frac{1}{2}<m \le s$, and from $H_0^{\frac{1}{2}}(\R)$ to $H^{-\frac{1}{2}}(\R)$ for $m=\frac{1}{2}$. Moreover, there exists a constant $C=C(\ath)>0$ such that for all $\frac{1}{2}\le m \le s$,
	\begin{equation*}
		\|\mG[\eta](\phi)\|_{H^{m-1}(\R)} \le C \dfrac{\ath^{-1}+\fU_s(\teta)}{\sin\big( \ath - \|\teta\|_{L^{\infty}(\R)} \big)} \Big|\dfrac{\fU_s(\teta)}{\fl(\teta)}\Big|^{m+\frac{1}{2}} \|\phi\|_{H^{m}(\R)} \quad \text{for } \ \phi\in H^{m}(\R), 
	\end{equation*}
	where the functional $\fU_s(\teta)$ and $\fl(\teta)$ are defined as: $(\text{with } \sinc\theta \vcentcolon= \frac{\sin\theta}{\theta} )$
	\begin{subequations}
		\begin{gather*}
			\fl(\teta) \vcentcolon= \sinc\big( \|\teta\|_{L^{\infty}} + \ath \big)\min\Big\{ \dfrac{1}{2}, \dfrac{\big(\ath-\|\teta\|_{L^{\infty}}\big)^2}{1+ 2 \|\d_{\sigma}\eta\|_{L^{\infty}}^2 }, \dfrac{\big(\ath- \|\teta\|_{L^{\infty}}\big)^2}{4} \Big\},\\
			\fU_s(\teta) \vcentcolon= \max\big\{  \|\teta\|_{H^{s-\frac{1}{2}}}, \|\d_{\sigma}\eta\|_{H^{s-\frac{1}{2}}},  \|\teta^2\|_{H^{s-\frac{1}{2}}}, \|\teta\d_{\sigma} \eta \|_{H^{s-\frac{1}{2}}}, \| (\d_{\sigma}\eta)^2\|_{H^{s-\frac{1}{2}}} \big\}.
		\end{gather*}
	\end{subequations}
\end{theorem}

\subsection{Linearization of DN operator} The second theorem gives the precise shape derivative formula for $\mG[\eta](\phi)$.
\begin{theorem}\label{thm:shape}
Let $s>\frac{5}{2}$, $1\le k \le s$. Suppose $\phi\in H^{k}(\R)$ and $\teta\equiv \eta-\ath\in H^{s+\frac{1}{2}}(\R)$ satisfy $\|\teta\|_{L^{\infty}(\R)}<\min\big\{\ath\,, \pi-\ath\big\}$. Then there exists a neighbourhood $\mathcal{U}_{\eta} \subset H^{s+\frac{1}{2}}(\R)$ such that $\eta\in \mathcal{U}_{\eta}$ and the mapping
\begin{equation}
    \varrho \mapsto \mG[\varrho] (\phi) \in H^{k-1}(\R) \ \text{ is differentiable in } \ \varrho\in \mathcal{U}_{\eta} \subset H^{s+\frac{1}{2}}(\R) .
\end{equation}
Moreover, for $h\in H^{s+\frac{1}{2}}(\R)$, the shape derivative $\dif_{\eta} \mG[\eta] (\phi) \cdot h$ is given by
\begin{gather}
\begin{aligned}
	\dif_{\eta} \mG[\eta] (\phi) \cdot h \vcentcolon=& \lim\limits_{\ep\to 0} \dfrac{\mG[\eta+\ep h](\phi) - \mG[\eta](\phi)}{\ep}\\
	 =& - \mG[\eta]\big(h\mathcal{B}+V\big) - \d_{\sigma}\big(h V - \mB\big) + (h-\d_\sigma \eta) \Big\{ \frac{\phi}{4} - \mB \cot(\eta) \Big\},
\end{aligned}\nonumber\\
\text{where } \quad  \mathcal{B}\vcentcolon= \dfrac{\d_\sigma \eta \d_\sigma \phi + \mG[\eta](\phi)}{1+|\d_\sigma \eta|^2}, \qquad V \vcentcolon= \d_\sigma \phi- \mathcal{B} \d_\sigma \eta .\label{BV}
\end{gather}
\end{theorem}

\begin{remark}\label{rem:BV}
	The terms $\mB$ and $V$ in spherical coordinate system $(r,\theta)$ are expressed as:
	\begin{equation*}
		\mB\vert_{\sigma=-\ln r} = \sqrt{r} \d_{\theta} \Psi(r,\theta)\vert_{\theta=\Theta(r)}, \qquad V\vert_{\sigma=-\ln r} = r \d_{r} \big(\sqrt{r}\, \Psi\big) \big\vert_{\theta=\Theta(r)},
	\end{equation*}
	where $\Psi$ is the velocity potential function solving (\ref{harmonic}). Thus $\mB$ is the $r^{3/2}$ multiple of polar-angular velocity component at the boundary surface, which is $\frac{1}{r}\d_\theta \Psi\vert_{\theta=\Theta(r)}$, and $V$ depends on the radial velocity component $\d_r \Psi \vert_{\theta=\Theta(r)}$.   
\end{remark}

\subsection{Cancellation Property}
By Theorem \ref{thm:DNSob} and the definition of $\mB$, $V$ in (\ref{BV}), it can be verified that $\mB$, $V\in H^{s-1}(\R)$ for $(\teta,\phi)\in H^{s+1/2}(\R)\times H^s(\R)$. Therefore applying Theorem \ref{thm:DNSob} once again, we see that
\begin{equation*}
	\mG[\eta](\mB), \quad \mG[\eta](V), \quad \d_{\sigma} \mB, \quad \d_{\sigma} V \in H^{s-2}(\R).
\end{equation*}
In particular, this implies for general $h\in H^{s+1/2}(\R)$, the shape derivative $\dif_{\eta} \mG[\eta](\phi)\cdot h$ in Theorem \ref{thm:shape} loses regularity by an order of $2$. However, we show that a cancellation of higher order derivatives will occur in a specific linear combination of above four terms, so that it instead belongs to a more regular space $H^{s-1}(\R)$. This is stated as follows:
\begin{theorem}\label{lemma:cancel-Int}
	Suppose $\teta\in H^{s+\frac{1}{2}}(\R)$ with $s>\frac{5}{2}$. Let $\mB$, $V$ be given in (\ref{BV}). Then
	\begin{equation*}
		\mG[\eta](\mB+V) + \d_{\sigma}(V-\mB) \in H^{s-1}(\R).
	\end{equation*}  
\end{theorem}
\subsection{Organization of paper}
The paper is organized as follows: First, we reformulate the Dirichlet boundary value problem (\ref{harmonic}) into the curved strip domain in Section \ref{sec:DN-cone}. We also redefine the DN operator in the corresponding coordinate system. In Section \ref{sec:flat}, we consider the case when the surface is a cone with constant slope, described by the level set: $\Theta = \ath$ for fixed  $\ath\in(0,\pi)$. For this case, the solution is explicitly obtained as a convolution of boundary data with kernel involving conical Legendre functions. Using this form, we also derive its Sobolev estimates which are Lemmas \ref{lemma:bdryEst} and \ref{lemma:higherPhi}. In Section \ref{sec:nonflat}, we derive the elliptic boundary estimate for general conical surfaces. This is done by flattening the domain into a flat strip with coordinate transformation. Moreover, the construction for flat case given in Section \ref{sec:flat} is used as an extension function to ensure the existence and uniqueness of solutions to the elliptic Dirichlet boundary value problem (\ref{v-D}). Using these results, we construct the DN operator for general conical surfaces, and derive its Sobolev estimates in Section \ref{sec:DN-operator}. Section \ref{sec:shape} contains  the derivation of shape derivative for DN operator with respect to perturbation on the conical surface. This is achieved by finding an explicit solution to the Dirichlet boundary value problem (\ref{varpihEq}). The method used for this is very general and we believe it can be adopted to other similar problems. Finally, in Section \ref{sec:Stokes} we compute the coefficients of the Taylor expansion of the DN operator under the assumption that the perturbation on the conical surface, $\tilde{\eta}\equiv \eta-\ath$ and $\d_\sigma\eta$ are small.



\section{Dirichlet Boundary Value Problem for Conical Domain}\label{sec:DN-cone} 
The domain in consideration $\Omega_{\Theta}\subseteq \R^3$ is set to be 
\begin{equation*}
	\Omega_{\Theta}\vcentcolon= \bigg\{ (x,y,z)\in\R^3 \,\Big\vert\, \arccos\Big( \dfrac{z}{\sqrt{x^2+y^2+z^2}} \Big) \le \Theta\Big(\sqrt{x^2+y^2+z^2}\,\Big) \bigg\},
\end{equation*} 
for some function $\Theta\vcentcolon (0,\infty) \to (0,\pi)$. In addition, $\Omega_{\Theta}$ is considered to be a perturbation of the cone with slope $\tan\ath$: 
\begin{equation*}
 \textrm{Cone}(\ath)\vcentcolon=\Big\{(x,y,z)\in \R^3\,\Big\vert\, \arccos\Big(\dfrac{z}{\sqrt{x^2+y^2+z^2}}\Big)\le \ath \Big\}, \qquad \text{with } \ \ath\in(0,\pi).
\end{equation*}
More precisely, this means $\sup_{r>\infty}|\Theta(r)-\ath|< \min\{\ath, \pi-\ath\}$. The boundary can be expressed as the level set $\d\Omega_{\Theta}\vcentcolon= \big\{ \x\in\R^3 \,\big\vert\, F(\x) = 0 \big\}$ with
\begin{equation*}
	F(x,y,z)\vcentcolon=\arccos\Big(\dfrac{z}{\sqrt{x^2+y^2+z^2}}\Big)-\Theta\big(\sqrt{x^2+y^2+z^2}\big).
\end{equation*}
Let $\Psi(x,y,z)$ be the solution to linear Dirichlet boundary value problem:
\begin{subequations}\label{3Dharmonic-cone}
\begin{align}
&-\Delta \Psi = 0 && \text{for } \ (x,y,z)\in \Omega_{\Theta},\\
& \Psi\vert_{\d \Omega_{\Theta}} = \psi \vert_{\d \Omega_{\Theta}} && \text{for a given boundary data } \ \psi.
\end{align}
\end{subequations}
Then the DN operator for surface $\d\Omega_{\Theta}$, acting on $\psi\vcentcolon \d\Omega_{\Theta}\to \R$, is defined as:
\begin{equation}\label{0-DN}
	G[\Theta](\psi) \vcentcolon= \big(\nabla F \cdot \nabla \Psi\big) \big\vert_{F(\x)=0}.
\end{equation}
\subsection{Formulation in the spherical coordinate} Let $(r,\theta,\varphi)$ be the spherical coordinate with $r\in(0,\infty)$ being the radial variable; $\theta\in [0,\pi]$ being the polar angle; $\varphi\in[0,2\pi)$ being the azimuthal angle. The surface $\textrm{Cone}(\ath)$ under this coordinate is described by the level set $\{\theta=\ath\}$. If the domain $\Omega_{\Theta}$ and boundary data $\psi$ are axially symmetric, then so is the solution $\Psi$. Thus we can write  $\Psi=\Psi(r,\theta)$. Since $\Psi$ is a harmonic function, it is analytic at point $r>0$ and $\theta=0$. It follows that $\d_\theta \Psi(r,0)=0$ for $r>0$. Under the spherical coordinate, the boundary can be expressed as
\begin{equation*}
	\d\Omega_{\Theta}=\big\{ \big(r \cos\varphi \sin\theta, r \sin\varphi \sin\theta, r \cos\theta\big)\in \R^3 \,\vert\, r\ge 0, \ \varphi\in [0,2\pi),\ \theta=\Theta(r) \big\},
\end{equation*}
and it is also described by the equation: $F(r,\theta,\varphi)=\theta-\Theta(r) = 0$. The Laplacian with Dirichlet boundary condition in (\ref{3Dharmonic-cone}) can be reformulated as follows: Given a data $\psi(r)\vcentcolon (0,\infty)\to \R$, we set $\Psi(r,\theta)$ to be the solution to 
\begin{subequations}\label{ellip-cone}
\begin{align}
&\d_r (r^2 \d_r \Psi \sin\theta ) + \d_\theta ( \d_\theta \Psi   \sin\theta ) = 0 && \text{for } \ 0< \theta \le \Theta(r),\\
&\Psi\big(r,\Theta(r)\big) = \psi(r),\quad \d_\theta \Psi(r,0)   = 0 && \text{for } \ r\in(0,\infty).
\end{align}
\end{subequations}
Recall from (\ref{DN-Theta}) that the DN operator associated with $\Theta(r)$ in spherical coordinate is:
\begin{equation}\label{DN-cone}
G[\Theta](\psi) = \Big\{ \dfrac{\d_\theta \Psi}{r^2} - \d_r \Theta \d_r \Psi \Big\}\Big\vert_{\theta=\Theta(r)}.
\end{equation}
\subsection{Reformulation in the curved strip}
Set the coordinate $(\sigma,\theta)\in \R\times (0,\pi)$ as 
\begin{equation}
	\sigma = -\ln r, \qquad \theta = \theta \qquad \text{for } \ (r,\theta)\in(0,\infty)\times(0,\pi).
\end{equation}
With this, we define:
\begin{equation*}
	\Phi(\sigma,\theta)\vcentcolon= \sqrt{e^{-\sigma}}\, \Psi\big(e^{-\sigma},\theta\big), \quad \phi(\sigma) \vcentcolon= \sqrt{e^{-\sigma}}\, \psi\big(e^{-\sigma}\big), \quad \eta(\sigma)\vcentcolon= \Theta\big(e^{-\sigma}\big).
\end{equation*}
Then by the chain rule one has
\begin{subequations}\label{PhiC}
	\begin{gather}
		\d_\sigma \eta = -e^{-\sigma} \d_r \Theta \big( e^{-\sigma}\big), \qquad \d_r \Psi\vert_{r= e^{-\sigma}} = -  e^{3\sigma/2} \big\{ \d_\sigma \Phi + \dfrac{\Phi}{2} \big\},\\
		\d_r^2 \Psi \vert_{r=e^{-\sigma}} =  e^{5\sigma/2} \big\{ \d_\sigma^2 \Phi + 2 \d_\sigma \Phi + \dfrac{3}{4}\Phi \big\}.
	\end{gather}
\end{subequations}
Using the above identities and equations (\ref{ellip-cone}), it can be verified that $\Phi$ solves
\begin{subequations}\label{Phi-D}
	\begin{align}
		&\d_\sigma(\d_\sigma \Phi \sin\theta) + \d_\theta(\d_\theta \Phi \sin\theta) - \dfrac{\sin\theta}{4} \Phi = 0 && \text{for } \ \sigma\in\R \ \text{ and } \ 0< \theta \le \eta(\sigma),\\
		&\Phi\big(\sigma,\eta(\sigma)\big) = \phi(\sigma), \qquad \d_\theta \Phi(\sigma,0) = 0 && \text{for } \ \sigma\in\R.
	\end{align}
\end{subequations}
Moreover, by the chain rule identities (\ref{PhiC}), we also have
\begin{align*}
	\Big\{ \dfrac{\d_\theta \Psi}{r^2} - \d_r\Theta \d_r \Psi \Big\}\Big\vert_{r=e^{-\sigma}} 
	= e^{5\sigma/2} \big\{ \d_\theta \Phi - \d_{\sigma} \eta \d_\sigma \Phi -\dfrac{1}{2} \Phi \d_\sigma \eta \big\}.
\end{align*}
Substituting the above into the definition for DN operator (\ref{DN-cone}), one has
\begin{align*}
	G[\Theta](\psi)\vert_{r=e^{-\sigma}} 
	= e^{5\sigma/2}\Big\{ \big( \d_\theta \Phi -\d_\sigma \eta \d_\sigma \Phi \big)\big\vert_{\theta=\eta(\sigma)} - \dfrac{1}{2} \phi \d_\sigma \eta \Big\}.
\end{align*}
Given a pair of functions $(\eta,\phi)(\sigma)$, we define the auxiliary DN operator $\mG[\eta](\phi)$ as:
\begin{equation}\label{mG}
	\mG[\eta](\phi) \vcentcolon= \big( \d_\theta \Phi -\d_\sigma \eta \d_\sigma \Phi \big)\big\vert_{\theta=\eta(\sigma)},
\end{equation}
where $\Phi(\sigma,\theta)$ is the solution to (\ref{Phi-D}). Then we obtain
\begin{equation}\label{GmG}
	G[\Theta](\psi)\vert_{r = e^{-\sigma}} =  e^{5\sigma/2} \Big\{ \mG[\eta](\phi) - \dfrac{1}{2} \phi \d_\sigma \eta \Big\}.
\end{equation}

\section{Flat Cone Reformulated in the Strip Domain}\label{sec:flat}
Let $\Phi_{\ast}(\sigma,\theta)$ be the solution to (\ref{Phi-D}) with $\eta(\sigma)\equiv\ath \in (0,\pi)$ for all $\sigma\in\R$. Then $\Phi_{\ast}$ solves the linear equations: 
\begin{subequations}\label{ellip-cf}
\begin{align}
&\frac{\partial}{\partial\theta}(\d_\theta \Phi_{\ast} \sin\theta )+\d_{\sigma}^2 \Phi_{\ast} \sin\theta -\frac{\sin\theta}{4}\Phi_{\ast}=0 && \text{for } \ (\sigma,\theta)\in\R\times(0,\ath],\label{ellip-cf-a}\\
&\Phi_{\ast}(\sigma,\ath)= \phi(\sigma),\qquad \d_\theta \Phi_{\ast}  (\sigma,0) = 0 && \text{for } \ \sigma\in\R.\label{ellip-cf-b}
\end{align}
\end{subequations}
According to (\ref{mG})--(\ref{GmG}), the auxiliary DN operator for the surface $\{\theta=\ath\}$ under coordinate $(\sigma,\theta)$ is given by
\begin{equation}\label{Gath}
\mG[\ath](\phi) = e^{-5\sigma/2} G[\ath](\psi)\big\vert_{r = e^{-\sigma}} = \d_\theta \Phi_{\ast}(\sigma,\theta)\big\vert_{\theta=\ath}. 
\end{equation}
\subsection{Solutions as convolution with Poisson kernel}
Taking the Fourier transform with respect to $\sigma\in\R$ on both sides of (\ref{ellip-cf-a}), one has
\begin{equation}\label{Legendre'}
\d_\theta^2 \hat{\Phi}_{\ast} (\zeta,\theta)  +  \d_\theta \hat{\Phi}_{\ast}(\zeta,\theta) \cot\theta - \dfrac{1+4\zeta^2}{4} \hat{\Phi}_{\ast}(\zeta,\theta) = 0 \quad \text{ for } \ (\zeta,\theta)\in\R\times(0,\ath],
\end{equation}
where $\hat{\Phi}_{\ast}(\zeta,\theta)$ is the Fourier transform of $\sigma\mapsto \Phi_{\ast}(\sigma,\theta)$ with respect to variables $\sigma\xleftrightarrow[]{} \zeta$. Define the function $\mu=\mu(\zeta)$ as:
\begin{equation*}
\mu(\zeta)\vcentcolon= -\frac{1}{2} + i \zeta, \qquad \text{then one has} \qquad \mu(\mu+1)=-\dfrac{1+4\zeta^2}{4},
\end{equation*}
Moreover, set $\hat{f}(\zeta,x)\vcentcolon=\hat{\Phi}_{\ast}(\zeta,\arccos x)$ for $\cos\ath \le x < 1$. Then (\ref{Legendre'}) becomes
\begin{subequations}\label{Legendre}
\begin{align}
&(1-x^2)\d_x^2 \hat{f} - 2 x \d_x \hat{f} + \mu(\mu+1) \hat{f} = 0  && \text{for } \ (\zeta,x)\in \R\times [\cos\ath,1),\\
& \hat{f}(\zeta,\cos\ath) = \hat{\phi}(\zeta), \qquad \d_\theta\hat{f}(\zeta,1)=0 && \text{for } \ \zeta\in\R.  
\end{align}
\end{subequations}
Solutions to the ODE (\ref{Legendre}) is a linear combination of the Legendre's functions of first and second kinds with eigenvalue $\mu (\mu+1)$:
\begin{equation*}
\hat{f}(\zeta,x) = c_1(\zeta) P_{\mu}(x) + c_2(\zeta) Q_{\mu}(x), \quad \text{where } \ \ \mu=-\dfrac{1}{2} + i \zeta.
\end{equation*}
Note that the functions $P_{\mu}(x)$, $Q_{\mu}(x)$ with $\mu=-\frac{1}{2}+i\zeta$ and $\zeta\in\R$ are called conical functions (See Section 8.840 of \cite{GR2007}). Rewriting as $\hat{\Phi}_{\ast}(\zeta,\theta) = \hat{f}(\zeta,\cos\theta) $, it follows that for $(\zeta,\theta)\in \R \times(0,\ath]$, 
\begin{gather*}
\hat{\Phi}_{\ast}(\zeta,\theta) = c_1(\zeta) P_{\mu}(\cos\theta)+c_2(\zeta) Q_{\mu}(\cos\theta), \qquad \text{where } \ \mu = -\dfrac{1}{2} + i \zeta,\\
\hat{\Phi}_{\ast}(\zeta,\ath) = \hat{\phi}(\zeta) \qquad \text{and} \qquad \d_\theta \hat{\Phi}_{\ast}(\zeta,0)=0.
\end{gather*}
To determine the coefficient functions $c_1(\zeta)$ and $c_2(\zeta)$, we use the identities below, which can be found from Sections 8.706 and 8.752 of \cite{GR2007}:
\begin{align}
	&P_{\mu}^m (x) = (-1)^m(1-x^2)^{\frac{m}{2}} \dfrac{\dif^m P_{\mu}}{\dif x^m}(x),\qquad Q_{\mu}^m (x) = (-1)^m(1-x^2)^{\frac{m}{2}} \dfrac{\dif^m Q_{\mu}}{\dif x^m}(x),\label{dPQ}
\end{align}
for $\mu\in\mathbb{C}$, $x\in(-1,1)$, and $m\in\mathbb{N}$. Thus, $Q_{\mu}^1(\cos\theta) = \d_\theta \big\{Q_{\mu}(\cos\theta)\big\}$ for $\theta\in(0,\theta_{\ast}]$. Using the fact that $|Q_{\mu}^1(x)|\to \infty$ as $x\to 1^-$, we must have $c_2(\zeta)=0$. To determine $c_1(\zeta)$, we first denote for simplicity:
\begin{equation}\label{kDef}
	\fk(\zeta,\theta)\vcentcolon= P_{\mu}(\cos\theta) = P_{-\frac{1}{2}+i\zeta}(\cos\theta).
\end{equation}
Then, by the condition $\hat{\Phi}_{\ast}(\zeta,\ath)=\hat{\phi}(\zeta)$, we obtain that $c_1(\zeta) = \tfrac{\hat{\phi}(\zeta)}{\fk(\zeta,\ath)}$. Hence,
\begin{equation}\label{Uphat}
\hat{\Phi}_{\ast}(\zeta,\theta) = \hat{\phi}(\zeta) \hat{K}(\zeta,\theta), \quad \text{where } \ \hat{K}(\zeta,\theta)\vcentcolon= \dfrac{\fk(\zeta,\theta)}{\fk(\zeta,\ath)}, \ \text{ for } \ (\zeta,\theta)\in\R\times(0,\ath].
\end{equation}
Taking the inverse Fourier transform on both sides of (\ref{Uphat}), we obtain
\begin{equation}\label{UpConv}
\Phi_{\ast}(\sigma,\theta) = \int_{\R}\!\! \phi(s) K(\sigma-s,\theta)\, \dif s, \quad \text{where } \ K(\sigma,\theta)\vcentcolon= \dfrac{1}{2\pi}\int_{\R} \dfrac{\fk(\zeta,\theta)}{\fk(\zeta,\ath)} e^{i\sigma \zeta}\, \dif \zeta. 
\end{equation}

\begin{remark}
	The kernel $\fk(\zeta,\theta)=P_{\mu}(x)$ can be written in terms of the hypergeometric functions as follows:
	\begin{equation*}
		P_{\mu}(\cos\theta) = {}_2F_1\Big( \frac{1-2i\zeta}{2},\frac{1+2i\zeta}{2}; 1; \frac{1-\cos\theta}{2} \Big).
	\end{equation*}
	If $|z|<1$ and $c$ is not a negative integer, then ${}_2F_1(a,b;c;z)$ has the convergent series:
	\begin{equation*}
		{}_2F_1(a,b;c;z)=\sum_{n=0}^{\infty} \dfrac{(a)_n(b)_n}{(c)_n}\dfrac{z^n}{n!}, \qquad \text{where } \ (x)_n \equiv \dfrac{\Gamma(x+n)}{\Gamma(x)}, 
	\end{equation*}
	and $\Gamma(x)$ denotes the gamma function. In our case, we have that
	\begin{align*}
		&P_{\mu}(\cos\theta) = \sum_{n=0}^{\infty} \Big(\dfrac{1}{2}-i\zeta\Big)_n \Big(\dfrac{1}{2}+i\zeta\Big)_n \dfrac{(1-\cos\theta)^n}{2^n n!(n+1)!}\\
		=& \sum_{n=0}^{\infty} \dfrac{\Gamma(n+\tfrac{1}{2}-i\zeta)\Gamma(n+\tfrac{1}{2}+i\zeta)}{\Gamma(\tfrac{1}{2}-i\zeta)\Gamma(\tfrac{1}{2}+i\zeta)} \dfrac{(1-\cos\theta)^n}{2^n n!(n+1)!} = \sum_{n=0}^{\infty} \dfrac{\big|\Gamma(n+\tfrac{1}{2}+i\zeta)\big|^2}{\big|\Gamma(\tfrac{1}{2}+i\zeta)\big|^2} \dfrac{(1-\cos\theta)^n}{2^n n!(n+1)!}.
	\end{align*}
	By the identities for gamma function, the above series can be rewritten as  
	\begin{align*}
		\fk(\zeta,\theta)\vcentcolon= P_{\mu}(\cos\theta) = \sum_{n=0}^{\infty} \dfrac{(1-\cos\theta)^n}{2^n n!(n+1)!}  \prod_{m=1}^{n}\Big\{ \big(m-\dfrac{1}{2}\big)^2 + \zeta^2 \Big\}.
	\end{align*}
\end{remark}

\subsection{Bounds for conical Legendre function}\label{ssec:kbounds}
Next, we aim to obtain Sobolev estimates for $\Phi_{\ast}$. To achieve this, we wish to derive bounds for the kernel $\hat{K}(\zeta,\theta)$. The following list of properties for $\fk(\zeta,\theta)$ can be found in Section 8.84 of \cite{GR2007}:
\begin{gather}
\fk(-\zeta,\theta) = P_{-\frac{1}{2}-i\zeta}(\cos\theta) = P_{-\frac{1}{2}+i\zeta}(\cos\theta) = \fk(\zeta,\theta),\label{kEven}\\
\fk(\zeta,\theta) = \dfrac{2}{\pi}\int_{0}^{\theta}\!\!\! \dfrac{\cosh(\zeta \varphi)}{\sqrt{2(\cos \varphi - \cos\theta)}} \,\dif \varphi = \dfrac{2\cosh(\pi\zeta)}{\pi}\!\! \int_{0}^{\infty}\!\!\!\!\! \dfrac{\cos(\zeta \varphi)}{\sqrt{2(\cos\theta + \cosh \varphi)}} \, \dif \varphi, \label{kInt}
\end{gather}
for $(\zeta,\theta)\in\R\times(0,\pi)$. In addition, as $|\zeta|\to \infty$, the leading term for the asymptotic expansion of $\fk(\zeta,\theta)$  is given by: (See Section 4.10 of \cite{MOS1966})
\begin{equation}\label{kAsy}
\fk(\zeta,\theta) \sim (2\pi \sin\theta)^{-\frac{1}{2}} |\zeta|^{-1}   \qquad \text{as } \  |\zeta| \to \infty.
\end{equation}
From (\ref{kInt})--(\ref{kAsy}), one immediately observes the following properties for $\fk$:
\begin{proposition}\label{prop:kbound}
Let $\theta_{\ast}\in (0,\pi)$ and $\fk(\zeta,\theta)$ be given in (\ref{kInt}). Then there exists some constant $C=C(\ath)>0$ such that for $\theta\in (0,\theta_{\ast}]$,
\begin{equation*}
C^{-1}\theta \le \fk(\zeta,\theta) \le C \cosh(\pi\zeta) \ \text{ for } \ \zeta\in\R \ \text{ and } \ \lim\limits_{|\zeta|\to\infty}\Big| \dfrac{\fk(\zeta,\theta)}{\cosh(\pi\zeta)} \Big| = 0.
\end{equation*}
Moreover, for $m\in\mathbb{N}$, there exists a constant $C(m)=C(\ath,m)>0$ such that 
\begin{equation*}
|\d_\theta^m \fk(\zeta,\theta)| \le C(m) \cosh(\pi\zeta) \quad \text{for } \ (\zeta,\theta)\in\R\times(0,\ath].
\end{equation*}
\end{proposition}
\begin{proof}
By (\ref{kInt}), there exists $C=C(\ath)>0$ such that for all $\theta\in(0,\ath]$
\begin{align*}
	|\fk(\zeta,\theta)| \le \dfrac{2\cosh(\pi\zeta)}{\pi} \int_{0}^{\infty}\!\!\! \dfrac{\dif \varphi}{\sqrt{2(\cos\theta+\cosh \varphi)}} \le C\cosh(\pi\zeta) \int_{0}^{\infty}\!\!\! e^{-\frac{\varphi}{2}}\, \dif \varphi\le C \cosh(\pi\zeta).
\end{align*}
Moreover, since $\sqrt{\cos\varphi-\cos\theta}\le 2$
for $0 \le \varphi \le \theta \le \pi$, it follows from (\ref{kInt}) that
\begin{align*}
\fk(\zeta,\theta) = \dfrac{2}{\pi} \int_{0}^{\theta} \dfrac{\cosh(\zeta \varphi)}{\sqrt{2(\cos\varphi-\cos\theta)}}\, \dif \varphi \ge \dfrac{\theta}{\pi\sqrt{2}}. 
\end{align*}
Next, for $\theta\in(0,\ath]$, there exists a constant $C=C(\ath)>0$ such that for all $\varphi\in(0,\infty)$, $2(\cos\theta+\cosh\varphi) \ge C^{-1} e^{\varphi}$. Using this and the integral representation (\ref{kInt}), we get
\begin{align*}
|\d_\theta^m \fk(\zeta,\theta)| 
\le C(m)\cosh(\pi\zeta) \sum_{k=1}^m \int_{0}^{\infty}\!\! \exp\big(-\tfrac{1+2k}{2}\varphi\big)\, \dif \varphi = C(m)\cosh(\pi\zeta), 
\end{align*}
for some constant $C(m)=C(\ath,m)>0$. Finally, by Riemann-Lebesgue lemma,
\begin{equation*}
\lim\limits_{|\zeta|\to\infty} \bigg|\int_{0}^{\infty}\!\! \dfrac{\cos(\zeta s)}{\sqrt{2(\cos\theta + \cosh s)}}\, \dif s\bigg| = 0 \qquad \text{for each } \ \theta\in (0,\theta_{\ast}].
\end{equation*}
This concludes the proof.
\end{proof}

The following integral bounds for the kernel $\hat{K}(\zeta,\theta)=\fk(\zeta,\theta)/\fk(\zeta,\theta_{\ast})$ will be useful for the $H^1$-estimates of $\Phi_{\ast}$.
\begin{lemma}\label{lemma:dkRatio}
Set $\absm{\zeta}\vcentcolon= \sqrt{1+|\zeta|^2}$, and $\theta_{\ast}\in (0,\pi)$. Let $\fk(\zeta,\theta)$ be given in (\ref{kInt}). Then there exist finite constant $C=C(\ath)>0$ such that for all $\zeta\in\R$,
\begin{equation}\label{kIneq}
	\absm{\zeta}\int_{0}^{\ath}\Big| \dfrac{\fk(\zeta,\theta)}{\fk(\zeta,\ath)} \Big|^2 \, \dif \theta \le C, \qquad \dfrac{1}{\absm{\zeta}}\int_{0}^{\ath}\Big| \dfrac{\d_\theta \fk(\zeta,\theta)}{\fk(\zeta,\ath)} \Big|^2 \, \dif \theta \le C.
\end{equation}
\end{lemma}
The key idea for the proof of Lemma \ref{lemma:dkRatio} is to use the asymptotic expansion for conical function stated below, which can be found in Page 473 of \cite{Olver1997}: for fixed $m\in\mathbb{N}\cup\{0\}$ and $0<\delta < \pi$, one has as $\zeta\to \infty$, 
\begin{equation}
	P^{-m}_{i\zeta-\frac{1}{2}}(\cos\theta) = \dfrac{I_{m}(\zeta \theta)}{\zeta^m\sqrt{\sinc \theta}}  \big\{ 1 + \mathcal{O}\big(\zeta^{-1}\big) \big\} \quad \text{uniformly in } \ \theta\in(0,\pi -\delta ),\label{coneAsy}
\end{equation}
where $\sinc\theta\vcentcolon= \frac{\sin\theta}{\theta}$, and $I_m(\cdot)$ is the modified Bessel's function of first kind. See also Page 22 of \cite{ZK66}. By obtaining similar integral bounds for the ratio $I_{0}(\zeta\theta)/I_{0}(\zeta\ath)$, one can translate the corresponding inequality into (\ref{kIneq}) via Proposition \ref{prop:kbound} and (\ref{coneAsy}). Note that $I_{m}(z)$ has the integral representation:
\begin{equation}\label{besselInt}
	I_{m}(z) = \dfrac{1}{\pi} \int_{0}^{\pi} \!\! e^{z\cos\varphi} \cos(m\varphi ) \, \dif \varphi \quad \text{for } \ m\in \mathbb{N}\cup\{0\} \ \text{ and } \ z\in\mathbb{C}.
\end{equation}
Using this, we prove the following:
\begin{proposition}\label{prop:i1i0}
	Let $I_0(x)$ be the modified Bessel's function of first kind with eigenvalue $0$, and denote $I_0^{(k)}(x) \vcentcolon= \frac{\dif^k I_0}{\dif x^k} (x)$. Then for all $x\in\R$ and $k\in\mathbb{N}\cup\{0\}$, 
	\begin{equation}\label{i1i0}
		\int_{0}^{1}\!\! \Big| \dfrac{I_0^{(k)}(y x)}{I_0(x)} \Big|^2 \, \dif y  \le 1  \qquad \text{ and } \qquad |x|\int_{0}^{1} \Big| \dfrac{I_0^{(k)}(y x)}{I_0(x)} \Big|^2 \, \dif y \le 3.
	\end{equation}
\end{proposition}
\begin{proof}
	First, we claim that for any $k\in\mathbb{N}\cup\{ 0 \}$, $y\in(0,1)$, and $x>0$,
	\begin{equation}\label{I1I0}
		\Big| \dfrac{I_0^{(k)}( y x)}{I_0(x)} \Big|^2 \le \pi^{2-2 y} \Big( \int_{0}^{\pi} e^{x\cos\theta}\,\dif \theta \Big)^{-2(1-y)} = |I_0(x)|^{-2(1-y)}.
	\end{equation}
	To show this, we observe that for any $k\in\mathbb{N}\cup\{0\}$ and $0< y \le 1$, 
	\begin{align}\label{I1I0'}
		\dfrac{I_0^{(k)}( y x)}{I_0(x)} = \dfrac{\smallint_{0}^{\pi} e^{ y x \cos\theta} \cos^k\! \theta \, \dif \theta}{\smallint_{0}^{\pi} e^{x\cos\theta}\,\dif \theta} = \Big( \int_{0}^{\pi} e^{x\cos\theta}\,\dif \theta \Big)^{-1+ y} \int_{0}^{\pi} |q(\theta,x)|^{y} \cos^k \!\theta\, \dif \theta,
	\end{align}
	where $q(\theta,x)$ is such that
	\begin{equation*}
		q(\theta,x) \vcentcolon= \Big( \int_{0}^{\pi} e^{x\cos\theta}\,\dif \theta \Big)^{-1} e^{x\cos\theta} \ \text{ which implies that } \ \int_{0}^{\pi} q(x,\theta) \, \dif \theta = 1.
	\end{equation*}
	Exponentiating both sides of (\ref{I1I0'}) by $\frac{1}{y}$, then applying Jensen's inequality, we have
	\begin{gather*}
		\Big| \dfrac{I_0^{(k)}( y x)}{I_0(x)} \Big|^{\frac{1}{y}} 
		\le \Big( \int_{0}^{\pi} e^{x\cos\theta}\,\dif \theta \Big)^{\frac{y-1}{y}}  \pi^{\frac{1}{y}-1} \int_{0}^{\pi} q(x,\theta) |\cos\theta|^{\frac{k}{y}}\, \dif \theta \le \pi^{\frac{1}{y}-1} \Big( \int_{0}^{\pi} e^{x\cos\theta}\,\dif \theta \Big)^{\frac{y-1}{y}}.
	\end{gather*}
	Exponentiating both sides of the above inequalities with $2y$, one obtains (\ref{I1I0}).
	
	Next, we wish to complete the proof using (\ref{I1I0}). Since $I_0^{(k)}(-x)=I_0^{(k)}(x)$ for even $k\in\mathbb{N}\cup\{0\}$ and $I_0^{(k)}(-x)=-I_0^{(k)}(x)$ for odd  $k\in\mathbb{N}$. This implies $x\mapsto |I_0^{(k)}(x)|^2$ is an even function for all $k\in\mathbb{N}\cup\{0\}$. Thus it suffices to prove the case only for $x\ge 0$. By the integral representation (\ref{besselInt}), we have for all $k\in\mathbb{N}\cup\{0\}$, $y\in[0,1]$, and $x\ge 0$,
	\begin{equation*}
		\big| I_0^{(k)} (yx)\big| = \bigg|\dfrac{1}{\pi}\int_{0}^{\pi}e^{yx\cos\theta}\cos^k\theta\, \dif \theta\bigg| \le \bigg|\dfrac{1}{\pi}\int_{0}^{\pi}e^{yx\cos\theta}\, \dif \theta\bigg| = \big| I_0(yx) \big| \le \big| I_0(x) \big|,
	\end{equation*}
	where in the last inequality, we used the fact that $x\mapsto I_0(x)$ is monotone increasing in the domain $x\ge 0$. This proves the first inequality in (\ref{i1i0}). For the second inequality in (\ref{i1i0}), we use th integral formula (\ref{besselInt}) to get
	\begin{align}\label{Ipi3}
		|I_0(x)| = \dfrac{1}{\pi} \int_{0}^{\pi}\!\! e^{x\cos\theta}\, \dif \theta  \ge  \dfrac{1}{\pi} \int_{0}^{\pi/3}\!\! e^{x\cos\theta}\, \dif \theta \ge \dfrac{1}{3} e^{x/2}\quad \text{for } \ x\ge 0.
	\end{align}
	Combining the above bound with (\ref{I1I0}), one has for all $k\in\mathbb{N}\cup \{0\}$ and $x\ge 0$,
	\begin{align*}
		x\!\!\int_{0}^{1}\! \Big| \dfrac{I_0^{(k)}(yx)}{I_0(x)} \Big|^2 \dif y \le x\!\! \int_{0}^{1}\!\! \big|I_0(x)\big|^{-2(1-y)} \dif y 
		\le 9 x e^{-x}\!\! \int_0^1\!\!\! \big(9^{-1} e^x \big)^y \dif y = \dfrac{x(1-9e^{-x})}{x- \ln 9} \le 3. 
	\end{align*}
	This proves the second inequality of (\ref{i1i0}), which concludes the proof.
\end{proof}

Combining the asymptotic expansion (\ref{coneAsy}) and Propositions \ref{prop:kbound}, \ref{prop:i1i0}, we are ready to show Lemma \ref{lemma:dkRatio}. 
\begin{proof}[Proof of Lemma \ref{lemma:dkRatio}]
By the asymptotic expansion (\ref{coneAsy}), there exist finite constants $C_0>0$, $M_0>0$ independent of $(\zeta,\theta)$ such that for $|\zeta|\ge M_0$,
\begin{equation}\label{asym0}
	C_0^{-1} \dfrac{I_0(\zeta\theta)}{\sqrt{\sinc\theta}}  \le | \fk(\zeta,\theta) | \le C_0 \dfrac{I_0(\zeta\theta)}{\sqrt{\sinc\theta}} .
\end{equation}
Since $1 \le 1/\sinc\theta \le 1/\sinc\ath $ for $\theta\in[0,\ath]$, it follows that for $|\zeta|\ge M_0$,
\begin{align*}
\int_{0}^{\ath} \Big| \dfrac{\fk(\zeta,\theta)}{\fk(\zeta,\ath)} \Big|^2\, \dif \theta \le C_0^2 \int_{0}^{\ath} \Big| \dfrac{I_0(\zeta\theta)}{I_0(\zeta\ath)} \Big|^2 \, \dif \theta  = C_0^2\ath \int_{0}^{1} \Big| \dfrac{I_0(y x )}{I_0(x)} \Big|^2 \, \dif y,
\end{align*}
where in the last equality, we applied the change of variables $x\vcentcolon= \zeta \ath$ and $y\vcentcolon=\theta/\ath$. Applying Proposition \ref{prop:i1i0}, we obtain the first inequality of (\ref{kIneq}) for $|\zeta|\ge M_0$. 

Next, by definition (\ref{kDef}) and the identity (\ref{dPQ}), we have for $(\zeta,\theta)\in(0,\ath)\times\R$,
\begin{equation*}
\d_\theta \fk(\zeta,\theta) = -\sin\theta \dfrac{\dif P_{-\frac{1}{2}+i\zeta}}{\dif x} (x) \Big\vert_{x=\cos\theta} = P_{-\frac{1}{2}+i\zeta}^{1}(\cos\theta).
\end{equation*} 
From Section 8.752 of \cite{GR2007}, the following holds for $\mu\in\mathbb{C}$, $x\in(-1,1)$, and $m\in\mathbb{N}$,
\begin{equation}
	P_{\mu}^{-m}(x) = (-1)^m\dfrac{\Gamma(\mu-m+1)}{\Gamma(\mu+m+1)}P_{\mu}^m(x),\label{PG}
\end{equation}
where $\Gamma$ denotes the Gamma function, and it satisfies the following identities (See Sections 8.332 of \cite{GR2007}),
\begin{equation}
	\Big|\Gamma\big(\dfrac{1}{2}+i\zeta\big)\Big|^2 = \dfrac{\pi}{\cosh(\pi\zeta)}, \qquad \overline{\Gamma(z)} = \Gamma(\bar{z}), \qquad \Gamma(z)\Gamma(1-z) = \dfrac{\pi}{\sin(\pi z)}, \label{GammaF} 
\end{equation}
for $\zeta\in\R$ and $z\in\mathbb{C}$. By induction argument, it can be shown that
\begin{equation}
	\Big| \Gamma\big(\dfrac{1}{2} \pm m + i\zeta \big) \Big|^2 = \dfrac{\pi}{\cosh(\pi \zeta)} \prod_{k=1}^{m}\Big((k-\tfrac{1}{2})^2 + \zeta^2 \Big)^{\pm 1},\label{GaHalf}
\end{equation}
for $m\in\mathbb{N}$ and $\zeta\in\R$. Using (\ref{GaHalf}) and (\ref{PG}), it follows that
\begin{align*}
\big|\d_\theta \fk(\zeta,\theta)\big| = \Big|P_{-\frac{1}{2}+i\zeta}^{1}(\cos\theta)\Big| = \Big|\dfrac{\Gamma(\frac{3}{2}+i\zeta)}{\Gamma(-\frac{1}{2}+i\zeta)} P_{-\frac{1}{2}+i\zeta}^{-1}(\cos\theta)\Big| = \big(\frac{1}{4}+\zeta^2\big) \big|P_{-\frac{1}{2}+i\zeta}^{-1}(\cos\theta)\big|. 
\end{align*}
By the asymptotic expansion (\ref{coneAsy}), there exist finite constants $C_1>0$, $M_1>0$ independent of $(\zeta,\theta)$ such that for $|\zeta|\ge M_1$,
\begin{equation}\label{asym1}
C_1^{-1}\dfrac{1+4\zeta^2}{\zeta} \dfrac{I_1(\zeta\theta) }{\sqrt{\sinc\theta}}  \le \big|\d_\theta \fk(\zeta,\theta) \big| \le C_1\dfrac{(1+4\zeta^2)}{\zeta} \dfrac{I_1(\zeta\theta)}{\sqrt{\sinc\theta}}. 
\end{equation}
Combining (\ref{asym0}) and (\ref{asym1}), we get for $|\zeta|\ge M\equiv \max\{M_0,M_1\}$ and $\theta\in(0,\ath]$,
\begin{align*}
	\Big|\dfrac{\d_\theta \fk(\zeta,\theta)}{\fk(\zeta,\ath)}\Big|^2 \le C_0 C_1 \dfrac{\sinc \ath}{\sinc\theta} \dfrac{(1+4\zeta^2)^2}{\zeta^2} \Big|\dfrac{I_1(\zeta\theta)}{I_0(\zeta\ath)}\Big|^2 \le C \absm{\zeta}^2 \Big|\dfrac{I_1(\zeta\theta)}{I_0(\zeta\ath)}\Big|^2.
\end{align*}  
Setting the change of variables $x=\zeta\ath$ and $y=\theta/\ath$ then 
\begin{align*}
\dfrac{1}{\absm{\zeta}}\int_{0}^{\ath}\!\!\Big|\dfrac{\d_\theta \fk(\zeta,\theta)}{\fk(\zeta,\ath)}\Big|^2 \,\dif \theta \le & C \absm{\zeta} \int_{0}^{\ath} \! \Big|\dfrac{I_1(\zeta\theta)}{I_0(\zeta\ath)}\Big|^2 \, \dif \theta \le C \absm{x} \int_{0}^{1} \!  \Big|\dfrac{I_1(y x)}{I_0(x)}\Big|^2 y \, \dif y.
\end{align*}
Applying Proposition \ref{prop:i1i0}, we obtain the second inequality of (\ref{kIneq}) for $|\zeta| \ge M$.

Finally, If $|\zeta|\le M$ then one can employ Proposition \ref{prop:kbound} to get
\begin{align*}
\absm{\zeta}\!\! \int_{0}^{\ath}\! \Big| \dfrac{\fk(\zeta,\theta)}{\fk(\zeta,\ath)} \Big|^2\dif \theta \le \dfrac{C\absm{M}\cosh^2(\pi M)}{\ath}, \quad \dfrac{1}{\absm{\zeta}}\!\! \int_{0}^{\ath}\! \Big| \dfrac{\d_\theta\fk(\zeta,\theta)}{\fk(\zeta,\ath)} \Big|^2\dif \theta  \le  \dfrac{ C\cosh^2(\pi M)}{\ath}.
\end{align*}
Since $M=\max\{M_0,M_1\}>0$ is independent of $(\zeta,\theta)$, this concludes the proof. 
\end{proof}

The next task is to generalise Lemma \ref{lemma:dkRatio} for the derivative $\d_{\theta}^m$ with $m\ge 2$. This will be used to prove the $H^{s-m+1/2}(\R)$-estimate for $\d_{\theta}^m\Phi_{\ast}$.
\begin{proposition}\label{prop:dkRatio-m}
	Let $\theta_{\ast}\in (0,\pi)$ and $\fk(\zeta,\theta)$ be the function given in (\ref{kInt}). Then for each $m\ge 2$, there exist finite constants $C=C(m,\ath)>0$ such that for $\zeta\in\R$,
	\begin{equation}\label{kIneq-m}
		\dfrac{1}{\absm{\zeta}^{2m-1}}\int_{0}^{\ath}\Big| \dfrac{\d_\theta^m \fk(\zeta,\theta)}{\fk(\zeta,\ath)} \Big|^2 \theta^{2 m} \, \dif \theta \le C \ \text{ where } \ \absm{\zeta}\vcentcolon= \sqrt{1+\zeta^2}.
	\end{equation}
\end{proposition}
\begin{proof}
	Since $\d_\theta^k \cos\theta = \cos(\theta+k\frac{\pi}{2})$, one has from (\ref{kDef}) and Fa\`a di Bruno's formula
	\begin{gather*}
		\d_\theta^m \fk(\zeta,\theta) = \sum_{k=1}^{m} c_{m,k}(\theta) \dfrac{\dif^k P_{\mu}}{\dif x^k}(\cos\theta) \qquad \text{where}\\
		\mu=-\frac{1}{2}+i\zeta \ \ \text{ and } \ \ c_{m,k}(\theta)\vcentcolon= B_{m,k}\big( \cos(\theta+\tfrac{\pi}{2}), \dotsc, \cos(\theta+(m-k+1)\tfrac{\pi}{2}) \big).
	\end{gather*} 
	Here $B_{m,k}$ denotes the $(m,k)$-th Bell's polynomial given by
	\begin{equation*}
		B_{m,k}(x_1,\dotsc,x_{m-k+1}) \vcentcolon= \sum \dfrac{m!}{j_1!\cdots j_{m-k+1}!} \Big(\dfrac{x_1}{1!}\Big)^{j_1} \!\! \cdots  \Big(\dfrac{x_{m-k+1}}{(m-k+1)!}\Big)^{j_{m-k+1}},
	\end{equation*}
	where the sum is taken over all sequences $j_1,j_2,\dotsc,j_{m-k+1}$ of non-negative integers satisfying the constraints:
	\begin{gather*}
		j_1+j_2 + \cdots + j_{m-k+1} = k, \quad \text{ and } \quad j_1 + 2 j_2 + \cdots + (m-k+1) j_{m-k+1} = m.
	\end{gather*}
	It follows that there exists some constant $C=C(m)>0$ independent of $\theta\in(0,\ath]$ such that $c_{m,k}(\theta)\le C$ for all $k=1,\dotsc,m$. Next, by the identities for Legendre's function (\ref{dPQ}) and (\ref{PG}), we also have
	\begin{align}
		\d_\theta^m \fk(\zeta,\theta) = \sum_{k=1}^m (-1)^{k}  c_{m,k}(\theta) \dfrac{P_{\mu}^k(\cos\theta)}{(\sin\theta)^k} 
		=& \sum_{k=1}^m c_{m,k}(\theta) \dfrac{\Gamma(\tfrac{1}{2}+k+i\zeta)}{\Gamma(\tfrac{1}{2}-k+i\zeta)} \dfrac{P_{\mu}^{-k}(\cos\theta)}{(\sin\theta)^k}, \label{dmfk}
	\end{align}
	where by (\ref{GaHalf}), we have for each $k=1,\dotsc,m$
	\begin{equation}
		\Big|\dfrac{\Gamma(\tfrac{1}{2}+k+i\zeta)}{\Gamma(\tfrac{1}{2}-k+i\zeta)} \Big| = \prod_{\alpha=1}^k \Big((\alpha-\tfrac{1}{2})^2 + \zeta^2\Big).\label{kGamma}
	\end{equation}
	Next, let $\{T_m(z)\}_{m=\mathbb{N}\cup\{0\}}$ be Chebyshev's polynomial of the first kind, which is defined indirectly by the angle sum formula as:
	\begin{equation*}
		T_m(\cos\varphi) = \cos(m\varphi) \quad \text{or} \quad T_m(\cosh\varphi) = \cosh(m\varphi).
	\end{equation*}
	Then (\ref{besselInt}) implies that for $m\in\mathbb{N}\cup\{0\}$,
	\begin{equation*}
		I_m(z) = T_{m}\big(\dfrac{\dif}{\dif z}\big) I_0(z) = \sum_{k=0}^{m} \dfrac{1}{k!} \dfrac{\dif^k T_m}{\dif z^k}\Big\vert_{z=0} \dfrac{\dif^k I_0}{\dif z^k}(z). 
	\end{equation*}
	Thus for each $m\in\mathbb{N}\cup\{0\}$, there exists a constant $C(m)>0$ such that
	\begin{equation}\label{Ider}
		|I_{m}(z)| \le C(m) \sum_{k=0}^m |I_0^{(k)}(z)|, \quad \text{where} \quad I_0^{(k)}(z) \vcentcolon= \dfrac{\dif^k I_0}{\dif x^k} (z).
	\end{equation}
	Using the asymptotic formula (\ref{coneAsy}) and the inequality (\ref{Ider}), there exists some $M>0$ independent of $\ath$, and $C=C(m,\ath)>0$ such that for $|\zeta|\ge M$ and $\theta\in(0,\ath]$,
	\begin{equation}
		\big|P_{\mu}^{-k}(\cos\theta)\big| \le C\dfrac{|I_k(\zeta\theta)|}{\zeta^k\sqrt{\sinc\theta}}  \le \dfrac{C}{\zeta^k\sqrt{\sinc\theta}} \sum_{\alpha=0}^k \big|I_0^{(\alpha)}(\zeta\theta)\big|,\label{kBessel}
	\end{equation}
	where $I_0^{(\alpha)}(x)\vcentcolon=\tfrac{\dif^\alpha I_0}{\dif x^{\alpha}}(x)$ denotes the $\alpha$-th derivative of $I_0(x)$. Substituting inequalities (\ref{kGamma})--(\ref{kBessel}) into (\ref{dmfk}), we obtain that for $|\zeta|\ge M$ and $\theta\in(0,\ath]$,
	\begin{align*}
		|\d_\theta^m \fk(\zeta,\theta)| \le  \dfrac{C}{\sqrt{\sinc\theta}} \sum_{k=0}^m \absm{\zeta}^{k} \dfrac{|I_{0}^{(k)}(\zeta\theta)|}{(\sin\theta)^k}. 
	\end{align*}
	Repeating the same argument in the proof of Lemma \ref{lemma:dkRatio}, one has
	\begin{align*}
		& \dfrac{1}{\absm{\zeta}^{2m-1}}\int_0^{\ath} \Big|\dfrac{\d_\theta^m \fk (\zeta,\theta)}{\fk(\zeta,\ath)}\Big|^2 (\sin\theta)^{2m}\,\dif\theta  \le \left\{ \begin{aligned}
			&C && \text{if } \ |\zeta| < M,\\
			& C \absm{\zeta} \int_{0}^{1} \sum_{k=0}^m\Big| \dfrac{I_0^{(k)}(yx)}{I_0(x)} \Big|^2\, \dif y && \text{if } \ |\zeta| \ge M.
		\end{aligned}  \right.
	\end{align*}
	Applying Proposition \ref{prop:i1i0}, and using that $0< \sinc\theta \le 1$ for $\theta\in (0,\ath]$, we obtain the desired estimate.
\end{proof}

\subsection{Sobolev Estimates for \texorpdfstring{$\Phi_{\ast}$}{Phistar}}
\begin{lemma}\label{lemma:bdryEst}
Let $\Phi_{\ast}(\sigma,\theta)$ be given by (\ref{UpConv}), and $\ath\in(0,\pi)$. Then there exists a constant $C=C(\ath)>0$ such that the following estimates hold:
\begin{align*}
\int_{0}^{\ath}\!\!\Big\{ \| \Phi_{\ast}(\cdot,\theta) \|_{H^{s+\frac{1}{2}}(\R)}^2 + \| \d_\theta \Phi_{\ast}(\cdot,\theta)\|_{H^{s-\frac{1}{2}}(\R)}^2 \Big\} \, \dif \theta \le C \|\phi\|_{H^{s}(\R)}^2.
\end{align*}
\end{lemma}
\begin{proof}
By the construction (\ref{Uphat}) and Lemma \ref{lemma:dkRatio}, we have
\begin{align*}
&\int_{0}^{\ath}\!\! \| \Phi_{\ast}(\cdot,\theta) \|_{H^{s+\frac{1}{2}}(\R)}^2 \, \dif \theta 
= \int_{0}^{\ath}\!\!\!\! \int_{\R} \absm{\zeta}^{2s+1} |\hat{\phi}(\zeta)|^2 \Big| \dfrac{\fk(\zeta,\theta)}{\fk(\zeta,\ath)} \Big|^2\, \dif \zeta \dif \theta\\
=& \int_{\R}\!\! \absm{\zeta}^{2s} |\hat{\phi}(\zeta)|^2 \bigg( \absm{\zeta}\!\!\int_{0}^{\ath}\! \Big| \dfrac{\fk(\zeta,\theta)}{\fk(\zeta,\ath)} \Big|^2 \, \dif \theta \bigg)\, \dif \zeta \le C \int_{\R}\!\! \absm{\zeta}^{2s} |\hat{\phi}(\zeta)|^2\,\dif \zeta = C \|\phi\|_{H^{s}(\R)}^2.
\end{align*}
Similarly, we also have 
\begin{align*}
	&\int_{0}^{\ath}\!\! \| \d_\theta\Phi_{\ast}(\cdot,\theta) \|_{H^{s-\frac{1}{2}}(\R)}^2 \, \dif \theta 
	= \int_{0}^{\ath}\!\!\!\! \int_{\R} \absm{\zeta}^{2s-1} |\hat{\phi}(\zeta)|^2 \Big| \dfrac{\d_\theta\fk(\zeta,\theta)}{\fk(\zeta,\ath)} \Big|^2 \dif \zeta \dif \theta\\
	&= \int_{\R}\!\! \absm{\zeta}^{2s} |\hat{\phi}(\zeta)|^2 \bigg( \dfrac{1}{\absm{\zeta}} \int_{0}^{\ath}\! \Big| \dfrac{\d_\theta\fk(\zeta,\theta)}{\fk(\zeta,\ath)} \Big|^2 \, \dif \theta \bigg)\, \dif \zeta \le C \int_{\R}\!\! \absm{\zeta}^{2s} |\hat{\phi}(\zeta)|^2\,\dif \zeta = C \|\phi\|_{H^{s}(\R)}^2.
\end{align*}
This completes the proof.
\end{proof}

\begin{lemma}\label{lemma:higherPhi}
	Let $\Phi_{\ast}(\sigma,\theta)$ be given by (\ref{UpConv}), $m\in\mathbb{N}$, and $\ath\in(0,\pi)$. Then for each $m\ge 2$, there exists a constant $C=C(m,\ath)>0$ such that:
	\begin{align*}
		\int_{0}^{\ath}\!\! \| \d_\theta^m \Phi_{\ast}(\cdot,\theta)\|_{H^{s+\frac{1}{2}-m}(\R)}^2 \theta^{2m} \, \dif \theta \le C \|\phi\|_{H^{s}(\R)}^2.
	\end{align*}
\end{lemma}
\begin{proof}
	By the construction (\ref{Uphat}) and Proposition \ref{prop:dkRatio-m}, we have
	\begin{align*}
		&\int_{0}^{\ath}\!\!\! \| \d_\theta^m\Phi_{\ast}(\cdot,\theta) \|_{H^{s+\frac{1}{2}-m}}^2 \theta^{2m}\, \dif \theta 
		=\!\! \int_{0}^{\ath}\!\!\!\! \int_{\R} \absm{\zeta}^{2s-2m+1} |\hat{\phi}(\zeta)|^2 \Big| \dfrac{\d_\theta^m\fk(\zeta,\theta)}{\fk(\zeta,\ath)} \Big|^2 \theta^{2m}\, \dif \zeta \dif \theta\\
		=& \int_{\R}\!\! \absm{\zeta}^{2s} |\hat{\phi}(\zeta)|^2 \bigg( \dfrac{1}{\absm{\zeta}^{2m-1}}\!\! \int_{0}^{\ath}\! \Big| \dfrac{\d_\theta^m\fk(\zeta,\theta)}{\fk(\zeta,\ath)} \Big|^2 \theta^{2m} \, \dif \theta \bigg) \dif \zeta \le C\!\! \int_{\R}\!\! \absm{\zeta}^{2s} |\hat{\phi}(\zeta)|^2\,\dif \zeta\! =\! C \|\phi\|_{H^{s}(\R)}^2.
	\end{align*}
	This completes the proof.
\end{proof}

\begin{corollary}\label{corol:vast}
Let $v_{\ast}(\sigma,y)\vcentcolon= \Phi_{\ast}(\sigma,y\ath)$ for $(\sigma,y)\in\R\times (0,1]\equiv \mS$. Then $v_{\ast}$ solves:
\begin{subequations}\label{vast-eq}
\begin{align}
&-\div_{(\sigma,y)} \big( A_{\ast} \cdot \nabla_{(\sigma,y)} v_{\ast} \big) + \dfrac{\ath \sin(y\ath)}{4} v_{\ast} = 0 && \text{in } \ (\sigma,y)\in \mS,\\
& v_{\ast}(\sigma,1)=\phi(\sigma), \qquad \d_{y}v_{\ast}(\sigma,0) = 0 && \text{for } \ \sigma\in \R,\\
\text{where } \ & \nabla_{(\sigma,y)}\vcentcolon= (\d_{\sigma},\d_y), \qquad \div_{(\sigma,y)}\vcentcolon= \d_{\sigma} + \d_y, && A_{\ast} \vcentcolon= \sin(y\ath)\begin{pmatrix}
	\ath & 0 \\
	0 & 1/\ath
\end{pmatrix}.\nonumber
\end{align}
\end{subequations}
In addition, for $m\ge 2$, there exists constant $C=C(\ath,m)>0$ such that
\begin{align*}
\int_{0}^{1}\!\! \Big\{ \|v_{\ast}(\cdot,y)\|_{H^{s+\frac{1}{2}}(\R)}^2 \!\!+ \|\d_{y} v_{\ast}(\cdot,y)\|_{H^{s-\frac{1}{2}}(\R)}^2\!\! + \|y^m\d_{y}^m v_{\ast}(\cdot, y)\|_{H^{s+\frac{1}{2}-m}(\R)}^2 \Big\} \dif y \le C \|\phi\|_{H^{s}(\R)}.
\end{align*}
\end{corollary}

\section{Elliptic Estimates for Non-Flat Conical Domain}\label{sec:nonflat}
\subsection{Existence of weak solutions}
Set $\Omega_{\eta}\vcentcolon= \big\{(\sigma,\theta) \in \R\times (0,\pi) \,\vert\, 0 <\theta\le \eta(\sigma)\big\}$. Let $v_{\ast}(\sigma,y)$ be given in Corollary \ref{corol:vast}. Using this, we define
\begin{equation}\label{barPhi}
\bar{\Phi}(\sigma,\theta) \vcentcolon= v_{\ast} \big(\sigma, \frac{\theta}{\eta(\sigma)} \big) \qquad \text{for } \ \sigma\in \R \ \text{ and } \ 0<\theta \le \eta(\sigma).
\end{equation}
Then $\bar{\Phi}(\sigma,\eta(\sigma)) = \phi(\sigma)$ for $\sigma \in \R$. Taking derivatives, one has
\begin{equation*}
	\d_{\sigma} \bar{\Phi}(\sigma, \theta) = \Big\{ \d_{\sigma} v_{\ast}- \dfrac{\theta \d_\sigma \eta}{\eta^2} \d_{y} v_{\ast} \Big\}\Big\vert_{y=\theta/\eta(\sigma)}, \qquad \d_{\theta} \bar{\Phi}(\sigma, \theta) = \dfrac{1}{\eta} \d_y v_{\ast} \Big\vert_{y=\theta/\eta(\sigma)}.
\end{equation*}
Taking $L^2$-norm, then by change of variables $(\sigma,y)=(\sigma,\tfrac{\theta}{\eta})\in \R\times (0,1] \equiv \mS$, and Corollary \ref{corol:vast}, there exists a constant $C=C\big( \|(\eta,\eta^{-1},\d_{\sigma}\eta)\|_{L^{\infty}(\R)}, \ath \big)>0$ such that
\begin{align*}
	\iint_{\Omega_{\eta}} \!\!\! |\bar{\Phi}|^2 \sin\theta\, \dif \sigma  \dif \theta 
	=& \iint_{\mS} |v_{\ast}(\sigma,y)|^2 \eta \sin(y\eta)\,\dif\sigma\dif y\\
	\le& \|\eta\|_{L^{\infty}(\R)} \int_{0}^1\!\!\|v_{\ast}(\cdot,y)\|_{L^2(\R)}^2\, \dif y \le C \|\phi\|_{H^{-1/2}(\R)},\\
	\iint_{\Omega_{\eta}}\!\!\! |\d_{\sigma} \bar{\Phi}(\sigma,\theta)|^2 \sin\theta \, \dif \sigma \dif \theta =& \iint_{\mS} \big| \d_{\sigma} v_{\ast} - \dfrac{y \d_{\sigma}\eta}{\eta} \d_{y} v_{\ast} \big|^2 \eta \sin(y\eta) \, \dif \sigma \dif y
	 \le C \|\phi\|_{H^{1/2}(\R)}^2,\\
	 \iint_{\Omega_{\eta}} \!\!\! |\d_{\theta} \bar{\Phi}(\sigma,\theta)|^2\, \dif \sigma \dif \theta =& \iint_{\mS} \dfrac{|\d_y v_{\ast}|^2}{\eta^2} \eta \sin(y\eta)\, \dif \sigma \dif y  \le C \|\phi\|_{H^{1/2}(\R)}. 
\end{align*}
In summary, we obtained the following estimate for $\bar{\Phi}$
\begin{equation}\label{barPhiEst}
	\|\bar{\Phi}\|_{\mathscr{H}^1(\Omega_{\eta})}^2 \vcentcolon=\iint_{\Omega_{\eta}}\! \big\{ |\bar{\Phi}|^2 + |\d_{\sigma}\bar{\Phi}|^2 + |\d_{\theta}\bar{\Phi}|^2 \big\} \sin\theta \, \dif \sigma \dif \theta \le C \|\phi\|_{H^{1/2}(\R)}^2.
\end{equation}
Denote $L\vcentcolon= \d_\theta (\sin\theta \d_\theta) + \sin\theta \d_\sigma^2 - \frac{1}{4} \sin\theta $. Let $\mathring{\Phi}(\sigma,\theta)$ be the weak solution to 
\begin{subequations}\label{oPhi}
\begin{align}
&-L \mathring{\Phi} = L \bar{\Phi} && \text{for } \ (\sigma,\theta) \in \Omega_{\eta},\\
&\mathring{\Phi}\big(\sigma,\eta(\sigma)\big)=0, \quad \d_{\theta} \mathring{\Phi}(\sigma,0) = 0 && \text{for } \ \sigma \in \R.
\end{align}
\end{subequations} 
Applying Lax-Milgram theorem, a unique weak solution $\mathring{\Phi}$ exists in the space 
\begin{equation*}
	\Big\{ f\in \mathscr{H}^1(\Omega_{\eta}) \ \,\Big\vert\, \  f\big(\sigma,\eta(\sigma)\big) = \phi, \ \ \d_\theta f \vert_{\theta=0} = 0 \Big\}.
\end{equation*}
Setting $\Phi \vcentcolon= \mathring{\Phi} + \bar{\Phi}$. Then it can be verified that $\Phi$ is a $\mathscr{H}^1(\Omega_{\eta})$ weak solution to the problem (\ref{Phi-D}). Moreover, by the standard elliptic estimate procedure and (\ref{barPhiEst}), there exists a constant  $C=C\big( \|(\eta,\eta^{-1},\d_{\sigma}\eta)\|_{L^{\infty}(\R)}, \ath \big)>0$ such that
\begin{equation}
	\|\Phi\|_{\mathscr{H}^1(\Omega_{\eta})} \le C \|\phi\|_{H^{1/2}(\R)}.
\end{equation}
\subsection{Coordinate transformation into flat strip}
We consider a coordinate transformation $\mathcal{T}\vcentcolon \Omega_{\eta} \to \mS$ such that the curved domain $\Omega_{\eta}$ is flattened into a strip $\mS$. More specifically, this is given by:
\begin{equation*}
(\sigma,y)=\mathcal{T}(\sigma,\theta)=\big(\sigma,\tau(\sigma,\theta)\big), \qquad \text{where } \ \tau(\sigma,\theta)\vcentcolon=\dfrac{\theta}{\eta(\sigma)}.
\end{equation*}
The corresponding inverse mapping $\mathcal{T}^{-1}\vcentcolon\mS \to \Omega_{\eta}$ is given as:
\begin{equation*}
	(\sigma,\theta) = \mathcal{T}^{-1}(\sigma,y) = \big(\sigma, \rho(\sigma,y)\big) \qquad \text{where } \ \rho(\sigma,y) \vcentcolon= y\eta(\sigma).
\end{equation*}
Note that the Jacobian for $\mathcal{T}$ and $\mathcal{T}^{-1}$ are given by:
\begin{equation*}
	\big| \dfrac{\d(\sigma,y)}{\d(\sigma,\theta)} \big| = \d_\theta \tau = \dfrac{1}{\eta}, \qquad \big| \dfrac{\d(\sigma,\theta)}{\d(\sigma,y)} \big| = \d_y \rho = \eta.
\end{equation*}
For function $f(\sigma,\theta)\in \mC^1(\Omega_{\eta})$, we set $\tilde{f}(\sigma,y)\vcentcolon=f\big(\sigma,\rho(\sigma,y)\big)= f\circ \mathcal{T}^{-1}(\sigma,y)$. If $\eta\in \mC^1(\R)$, then $\tilde{f}\in \mC^1(\mS)$. In light of this, we define $\mathcal{T}_{\ast}$, $\mathcal{T}_{\ast}^{-1}$ as the adjoint operators of $\mathcal{T}$, $\mathcal{T}^{-1}$ respectively:
\begin{align*}
	\mathcal{T}_{\ast} & \vcentcolon \mC^1(\mS) \to \mC^1(\Omega_{\eta}), && \text{with} \quad \mathcal{T}_{\ast} \tilde{f} \vcentcolon= \tilde{f}\circ \mathcal{T} = f,\\
	\mathcal{T}_{\ast}^{-1} & \vcentcolon \mC^1(\Omega_{\eta}) \to \mC^1(\mS), && \text{with } \quad \mathcal{T}_{\ast}^{-1} f \vcentcolon= f\circ \mathcal{T}^{-1} = \tilde{f}.
\end{align*}
By construction, we have $\tau\big(\sigma,\rho(\sigma,y)\big)=y$. Thus by chain rule, one has
\begin{equation*}
	\d_{y}\rho (\sigma,y) = \dfrac{1}{\d_{\theta}\tau} \Big\vert_{\theta=\rho(\sigma,y)}, \qquad \d_{\sigma} \rho(\sigma,y) = - \dfrac{\d_{\sigma}\tau}{\d_{\theta}\tau}\Big\vert_{\theta=\rho(\sigma,y)}.
\end{equation*}
Also, differentiating $f(\sigma,\theta)= \big(\mathcal{T}_{\ast} \circ \mathcal{T}_{\ast}^{-1} f\big)(\sigma,\theta)= \big(\mathcal{T}_{\ast} \tilde{f}\big)(\sigma,\theta)= \tilde{f}\big(\sigma,\tau(\sigma,\theta)\big)$, we get
\begin{subequations}\label{d1d2}
\begin{gather}
	\d_{\sigma} f(\sigma,\theta) = \big\{ \mathcal{T}_{\ast} \circ \d_{1} \circ \mathcal{T}_{\ast}^{-1} f \big\}(\sigma,\theta), \quad \d_{\theta} f(\sigma,\theta) = \big\{ \mathcal{T}_{\ast}\circ \d_2 \circ \mathcal{T}_{\ast}^{-1} f \big\}(\sigma,\theta),\\
	\text{where } \quad \d_{1} \vcentcolon= \d_{\sigma} - \dfrac{\d_{\sigma}\rho}{\d_{y}\rho} \d_{y}, \qquad \d_{2} \vcentcolon= \dfrac{1}{\d_y \rho} \d_y.
\end{gather}
\end{subequations}
With these constructions, we define the function:
\begin{equation*}
 v(\sigma,y) \vcentcolon= \big(\mathcal{T}_{\ast}^{-1} \circ \Phi\big)(\sigma,y) =\Phi\big(\sigma, y \eta(\sigma)\big).
\end{equation*}
Then applying the operator $\mathcal{T}_{\ast}^{-1}$ on the equation (\ref{Phi-D}), and using (\ref{d1d2}), we get $$\d_1\big( \sin(\rho) \d_1 v \big) + \d_2 \big( \sin(\rho) \d_2 v \big) - \frac{1}{4} v \sin(\rho) =0.$$
\begin{subequations}\label{v-D}
Multiplying this equation with $\d_y\rho$, and integrating by parts, it can be rewritten in the divergence form as:
Then the Dirichlet boundary problem (\ref{Phi-D}) can be reformulated as
\begin{align}
&- \mL_{\eta} v \vcentcolon=-\div_{(\sigma,y)} \big( A \cdot \nabla_{(\sigma,y)} v  \big) + \gamma v  = 0 && \text{for } \ (\sigma,y)\in \mS,\\
& v(\sigma,1) = \phi(\sigma), \qquad \d_y v(\sigma,0) = 0 && \text{for } \ \sigma\in\R.
\end{align}
where $\div_{(\sigma,y)} f \vcentcolon= \d_\sigma f + \d_y f$ and $\nabla_{(\sigma,y)} f \vcentcolon= (\d_\sigma f, \d_ y f )^{\top}$, and $A=A(y,\eta,\d_{\sigma}\eta)$ is a symmetric $2$-by-$2$ matrix given by:
\begin{equation}\label{A}
	A \vcentcolon= \sin\rho \begin{pmatrix}
		\d_y \rho & -\d_\sigma \rho \\
		-\d_\sigma \rho & \dfrac{1+|\d_\sigma\rho|^2}{\d_y\rho}
	\end{pmatrix} = \sin(y\eta) \begin{pmatrix}
		\eta  & -y\d_{\sigma}\eta \\
		- y\d_{\sigma}\eta & \dfrac{1+ y^2|\d_\sigma\eta|^2}{\eta} 
	\end{pmatrix},
\end{equation}\label{gamma}
and the scalar function $\gamma=\gamma(y,\eta)$ is given by:
\begin{equation}
    \gamma \vcentcolon= \dfrac{\eta\sin(y\eta)}{4}.
\end{equation}
\end{subequations}
Under this coordinate the auxiliary DN operator, according to (\ref{mG}) is given by:
\begin{equation}\label{mG-flat}
	\mG[\eta](\phi) = \Big\{ \dfrac{1+y|\d_\sigma\eta|^2}{\eta}\d_y v - \d_\sigma \eta \d_\sigma v \Big\}\Big\vert_{y=1}.
\end{equation}
Let $\mathring{\Phi}$ be the solution constructed in (\ref{oPhi}). We set $w(\sigma,y)\vcentcolon= \mathring{\Phi}(\sigma,\rho(\sigma,y))$. Moreover, for a given boundary data $\phi(\sigma)$, recall the functions $v_{\ast}(\sigma,y)$ constructed in (\ref{UpConv}), and $\bar{\Phi}(\sigma,\theta)$ in (\ref{barPhi}). Then $\mathcal{T}_{\ast}^{-1} \bar{\Phi} (\sigma,y) = \bar{\Phi}\big(\sigma,\rho(\sigma,y)\big) = v_{\ast}(\sigma,y)$. Thus applying $\mathcal{T}_{\ast}^{-1}$ on both sides of (\ref{oPhi}), it follows that $w(\sigma,y)$ solves 
\begin{subequations}\label{w-D}
	\begin{align}
		\label{w-D-eq}&-\mL_\eta w  = \mL_{\eta} v_{\ast}  && \text{for } \ (\sigma,y)\in \mS\equiv \R\times (0,1],\\
		\label{w-D-bd}& w(\sigma,1) = 0, \qquad \d_y w(\sigma,0) = 0 && \text{for } \ \sigma\in\R.
	\end{align}
	where we denoted $\mL_{\eta} f \vcentcolon= \div_{(\sigma,y)} \big( A \cdot \nabla_{(\sigma,y)} f  \big) - \gamma f$.
\end{subequations}
\begin{definition}\label{def:weakSol}
	Denote the space $\mathfrak{D}_{\mS}$ as:
	\begin{equation*}
		\mathfrak{D}_{\mS} \vcentcolon= \Big\{ f\in \mathscr{H}(\mS) \ \Big\vert \ \lim\limits_{y\to 1^{-}} f(\cdot,y) = 0 \ \text{ in $L^{\frac{1}{2}}(\R)$ in the sense of trace.}  \Big\}.
	\end{equation*}
	Then $w\in\mathscr{H}^1(\mS)$ is said to be a weak solution to the boundary value problem (\ref{w-D}) if
	\begin{align*}
		&\iint_{\mS} \Big\{ \nabla_{(\sigma,y)} \varphi \cdot A \nabla_{(\sigma,y)} w + \gamma \varphi w  \Big\} \, \dif \sigma \dif y \\ &\qquad = \iint_{\mS} \Big\{ \nabla_{(\sigma,y)} \varphi \cdot A \nabla_{(\sigma,y)} v_{\ast} + \gamma \varphi v_{\ast}  \Big\} \, \dif \sigma \dif y.
	\end{align*}
	Moreover $v\in\mathscr{H}^1(\mS)$ is said to be a weak solution to the boundary value problem (\ref{v-D}) if $v=w+v_{\ast}$ for some $w$ being the weak solution to (\ref{w-D}). 
\end{definition}
\begin{definition}\label{def:solOp}
	For a given function $\eta(\sigma)\vcentcolon \R \to (0,\infty)$ and $\varphi(\sigma)\vcentcolon\R\to \R$, we denote $E_{\eta}[\varphi](\sigma,y)$ as the weak solution of (\ref{v-D}) with boundary data $E_{\eta}[\varphi]\vert_{y=1}=\varphi$.
\end{definition}
\begin{remark}\label{rem:flatSolOp}
	If $\eta\equiv \ath\in (0,\pi)$ is a constant function, then the problem (\ref{v-D}) coincides with (\ref{vast-eq}) in Corollary \ref{corol:vast}. In particular $v_{\ast}$ constructed in Corollary \ref{corol:vast} is given by $v_{\ast}=E_{\ath}[\phi]$. 
\end{remark}
\subsection{Elliptic Estimates for potential function \texorpdfstring{$v$}{v}}
Throughout this section, we will use the following notations for simplicity:
\begin{equation*}
	\nabla \equiv \nabla_{(\sigma,y)} = (\d_\sigma,\d_y)^{\top}, \qquad  f_{\sigma} \equiv \d_\sigma f, \qquad f_y \equiv \d_y f \qquad \text{for function } \ f(\sigma,y).
\end{equation*}
Moreover, we will denote $\absm{D}\vcentcolon= \sqrt{1+|\d_\sigma|^2}$ as the Fourier multiplier acting with respect to the variable $\sigma\in\R$, in other words:
\begin{equation*}
	\absm{D}f(\sigma) \vcentcolon= \dfrac{1}{\sqrt{2\pi}} \int_{\R} (1+|\zeta|^2)^{\frac{1}{2}} \widehat{f}(\zeta) e^{i\sigma\zeta} \, \dif \zeta.
\end{equation*}
\begin{remark}\label{remark:ellipticity}
Some properties of the matrix $A$ is summarised in this remark:
\begin{enumerate}[label=\textnormal{(\roman*)},ref=\textnormal{(\roman*)}]
\item\label{item:ellip1} The characteristic polynomial $\det(A-\lambda I_{2\times 2})=0$ is given by
	\begin{equation*}
		\lambda^2 - y \sinc(y\eta) \big\{1+\eta^2 +y^2 \eta_\sigma^2\big\} \lambda + \sin^2(y\eta) = 0,
	\end{equation*}
	where $\sinc\theta = \frac{\sin\theta}{\theta}$. Solving the quadratic equation, one gets eigenvalues: 
	\begin{align*}
		&\lambda_1,\,\lambda_2\\ &= \dfrac{y\sinc(y\eta)}{2} \bigg\{ 1 + \eta^2 +  y^2 \eta_\sigma^2  \pm  \sqrt{\big(|1+\eta|^2 +  y^2\eta_\sigma^2\big)\big(|1-\eta|^2 +  y^2\eta_\sigma^2\big)} \bigg\}.
	\end{align*}
	It can be shown that $\lambda_1 > 0$ and $\lambda_2 >0$ for $y>0$ and $y\eta \in (0,\pi)$.
	\item\label{item:ellip2} The matrix $A$ can be decomposed as follows:
	\begin{equation*}
		A = \textrm{Q}^{\top} \textrm{Q}, \qquad \text{where } \ \mathrm{Q}\vcentcolon=\sqrt{y\sinc(y\eta)}\begin{pmatrix}
			\eta & -y\d_\sigma \eta \\
			0 & 1
		\end{pmatrix}.
	\end{equation*}
	Thus for any vector $V=\big(V^{\sigma},V^{y}\big)^{\top}$, we have the coercivity:
	\begin{align*}
		V^{\top} \cdot A \cdot V =& V^{\top} \mathrm{Q}^{\top} \mathrm{Q} V = \big(\mathrm{Q} V\big)^{\top} \mathrm{Q} V\\
		=& y \sinc(y\eta) \big\{ |\eta V^{\sigma} - y V^{y} \d_\sigma \eta |^2 + |V^y|^2 \big\} \ge 0,
	\end{align*}
	for $y>0$ and $y\eta \in (0,\pi)$.
\end{enumerate}
\end{remark}
\begin{lemma}\label{lemma:ellip}
	Denote $\teta=\eta-\ath$ for $\ath\in(0,\pi)$. Suppose $(\teta,\phi)\in H^{s+\frac{1}{2}}(\R)\times H^s(\R)$ for some $s>\frac{5}{2}$. Moreover, assume that
	\begin{equation}\label{assump:teta}
		\|\teta\|_{L^{\infty}(\R)} < \min \big\{ \ath\,, \pi-\ath \big\} \quad \text{ and } \quad \|\eta_{\sigma}\|_{L^{\infty}(\R)} <\infty.
	\end{equation}
	Then there exists a constant $C=C(s,\ath)>0$ depending only on $s$ and $\ath$ such that for all $0\le m \le s$, the solution $w(\sigma,y)$ to (\ref{w-D}) satisfies:
	\begin{align*}
		\int_{0}^{1}\!\! \Big\{ \|w(\cdot,y)\|_{H^{m-1/2}(\R)}^2 + \|\nabla w(\cdot,y)\|_{H^{m-1/2}(\R)}^2\Big\} \, y \dif y \le C \Big|\dfrac{\fU_s(\teta)}{\fl(\teta)}\Big|^{2m+1} \|\phi\|_{H^{m}(\R)}, 
	\end{align*} 
	where the functional $\fU_s(\teta)$ and $\fl(\teta)$ are defined as
	\begin{subequations}
	\begin{gather}
		\fl(\teta) \vcentcolon= \sinc\big( \|\teta\|_{L^{\infty}} + \ath \big)\min\Big\{ \dfrac{1}{2}, \dfrac{\big(\ath-\|\teta\|_{L^{\infty}}\big)^2}{1+ 2 \|\teta_\sigma\|_{L^{\infty}}^2 }, \dfrac{\big(\ath- \|\teta\|_{L^{\infty}}\big)^2}{4} \Big\},\\
		\fU_s(\teta) \vcentcolon= \max\big\{  \|\teta\|_{H^{s-\frac{1}{2}}}, \|\teta_\sigma\|_{H^{s-\frac{1}{2}}},  \|\teta^2\|_{H^{s-\frac{1}{2}}}, \|\teta_\sigma \teta \|_{H^{s-\frac{1}{2}}}, \|\teta_\sigma^2\|_{H^{s-\frac{1}{2}}} \big\}.
	\end{gather}
	\end{subequations}
\end{lemma}
\begin{proof}
	The proof is divided into $3$ steps:
	\paragraph{Step 1: Coercivity and Upper Bound of $A$.} Due to the assumption (\ref{assump:teta}), we have
	\begin{equation*}
		0<\ath - \|\teta\|_{L^{\infty}(\R)}\le \eta = \ath + \teta  \le \ath + \|\teta\|_{L^{\infty}(\R)} < \pi. 
	\end{equation*}
	Moreover, by the fact that $\theta\mapsto \sinc\theta$ is strictly positive and monotone decreasing in $[0,\alpha]$ for $\alpha\in(0,\pi)$. It follows that $\min_{0\le \theta\le \alpha} \sinc\theta = \sinc{\alpha} >0$, hence
	\begin{equation*}
		y\sinc(y\eta) \ge y \sinc(\eta) \ge y \sinc \big( \|\teta\|_{L^{\infty}(\R)} + \ath \big) >0 \quad \text{ for } \ y\in(0,1].
	\end{equation*}
	By triangular and Cauchy-Schwartz's inequalities, it holds that
	\begin{equation*}
		|a-b|^2 \ge \dfrac{\delta}{1+\delta} |a|^2 - \delta |b|^2 \qquad \text{for all } \ a,\,b\in\R \ \text{ and } \ \delta\in(0,1).
	\end{equation*}
	Set $\delta_1\vcentcolon= \frac{1}{2}\|\eta_\sigma\|_{L^{\infty}(\R)}^{-2}$, $\delta_2\vcentcolon= \big( \ath - \|\teta\|_{L^{\infty}(\R)} \big)^2$, and $\delta_3\vcentcolon= \sinc\big( \|\teta\|_{L^{\infty}(\R)} + \ath \big)$. Then it follows from Remark \ref{remark:ellipticity}\ref{item:ellip2} that for $(\sigma,y)\in \mS$, and for vector $V=\big(V^{\sigma},V^{y}\big)^{\top}\in\R^2$,
	\begin{align}\label{coerA}
		V^{\top} A V =& y\sinc(y\eta) \big\{ \big| \eta V^{\sigma} - y \d_\sigma \eta V^y \big|^2 + \big|V^y\big|^2 \big\}\\
		\ge& y\sinc(y\eta) \Big\{ \dfrac{\delta_1|\teta+\ath|^2}{1+\delta_1} \big|V^\sigma\big|^2 + (1-\delta_1 y^2|\d_\sigma\eta|^2) \big|V^y\big|^2 \Big\}\nonumber\\
		\ge& y \delta_3 \Big\{ \dfrac{\delta_1\delta_2}{1+\delta_1} \big|V^{\sigma}\big|^2 + \dfrac{1}{2} \big|V^y\big|^2 \Big\} \ge y \fl |V|^2,\nonumber
	\end{align}
	where we denote until the end of this proof that
	\begin{align*}
		\fl\vcentcolon=& \delta_3 \min\Big\{ \dfrac{1}{2} , \dfrac{\delta_1 \delta_2}{1+\delta_1}, \dfrac{\delta_2}{4} \Big\}\\ 
		=& \sinc\big(\|\teta\|_{L^{\infty}(\R)}+\ath\big) \min\bigg\{ \dfrac{1}{2}, \dfrac{\big(\ath - \|\teta\|_{L^{\infty}(\R)}\big)^2}{1+ 2 \| \eta_{\sigma}\|_{L^{\infty}(\R)}^2}, \dfrac{\big(\ath- \|\teta\|_{L^{\infty}(\R)}\big)^2}{4} \bigg\}>0.
	\end{align*}
	Next, we derive an upper bound for $A$. Recall the matrix $A_{\ast}$ from Corollary \ref{corol:vast}. We define the difference matrix $\tilde{A}$ as:
	\begin{align}
		\tilde{A} \vcentcolon=& A - A_{\ast}\\ 
		=& y\sinc(y\eta) \begin{pmatrix}
			\teta^2 + 2\ath \teta & -y\eta_\sigma \eta\\
			-y \eta_\sigma \eta & y^2 \eta_\sigma^2 
		\end{pmatrix} + y \{\sinc(y\eta)-\sinc(y\ath)\}\begin{pmatrix}
			\ath^2 & 0 \\ 0 & 1
	\end{pmatrix}.\nonumber
	\end{align}
    Similarly, we also define the difference scalar function $\tilde{\gamma}$ as
    \begin{align}\label{tGamma}
        \tilde{\gamma} \vcentcolon= \dfrac{\eta\sin(y\eta)}{4} - \frac{\ath\sin(y\ath)}{4}.
    \end{align}
	Using the fact that $|\d_{\theta} \sinc\theta | \le \frac{1}{2}$ for $\theta\in(0,\pi)$, we obtain for $y\in(0,1]$,
	\begin{equation}\label{sincMLT}
		|\sinc(y\eta) - \sinc(y\ath)|\le |\eta-\ath| \sup\limits_{0\le \theta\le \pi/2} y\big| \d_{\theta}\sinc \theta \big| \le \tfrac{1}{2} y \|\teta\|_{L^{\infty}(\R)}.
	\end{equation}
	In addition, using the fact that $\teta\in H^{s+1/2}(\R)$ with $s>\frac{5}{2}$ and the Sobolev estimate $L^{\infty}(\R)\xhookrightarrow[]{} H^{s-2}(\R)$, there exists a constant $C>0$ independent of $\eta$, $\ath$ such that 
	\begin{equation}\label{tAU}
		\|\tA(\cdot,y)\|_{H^{s-1/2}(\R)} + \| \tA(\cdot,y) \|_{L^{\infty}(\R)} \le y C \fU_s(\teta) \quad \text{for } \ y \in(0,1].
	\end{equation}
	\paragraph{Step 2: $H^1$-Estimate.} Since $\mL_{\ath} v_{\ast} \equiv \div\big(A_{\ast} \nabla v_{\ast}\big) - \frac{1}{4}\ath \sin(y\ath) v_\ast=0$ by Corollary \ref{corol:vast}, it follows from (\ref{w-D}) that $w$ solves:
	\begin{equation}\label{tilde-Eq}
		-\div(A \nabla w) + \gamma w = \div (\tA \nabla v_{\ast}) - \tilde{\gamma} v_{\ast}.
	\end{equation}
	Since $w\!\in\! \mathscr{H}^1(\mS)$, we can take a sequence $\varphi_k \! \in\! \mC_c^{\infty}(\mS)$ such that $\nabla \varphi_k \!\to\! \nabla w$ in $\mathscr{H}^0(\mS)$. Take $\varphi_k$ as the test function in the weak form of (\ref{tilde-Eq}), then let $k\to\infty$ to get
	\begin{align*}
		\int_{0}^{1}\!\! \int_{\R} \! \Big\{\nabla w \cdot A \nabla w + \gamma |w|^2 \Big\}\, \dif \sigma \dif y = - \int_{0}^{1}\!\! \int_{\R} \!\! \Big\{ \nabla w \tA \nabla v_{\ast} + \tilde{\gamma} w v_{\ast} \Big\} \,\dif\sigma \dif y.
	\end{align*} 
	Thus applying Cauchy-Schwartz's inequality, the coercivity condition (\ref{coerA}), the upper bound of $\tA$ in (\ref{tAU}), and Corollary \ref{corol:vast}, we have
	\begin{equation}\label{dw-1st}
		\dfrac{\fl}{2}\! \int_{0}^{1} \!\! \Big\{ \|w\|_{L^2}^2 + \|\nabla w\|_{L^2}^2 \Big\}\, y\dif y \le\! \dfrac{C|\fU_s(\teta)|^2}{\fl}\!\! \int_{0}^{1}\!\! \|\nabla v_{\ast}\|_{L^2}^2 \, y \dif y \le\! \dfrac{C|\fU_s(\teta)|^2}{\fl} \|\phi\|_{H^{\frac{1}{2}}}. 
	\end{equation}
	\paragraph{Step 3: $H^{s+\frac{1}{2}}(\R)$-Estimate.} Assume by induction that for $1\le k \le s+\frac{1}{2}$, we have
	\begin{align}\label{indhyp}
		\int_{0}^{1}\!\! \big\{ \| \absm{D}^{k-1} w \|_{L^2}^2 + \| \absm{D}^{k-1} \nabla w \|_{L^2}^2 \big\} \, y\dif y \le \dfrac{C}{\fl^{2k}} \big|\fU_s(\teta)\big|^{2k} \|\phi\|_{H^{k-\frac{1}{2}}}^2.
	\end{align}
	We wish to show that this inequality holds with $k-1$ replaced by $k$. Let $\chi(\zeta)\in \mC_{c}^{\infty}(\R)$ be such that $\chi=1$ if $|\zeta|\le \tfrac{1}{2}$ and $\chi=0$ if $|\zeta|\ge 1$. Moreover, there exists some constant $C>0$ such that $|\chi^{(i)}(\zeta)| \le C$ for $i=1,\dotsc,4$. Then for $h\in(0,1)$, we define $\chi_h(\zeta)\vcentcolon= \chi(h\zeta)$. it follows that for $i=1,\dotsc,4$,
	\begin{equation*}
		\supp(\chi_h) \subset [-\tfrac{1}{h},\tfrac{1}{h}], \qquad \supp\Big( \chi_{h}^{(i)} \Big) \subseteq [-\tfrac{1}{h},-\tfrac{1}{2h}] \cup [\tfrac{1}{2h},\tfrac{1}{h}].
	\end{equation*}
	Moreover, for each $\zeta\in\R$, $\chi_h(\zeta)\to 1$ and $\chi^{(i)}(h\zeta)\to 0$ as $h\to 0^{+}$ with $i=1,\dotsc,4$. Using this, we define the operators:
	\begin{equation}\label{Dkh}
		\absm{D}_h^k f(\sigma,y) \vcentcolon= \dfrac{1}{\sqrt{2\pi}} \int_{\R}\!\! \absm{\zeta}^k \chi_h(\zeta) \widehat{f}(\zeta,y) e^{i\sigma \zeta}\, \dif \sigma \quad \text{and} \quad \Lambda_h^{2k} \vcentcolon= \big( \absm{D}_h^k \big)^2.
	\end{equation}
	Since this is a linear operator, it follows from the condition (\ref{w-D-bd}) that $\Lambda_h^{2k} w(\sigma,1) = 0$ and $\d_y( \Lambda_h^{2k} w )(\sigma,0) = 0 $ for all $\sigma\in\R$. Moreover, since $\chi_h$ is a function with compact support, $\sigma \mapsto \big( \Lambda_h^{2k} w, (\Lambda_h^{2k} w)_y \big) (\sigma,y) \in H^{\infty}(\R)$ for each $y\in[0,1]$. Using $\Lambda_h^{2k}w$ as the test function for the weak form of (\ref{tilde-Eq}), we get
	\begin{align}\label{I-II-Est}
		\textrm{(I)}=&\int_{0}^{1}\!\! \int_{\R} \! \Big\{\nabla \Lambda_h^{2k} w \cdot A \nabla w + \gamma w \Lambda_h^{2k} w \Big\}\, \dif \sigma \dif y \\  
        =& - \int_{0}^{1}\!\! \int_{\R} \!\! \Big\{ \nabla \Lambda_h^{2k} w \cdot \tA \nabla v_{\ast} + \tilde{\gamma} v_{\ast} \Lambda_h^{2k} w  \Big\} \,\dif\sigma \dif y=\textrm{(II)}.\nonumber
	\end{align}
	Computing the first spatial derivative of $\Lambda_h^{2k}w$, we get
	\begin{equation*}
		\nabla \Lambda_h^{2k} w(\sigma,y) = \dfrac{1}{\sqrt{2\pi}} \int_{\R}\!\! \absm{\zeta}^{2k} |\chi_h(\zeta)|^2 (\widehat{w_z},\widehat{w}_y)^{\top} e^{i\sigma \zeta} \,\dif\zeta =\absm{D}_h^k \big( \absm{D}_h^k \nabla w \big).
	\end{equation*}
	Since $\absm{D}_h^k$ is a Fourier multiplier in the variable $\sigma$, and $A_\ast$ given in Corollary \ref{corol:vast} depends only on $y$, $\ath$, one has $[ \absm{D}_h^k, A_{\ast} ]=[ \absm{D}_h^k, \ath \sin(y\ath) ]=0$. Thus it follows that
	\begin{equation*}
		\big[\absm{D}_h^k, \tA\, \big]=\big[\absm{D}_h^k, A \big] \quad \text{and} \quad \big[ \absm{D}_h^k, \tilde{\gamma} \big] = \big[ \absm{D}_h^k, \gamma \big].
	\end{equation*}
	Using this and Parseval's theorem, we get
	\begin{align*}
		\textrm{(I)} 
		=& \int_{0}^1\!\!\! \int_{\R} \! \Big\{\absm{D}_h^k \nabla w \cdot A \absm{D}_h^k \nabla w + \gamma |\absm{D}_h^k w|^2 \Big\}\, \dif\sigma \dif y\\
		&+ \int_{0}^1\!\!\! \int_{\R} \! \Big\{ \absm{D}_h^k \nabla w \cdot \big[ \absm{D}_h^k, \tA\, \big]\nabla w + \absm{D}_h^k w \big[ \absm{D}_h^k, \tilde{\gamma} \big] w \Big\}\, \dif \sigma \dif y.\nonumber
	\end{align*} 
	By construction (\ref{Dkh}), it can be verified that there exists a constant $C(s)>0$ depending only on $s>\frac{5}{2}$ such that for all $0\le k \le s-\frac{1}{2}$,
	\begin{equation*}
		\mathcal{N}_h^k \vcentcolon= \sup\limits_{\beta\le 4} \sup\limits_{\zeta\in\R} \absm{\zeta}^{\beta-k} \big| \d_{\zeta}^{\beta}\big\{ \absm{\zeta}^k \chi_h(\zeta) \big\} \big| \le C(s).
	\end{equation*}
	In particular, this implies that $\absm{D}_h^k$ belongs to the class $\fS^k$ defined in Definition \ref{def:FM}. Setting $t_0\vcentcolon= s-\frac{3}{2}$. Since $0\le k \le s-\frac{1}{2}$ and $s>\frac{5}{2}$, it holds that $t_0>\frac{1}{2}$ and $0\le k \le t_0+1$. Thus by the commutator estimate Proposition \ref{prop:commu}, we get:
	\begin{align}\label{comA}
		\big\| \big[ \absm{D}_h^k, \tA \big]\nabla f \big\|_{L^2} \le& \mathcal{N}_h^k \|\d_{\sigma} A (\cdot,y)\|_{H^{t_0}} \|\nabla f\|_{H^{k-1}}\\ \le& C(s) \|\tA(\cdot,y)\|_{H^{s-\frac{1}{2}}} \|\nabla f\|_{H^{k-1}}.\nonumber
	\end{align}
	By the same argument, we also get
	\begin{align}\label{comSin}
		\big\| \big[ \absm{D}_h^k, \tilde{\gamma} \big] f \big\|_{L^2} \le C(s) \|\teta\|_{H^{s-\frac{1}{2}}} \| f \|_{H^{k-1}}.
	\end{align}
	For the integral $\textrm{(I)}$, we apply Cauchy-Schwartz's inequality, the coercivity condition (\ref{coerA}), upper bounds (\ref{tAU}), and the commutator estimates (\ref{comA})--(\ref{comSin}) to obtain  
	\begin{align*}
		\textrm{(I)} \ge& \dfrac{\fl}{2} \iint_{\mS}\! \big\{ |\absm{D}_h^k \nabla w|^2 + |\absm{D}_h^k w|^2 \big\}\, y\dif \sigma \dif y \\
		&-\dfrac{1}{2\fl}\iint_{\mS}\Big\{ \big\| \big[\absm{D}_h^k, \tA\, \big] \nabla w \big\|_{L^2}^2 + \big\| \big[ \absm{D}_h^k, \tilde{\gamma} \big] w \big\|_{L^2}^2  \Big\} \dfrac{\dif y}{y}\\
		\ge & \dfrac{\fl}{2}\! \iint_{\mS}\!\! \big\{ |\absm{D}_h^k \nabla w|^2 \!+\! |\absm{D}_h^k w|^2 \big\} y\dif \sigma \dif y \! - \dfrac{C |\fU_s(\teta)|^2}{2\fl}\!\!\! \iint_{\mS}\!\! \big\{ \|\nabla w\|_{H^{k-1}}^2 \!+\! \|w\|_{H^{k-1}}^2 \big\} y\dif y.
	\end{align*}
	On the other hand, for the integral \textrm{(II)}, we apply a similar argument to get
	\begin{align*}
		\textrm{(II)} =& -\int_0^1\!\!\!\int_{\R} \! \absm{D}_h^k \nabla w \big\{ \tA \absm{D}_h^k \nabla v_{\ast} + \big[ \absm{D}_h^k, \tA \big] \nabla v_\ast \big\}\, \dif \sigma \dif y \\ 
        &- \int_0^1\!\!\!\int_{\R} \! \absm{D}_h^k w \Big\{ \tilde{\gamma} \absm{D}_h^k v_{\ast} + \big[ \absm{D}_h^k, \tilde{\gamma} \big] v_{\ast}  \Big\}\, \dif \sigma \dif y \\
		\le & \dfrac{\fl}{4}\! \iint_{\mS}\!\! \big\{ |\absm{D}_h^k \nabla w|^2 + |\absm{D}_h^k w|^2  \big\} y \dif \sigma \dif y \!+\! \dfrac{C|\fU_s(\teta)|^2}{\fl}\!\! \int_{0}^{1}\!\!\! \big\{ \|\nabla v_{\ast}\|_{H^{k}}^2 + \|v_{\ast}\|_{H^{k}}^2 \big\} y\dif y.
	\end{align*}
	Combining the estimates for \textrm{(I)}--\textrm{(II)} in (\ref{I-II-Est}), then applying Corollary \ref{corol:vast}, we obtain
	\begin{align*}
		& \iint_{\mS}\!\! \big\{ |\absm{D}_h^k \nabla w|^2 + |\absm{D}_h^k w|^2  \big\}\, y \dif \sigma \dif y\\ 
		\le& \dfrac{C|\fU_s(\teta)|^2}{\fl^2} \iint_{\mS} \! \big\{ \|\nabla w\|_{H^{k-1}}^2 + \|w\|_{H^{k-1}}^2 + \|\nabla v_{\ast}\|_{H^{k}}^2 + \|v_{\ast}\|_{H^{k}}^2 \big\}\,y\dif y \\
		\le& \dfrac{C|\fU_s(\teta)|^2}{\fl^2} \Big\{ \|\phi\|_{H^{k+\frac{1}{2}}}^2 + \iint_{\mS} \! \big\{ \|\nabla w\|_{H^{k-1}}^2 + \|w\|_{H^{k-1}}^2 \big\} \, y \dif y  \Big\}.
	\end{align*}
	Taking the limit $h\to 0^+$, and using induction hypothesis (\ref{indhyp}), we obtain that
	\begin{align*}
		&\int_{0}^1\!\! \big\{ \|\absm{D}^k \nabla w \|_{L^2}^2 + \|\absm{D}^k w\|_{L^2}^2 \big\}\, y\dif y\\
		  \le& \dfrac{C|\fU_s(\teta)|^2}{\fl^2} \Big\{ \|\phi\|_{H^{k+\frac{1}{2}}}^2 + \dfrac{C}{\fl^{2k}} \big|\fU_s(\teta)\big|^{2k} \|\phi\|_{H^{k-\frac{1}{2}}}^2 \Big\} \le \dfrac{C|\fU_s(\teta)|^{2(k+1)}}{\fl^{2(k+1)}} \|\phi\|_{H^{k+\frac{1}{2}}}^2.
	\end{align*}
	Since $\|f\|_{H^k} \vcentcolon= \|\absm{D}^k f\|_{L^2}$, this concludes the proof of statement for $k+1$.
\end{proof}
The following corollary is an immediate consequence of Corollary \ref{corol:vast}, Lemma \ref{lemma:ellip}, and Poincar\'e's inequality:
\begin{corollary}\label{corol:vH1Est}
	Suppose the same assumption of Lemma \ref{lemma:ellip} holds. Then $v\vcentcolon= w + v_{\ast}$ is a weak solution to the problem (\ref{v-D}), and there exists a constant $C(\ath,s)>0$ such that for all $0 \le m \le s$,
	\begin{equation*}
		\int_{0}^{1}\!\! \big\{ \|v(\cdot,y) \|_{H^{m+\frac{1}{2}}(\R)}^2 + \| \nabla v(\cdot,y) \|_{H^{m-\frac{1}{2}}(\R)}^2 \big\}\, y \dif y \le C \Big|\dfrac{\fU_s(\teta)}{\fl(\teta)}\Big|^{2m+1} \|\phi\|_{H^{m}(\R)}^2.
	\end{equation*}
\end{corollary}

By the standard argument for elliptic boundary estimate, one can also verify:
\begin{lemma}\label{lemma:vHigh}
	Suppose the same assumption of Lemma \ref{lemma:ellip} holds. Then 
	\begin{equation*}
		v_y \in \mC^0\big((0,1]; H_\sigma^{s-1}(\R)\big), \quad v_{yy} \in \mC^0\big((0,1]; H_\sigma^{s-2}(\R)\big), \quad \d_y^3 v \in \mC^0\big((0,1]; H_\sigma^{s-3}(\R)\big).
	\end{equation*}
	In addition, there exists a positive monotone increasing function $x\mapsto C(x)$ such that 
	\begin{align*}
		\int_{0}^{1}\!\!\! \big\{ y^3 \|\d_y^2 v\|_{H^{s-\frac{3}{2}}(\R)}^2 \!\!+ y^5 \|\d_y^3 v\|_{H^{s-\frac{5}{2}}(\R)}^2 \!\! + y^7\|\d_y^4 v\|_{H^{s-\frac{7}{2}}(\R)}^2 \big\} \dif y \le C\big(\|\teta\|_{H^{s+\frac{1}{2}}(\R)}\big) \|\phi\|_{H^{s}(\R)}^2.
	\end{align*} 
\end{lemma}
The proof of this lemma is omitted since it follows the routine procedure for the boundary estimate for elliptic operator, which can be found in  \cite{Tools}. 

\section{Construction of DN Operator for Conical Surfaces}\label{sec:DN-operator}
The main aim of this section is to prove Theorem \ref{thm:DNSob}, which is the construction of operator $\mG[\eta](\psi)$ in Sobolev spaces $H^s(\R)$ with $s>\frac{5}{2}$. 

For $k\in\R$, we denote $H_0^k(\R)$ as the closure of $\mC_{c}^{\infty}(\R)$ with the $H^k(\R)$ norm. Recall from Definition \ref{def:solOp} that for a given function $\varphi(z)\vcentcolon \R \to \R$, we denoted $E_{\eta}[\varphi]$ as the solution to (\ref{v-D}) with boundary data $E_{\eta}[\varphi]\vert_{y=1}=\varphi$. Then for a given $\eta-\ath \in H^{s+\frac{1}{2}}(\R)$, and $m\in[\frac{1}{2},s]$, we set the bilinear functional $\fG[\eta]\vcentcolon H^{m}(\R)\times H_0^{\frac{1}{2}}(\R) \to \R$ as
\begin{equation}\label{fG}
	\fG[\eta](\phi,\varphi) \vcentcolon= \iint_{\mS} \nabla E_{\eta}[\varphi] \cdot \big(A\nabla E_{\eta}[\phi]\big) \dif \sigma \dif y. 
\end{equation}
Note that this is well-defined since $E_{\eta}[\varphi]\in \mathscr{H}^1(\mS)$ if $\varphi\in H^{\frac{1}{2}}(\R)$ by Corollary \ref{corol:vH1Est}. Next, we split our analysis into two cases:
\paragraph{Case 1: $m=\frac{1}{2}$.} Applying the estimate of Corollary \ref{corol:vH1Est}, one has
\begin{align*}
	|\fG[\eta](\phi,\varphi)| \le& \|A\|_{L^{\infty}(\R)} \bigg( \int_0^1 \|\nabla E_{\eta}[\phi]\|_{L^2(\R)}^2\, y\dif y \bigg)^{\frac{1}{2}} \bigg( \int_0^1 \|\nabla E_{\eta}[\varphi]\|_{L^2(\R)}^2\, y\dif y \bigg)^{\frac{1}{2}}\\
	\le& C \big\{ \ath^{-1} + \fU_s(\teta) \big\} \Big|\dfrac{\fU_s(\teta)}{\fl(\teta)}\Big|^2 \|\phi\|_{H^{\frac{1}{2}}(\R)} \|\varphi\|_{H^{\frac{1}{2}}(\R)}.
\end{align*}
Set the functional $\tilde{\mG}[\eta](\phi)\vcentcolon H_{0}^{\frac{1}{2}}(\R) \to \R$ by $\big\{ \tilde{\mG}[\eta](\phi)\big\}(\varphi) \vcentcolon= \fG[\eta](\phi,\varphi)$. Then the above estimate implies that $\tilde{\mG}[\eta](\phi)\in H^{-\frac{1}{2}}(\R)$, where $H^{-\frac{1}{2}}(\R)$ is the dual space of $H_0^{\frac{1}{2}}(\R)$. In particular,
\begin{equation*}
	\|\tilde{\mG}[\eta](\phi)\|_{H^{-\frac{1}{2}}(\R)} \le C \big\{ \ath^{-1} + \fU_s(\teta) \big\}\Big|\dfrac{\fU_s(\teta)}{\fl(\teta)}\Big|^2 \|\phi\|_{H^{\frac{1}{2}}(\R)}.
\end{equation*} 
\paragraph{Case 2: $\frac{1}{2}< m \le s$.} Let $\phi\in H^{m}(\R)$. Then by the same argument as above, the functional $\tilde{\mG}[\eta](\phi)$ belongs to $H^{-\frac{1}{2}}(\R)$. Next, we wish to show that $\absm{D}^{m-1}\tilde{\mG}[\eta](\phi)$ defined in the distributional sense can be extended as a functional on $L^2(\R)$. To do this, we construct the following extension: for $f\in L^1(\R)$, we set
\begin{equation}\label{dagger}
	f^{\dagger}(\sigma,y) \vcentcolon= \chi\big( (1-y) D \big) f = \dfrac{1}{\sqrt{2\pi}} \int_{\R} \chi\big((1-y)\zeta\big) \widehat{f}(\zeta) e^{i\sigma \zeta} \, \dif \zeta,
\end{equation}
where $\chi\in \mC_{c}^{\infty}(\R)$ is a positive even function such that
\begin{equation*}
	\int_{\R} |\chi(\zeta)|^2\, \dif \zeta = 1, \qquad \chi(\zeta) = \left\{\begin{aligned}
		&1 && \text{if } \ |\zeta|\le \tfrac{1}{2},\\ 
		&0 && \text{if } \ |\zeta|\ge 1.
	\end{aligned}\right.
\end{equation*}
With this construction, it can be verified that if $f\in L^1(\R)$ then $\sigma \mapsto f^{\dagger}(\sigma,y) \in \mC^{\infty}(\R)$ for each fixed $0\le y<1$. Moreover, $\absm{D}^{\alpha} f^{\dagger} = \big(\absm{D}^{\alpha} f\big)^{\dagger}$ for $\alpha\in\R$, and $f^{\dagger}(\sigma,y)\to f(\sigma)$ as $y\to 1^{-}$ for a.e. $\sigma\in\R$. The proof of following proposition can be found in Lemma 2.34 of the textbook \cite{lannes2}. 
\begin{proposition}\label{prop:dagger}
For $f\in L^2(\R)$, let $f^{\dagger}(\sigma,y)$ be function constructed in (\ref{dagger}). Then 
\begin{equation*}
	\int_{0}^{1}\!\!\! \int_{\R} \!\! \big|\absm{D}^{-\frac{1}{2}} \nabla f^{\dagger}\big|^2 \dif \sigma \dif y \le 2 \| f \|_{L^2(\R)}^2.
\end{equation*}
\end{proposition}

\begin{proposition}\label{prop:fGHm}
	Let $s>\frac{5}{2}$ and $m\in (\frac{1}{2},s]$. For $\phi\in H^{m}(\R)$, define the distribution:
	\begin{equation*}
		\big\{ \absm{D}^{m-1} \tilde{\mG}[\eta](\phi) \big\}(\varphi) \vcentcolon= \fG[\eta]\big( \phi, \absm{D}^{m-1} \varphi \big) \quad \text{for} \quad \varphi\in \mC_{c}^{\infty}(\R).
	\end{equation*}
	Then $\absm{D}^{m-1} \tilde{\mG}[\eta](\phi)$ can be extended as a bounded linear map on $L^2(\R)$, and there exists a unique $g_{\eta}[\phi]\in L^2(\R)$ such that for all $\varphi\in L^2(\R)$,
	\begin{gather*}
		\big\{ \absm{D}^{m-1} \tilde{\mG}[\eta](\phi) \big\}(\varphi) = \int_{\R} \varphi g_{\eta}[\phi]\, \dif \sigma \qquad \text{where}\\
		\big\|g_{\eta}[\phi]\big\|_{L^2(\R)} \le C \big(\ath^{-1}+\fU_s(\teta)\big) \Big|\dfrac{\fU_s(\teta)}{\fl(\teta)}\Big|^{m+\frac{1}{2}} \|\phi\|_{H^{m}(\R)}.
	\end{gather*}
\end{proposition}
\begin{proof}
	We denote $v\vcentcolon= E_{\eta}[\phi]$. Fix $\varphi\in \mC_{c}^{\infty}(\R)$. Then by Definition \ref{def:solOp} and the construction (\ref{dagger}), it holds that for all $\alpha\in\R$,
	\begin{equation}
		\varphi^{\dagger}\big\vert_{y=1} =\varphi = E_{\eta}[\varphi]\big\vert_{y=1} \quad \text{and} \quad \big(\absm{D}^{\alpha}\varphi\big)^{\dagger}\big\vert_{y=1}= \absm{D}^{\alpha} \varphi = E_{\eta}\big[ \absm{D}^{\alpha} \varphi \big]\big\vert_{y=1}.
	\end{equation}
	Thus by Corollary \ref{corol:vH1Est} and Lemma \ref{lemma:vHigh}, $E_{\eta}\big[\absm{D}^{\alpha}\varphi\big]-\big(\absm{D}^{\alpha}\varphi\big)^{\dagger}\in \mathfrak{D}_{\mS}$ where $\mathfrak{D}$ is the functional space defined in Definition \ref{def:weakSol}. Using this as a test function for the weak form of $v$ with equation (\ref{v-D}), we get
	\begin{align*}
		\fG[\eta]\big(\phi,\absm{D}^{\alpha}\varphi\big) \!= \!\iint_{\mS}\!\!\nabla E_{\eta}\big[ \absm{D}^{\alpha} \varphi \big] \cdot (A\nabla v)\, \dif \sigma \dif y = \!\iint_{\mS}\!\! \nabla \big( \absm{D}^{\alpha} \varphi \big)^{\dagger} \cdot (A\nabla v)\, \dif \sigma \dif y.
	\end{align*} 
	Taking $\alpha=m-1$, then by Parseval's theorem and $\absm{D}^{\alpha} f^{\dagger} = \big(\absm{D}^{\alpha} f\big)^{\dagger}$, we have
	\begin{align}\label{fG0}
		&\big\{ \absm{D}^{m-1} \tilde{\mG}[\eta](\phi) \big\}(\varphi) \vcentcolon= \fG[\eta]\big(\phi,\absm{D}^{m-1}\varphi\big) 
		=\iint_{\mS}\!\! \nabla \absm{D}^{m-1}\varphi^{\dagger} \cdot (A\nabla v)\, \dif \sigma \dif y\\
		=& \iint_{\mS} \!\! \nabla \absm{D}^{-\frac{1}{2}} \varphi^{\dagger} \big\{ A\absm{D}^{m-\frac{1}{2}} \nabla v + \big[ \absm{D}^{m-\frac{1}{2}}, \tA \big] \nabla v \big\} \nonumber\\
		\le & C \|\varphi\|_{L^2(\R)} \bigg( \iint_{\mS} \! \big\{ A\absm{D}^{m-\frac{1}{2}}\nabla v + \big[ \absm{D}^{}m-\frac{1}{2}, \tA \big] \nabla v \big\}^2\, \dif \sigma \dif y \bigg)^{\frac{1}{2}}, \nonumber
	\end{align}
	where we used Proposition \ref{prop:dagger} in the last line. Moreover, by Corollary \ref{corol:vH1Est} and (\ref{tAU}),
	\begin{align}\label{fG1}
		\iint_{\mS}\! |A\absm{D}^{m-\frac{1}{2}}\nabla v|^2\,\dif \sigma \dif y 
		\le C\big(\ath^{-1}+\fU_s(\teta)\big)^2 \Big|\dfrac{\fU_s(\teta)}{\fl(\teta)}\Big|^{2m+1} \|\phi\|_{H^{m}(\R)}^2.
	\end{align}
	Setting $\delta=1$, $k=m-\frac{1}{2}$, and $t_0=s-\frac{3}{2}$ in Proposition \ref{prop:commuLam}, it holds that
	\begin{align}\label{fG2}
		\iint_{\mS}\big| \big[\absm{D}^{m-\frac{1}{2}},\tA\big] \nabla v \big|^2 \dif \sigma \dif y \le& C \int_{0}^{1}\!\! \|\tA\|_{H^{s-\frac{1}{2}}}^2\|\nabla v\|_{H^{m-\frac{3}{2}}}^2 \dif y\\
		\le & C |\fU_s(\teta)|^2 \Big| \dfrac{\fU_s(\teta)}{\fl(\teta)} \Big|^{2m-1} \|\phi\|_{H^{m-1}}^2.\nonumber  
	\end{align}
	Substituting (\ref{fG1})--(\ref{fG2}) into (\ref{fG0}), it follows that
	\begin{align*}
		\big|\big\{ \absm{D}^{m-1} \tilde{\mG}[\eta](\phi) \big\}(\varphi)\big| \le C \big(\ath^{-1}+\fU_s(\teta)\big) \Big|\dfrac{\fU_s(\teta)}{\fl(\teta)}\Big|^{m+\frac{1}{2}} \|\phi\|_{H^{m}(\R)} \|\varphi\|_{L^2(\R)}.
	\end{align*}
	Since $\mC_{c}^{\infty}(\R)$ is dense in $L^2(\R)$, the above functional can be extended to $L^2(\R)$. Then by Riesz representation theorem, there is a unique $g_{\eta}[\phi]\in L^2(\R)$ such that
	\begin{gather*}
		\big\{ \absm{D}^{m-1} \tilde{\mG}[\eta](\phi) \big\}(\varphi) = \int_{\R} \varphi g_{\eta}[\phi]\dif \sigma \quad \text{ for } \ \varphi\in L^2(\R) \ \text{ and } \\
		\big\|g_{\eta}[\phi]\big\|_{L^2(\R)} \le C \big(\ath^{-1}+\fU_s(\teta)\big) \Big|\dfrac{\fU_s(\teta)}{\fl(\teta)}\Big|^{m+\frac{1}{2}} \|\phi\|_{H^{m}(\R)}.
	\end{gather*}  
	This proves the proposition.
\end{proof}

For all $\varphi\in \mC_{c}^{\infty}(\R)$, it follows that 
\begin{align*}
	\big\{ \tilde{\mG}[\eta](\phi) \big\} (\varphi) =& \big\{ \absm{D}^{m-1} \tilde{\mG}[\eta](\phi) \big\} \big(\absm{D}^{1-m}\varphi\big)\\
	=& \int_{\R} g_{\eta}[\phi] \absm{D}^{1-m}\varphi \,\dif \sigma = \int_{\R} \varphi \absm{D}^{1-m} g_{\eta}[\phi]\, \dif \sigma,
\end{align*}
where $\absm{D}^{1-m} g_{\eta}[\phi]\in H^{m-1}$. In light of this we define $\tilde{\mG}[\eta](\phi)\vcentcolon= \absm{D}^{1-m} g_{\eta}[\phi]$. Then the previous proposition implies that
\begin{equation}\label{tmGEst}
	\|\tilde{\mG}[\eta](\phi)\|_{H^{m-1}(\R)} = \|g_{\eta}[\phi]\|_{L^2(\R)} \le C \big(\ath^{-1}+\fU_s(\teta)\big) \Big|\dfrac{\fU_s(\teta)}{\fl(\teta)}\Big|^{m+\frac{1}{2}} \|\phi\|_{H^{m}(\R)}.
\end{equation}
In addition, the following states that if the regularity is large enough then $\tilde{\mG}[\eta](\phi)$ coincides with $\sin(\eta) \mG[\eta](\phi)$, which is the classical definition given in (\ref{mG}).
\begin{corollary}\label{corol:tmG}
	Suppose $s>\frac{5}{2}$ and $\frac{1}{2}<m\le s$ are large enough that $E_{\eta}[\phi]\in\mC^2(\mS)$. Let $\tilde{\mG}[\eta](\phi) \vcentcolon= \absm{D}^{1-m} g_{\eta}[\phi]$ where $g_{\eta}[\phi]$ is constructed in Proposition \ref{prop:fGHm}. Then 
	\begin{equation*}
		\tilde{\mG}[\eta](\phi) = \sin(\eta)\cdot \mG[\eta](\phi) \qquad \text{for } \ \sigma \in \R.
	\end{equation*}
\end{corollary}
\begin{proof}
	Let $\varphi\in \mC_{c}^{\infty}(\R)$ and denote $v=E_{\eta}[\phi]$. Then applying Proposition \ref{prop:fGHm}, and integrating by parts with the construction (\ref{fG}), we get
	\begin{align*}
		&\int_{\R} \varphi\cdot \tilde{\mG}[\eta](\phi)\, \dif \sigma = \big\{ \absm{D}^{m-1} \tilde{\mG}[\eta](\phi) \big\} \big(\absm{D}^{1-m}\varphi\big) = \fG[\eta](\phi,\varphi)\\
		=& \iint_{\mS} \nabla E_{\eta}[\varphi] \cdot \big( A \nabla E_{\eta}[\phi] \big)\, \dif \sigma \dif y =  \int_{\R} \varphi \cdot \big\{ A_{21} v_{\sigma} + A_{22} v_y \big\}\vert_{y=1} \dif \sigma.
	\end{align*}
	By the expression of $A$ given in (\ref{A}), and definition (\ref{mG-flat}), it follows that
	\begin{align*}
		\big\{A_{21} v_{\sigma} + A_{22} v_{y}\big\} \big\vert_{y=1} = \sin(\eta) \Big\{ \dfrac{1+y|\d_{\sigma}\eta|^2}{\eta} v_{y} - \eta_\sigma v_{\sigma} \Big\}\Big\vert_{y=1} = \sin(\eta) \mG[\eta](\phi).
	\end{align*} 
	Thus $ \tilde{G}[\eta](\phi) = \sin(\eta) \mG[\eta](\phi)$ by Fundamental Lemma of Calculus of Variation. This concludes the proof of corollary.
\end{proof}
To finish the proof of Theorem \ref{thm:DNSob}, we set $\mG[\eta](\phi) \vcentcolon= \tilde{\mG}[\eta](\phi)/\sin(\eta)$. Then for large enough $s$ and $m$, Corollary \ref{corol:tmG} implies that $\mG[\eta](\phi)$ coincides with the classical definition given in (\ref{mG}). Combining the inequality $\sin(\eta) \ge \sin\big( \ath - \|\teta\|_{L^{\infty}(\R)} \big)$ with the estimate (\ref{tmGEst}), we obtain the inequality stated in Theorem \ref{thm:DNSob}.\qed



\section{Shape Derivative of Conical DN operator}\label{sec:shape}
In this section, we derive the shape derivative formula for the operator $\mG[\eta](\phi)$, which is stated in Theorem \ref{thm:shape}. To start, we provide the following observations on terms $\mB$, $V$ defined in (\ref{BV}):

\begin{remark}\label{rem:BV'}
By the definition of DN operator (\ref{mG-flat}), and using that $v\vert_{y=1}=\phi$, one sees that $\mB$ can be written as 
\begin{align*}
\mathcal{B}= \dfrac{ \d_{\sigma} \eta \d_{\sigma} \phi + G[\eta](\phi) }{1+|\d_{\sigma} \eta|^2}
= \dfrac{ \eta \d_{\sigma} \eta \d_{\sigma} v + (1+y|\d_{\sigma} \eta|^2)\d_y v - \eta \d_{\sigma} \eta \d_{\sigma} v }{\eta (1+y|\d_{\sigma} \eta|^2)}\Big\vert_{y=1} = \dfrac{\d_y v}{\eta}\big\vert_{y=1}.
\end{align*}
Translating this back into the original spherical coordinate system $(r,\theta)$, we have:
\begin{equation*}
    \mB\vert_{\sigma=-\ln r} = \sqrt{r} \d_{\theta} \Psi(r,\theta)\vert_{\theta=\Theta(r)} = \sqrt{r} \d_{\theta} \Psi(r,\theta)\vert_{\theta=\Theta(r)},
\end{equation*}
where $\Psi$ is the velocity potential function solving the problem (\ref{ellip-cone}). Thus $\mB$ is the $r^{3/2}$ multiple of polar angular velocity at the boundary surface, which is $\frac{1}{r}\d_\theta \Psi\vert_{\theta=\Theta(r)}$. On the other hand, since $\phi(\sigma)= v(\sigma,y)\vert_{y=1}$, it holds that
\begin{align*}
V = \big\{ \d_\sigma v - \dfrac{\d_y v}{\eta} \d_\sigma \eta \big\}\big\vert_{y=1}, \quad  \text{ which implies }  \quad V\vert_{\sigma=-\ln r} = r \d_{r} \big(\sqrt{r}\, \Psi\big) \big\vert_{\theta=\Theta(r)}.
\end{align*}     
\end{remark}
\subsection{Surface perturbation}
We recall the Dirichlet boundary problem (\ref{v-D}) in the flat strip $(\sigma,y)\in \mS \equiv \R\times(0,1]$ stated as:
\begin{subequations}\label{vflat}
\begin{align}
&-\mL_{\eta} v \equiv -\div\big(A \cdot\nabla v\big) + \gamma v = 0 && \text{for } \ (\sigma,y) \in \mS,\label{vflat-1}\\
& v(\sigma,1)  = \phi(\sigma), \quad \d_y v(\sigma,0) = 0  && \text{for } \ \sigma\in \R,\label{vflat-2}
\end{align}
\vspace*{-0.5cm}
\begin{equation*}
\text{where } \quad  A=A(\eta) \vcentcolon= \sin(y\eta) \begin{pmatrix}
		\eta & -y\d_\sigma \eta \\
		- y \d_\sigma \eta & \dfrac{1+y^2 |\d_\sigma\eta|^2}{\eta}
	\end{pmatrix}, \quad \gamma=\gamma(\eta)\vcentcolon= \dfrac{\eta \sin(y\eta)}{4}.
\end{equation*}
\end{subequations}
Here, we used the abbreviations $\nabla\equiv(\d_\sigma,\d_y)^{\top}$, $\div\equiv\d_\sigma + \d_y$. Expanding (\ref{vflat-1}), and dividing the resultant equation by $-\eta \sin(y\eta)$, we also have
\begin{gather}
\d_\sigma^2 v + \mathfrak{a} \d_y^2 v  + \mathfrak{b} \d_\sigma \d_y v - \mathfrak{c} \d_y v - \dfrac{1}{4} v =0,\label{Leta} \\
	\text{where } \ \mathfrak{a}\vcentcolon= \dfrac{1+y^2|\d_\sigma\eta|^2}{\eta^2}, \quad \mathfrak{b} \vcentcolon= - \dfrac{2y\d_\sigma\eta}{\eta}, \quad \mathfrak{c}\vcentcolon= \dfrac{y\eta \d_\sigma^2 \eta - 2y |\d_\sigma\eta|^2 - \eta \cot(y\eta)}{\eta^2}.\nonumber
\end{gather}
Set $\rho\vcentcolon=y\eta(\sigma)$. For given $\ep>0$ and $h\in\mC_{c}^{\infty}(\R)$, we define the surface perturbations:
\begin{subequations}\label{Rep}
\begin{gather}
\eta_{\ep}\vcentcolon= \eta + \ep h , \qquad  \rho_\ep \vcentcolon= y ( \eta + \ep h ) = \rho+ \ep y h , \\
A_{\ep}\vcentcolon= A(\eta_{\ep}), \qquad \gamma_{\ep}\vcentcolon= \gamma(\eta_{\ep}), \qquad \mL_{\eta_{\ep}} = \div(A_{\ep} \nabla ) - \gamma_{\ep} .
\end{gather}
\end{subequations}
\begin{subequations}\label{vpert}
Moreover, we let $v_{\ep}(\sigma,y)$ be the solution to the perturbed Dirichlet boundary problem:
\begin{align}
&-\mL_{\eta_{\ep}} v_{\ep} = -\div\big( A_{\ep}\cdot \nabla v_{\ep} \big) + \gamma_{\ep} v_{\ep} = 0 && \text{for } \ (\sigma,y) \in \mS,\label{vpert-1}\\
& v_{\ep}(\sigma,1) = \phi(\sigma), \qquad  \d_y v_{\ep}(\sigma,0) = 0  && \text{for } \ \sigma\in \R.\label{vpert-2}
\end{align}
\end{subequations}
In addition, we define the limits
\begin{subequations}\label{detah}
\begin{gather}
	\dif_{\eta} v\cdot h \vcentcolon= \lim\limits_{\ep\to 0 } \dfrac{v_{\ep}-v}{\ep}, \qquad \dif_{\eta} \rho\cdot h\vcentcolon= \lim\limits_{\ep\to 0} \dfrac{\rho_{\ep}-\rho}{\ep},\\ \dif_{\eta} A \cdot h \vcentcolon= \lim\limits_{\ep\to 0} \dfrac{A(\eta_{\ep})-A(\eta)}{\ep}, \qquad \dif_{\eta} \gamma \cdot h  \lim\limits_{\ep\to 0} \dfrac{\gamma(\eta_{\ep})-\gamma(\eta)}{\ep}.
\end{gather}
\end{subequations}
It immediately follows that $\dif_{\eta} \rho \cdot h = y h(\sigma)$. Furthermore, subtracting (\ref{vpert}) with (\ref{vflat}), then dividing the resultant equation with $\ep$, and taking the limit $\ep\to 0$, it can be verified that $\dif_{\eta} v\cdot h$ solves the problem:
\begin{subequations}\label{dvh}
\begin{align}
&-\mL_{\eta} (\dif_{\eta} v \cdot h) = \div\big( (\dif_{\eta} A \cdot h) \cdot \nabla v \big) - (\dif_{\eta} \gamma \cdot h) v && \text{for } \ (\sigma,y)\in\mS,\\
& (\dif_{\eta} v\cdot h) (\sigma,1) = 0 \qquad \d_y(\dif_{\eta} v \cdot h)(\sigma,0)=0 && \text{for } \ \sigma\in\R.
\end{align}
\end{subequations}
\begin{proposition}\label{prop:dG}
Suppose $(\teta,\phi)\in H^{s+\frac{1}{2}}(\R)\times H^s(\R)$ for some $s>0$ large enough, and let $\mG[\eta](\phi)$ be the operator defined by (\ref{mG-flat}). For $h\in\mC_{c}^{\infty}(\R)$, set $$\dif_\eta \mG[\eta](\phi)\cdot h \vcentcolon= \lim\limits_{\ep \to 0} \dfrac{\mG[\eta+\ep h](\phi) - \mG[\eta](\phi)}{\ep},$$ and $\dif_{\eta} v\cdot h$ be the solution to (\ref{dvh}). Then 
\begin{align*}
\dif_{\eta} \mG[\eta](\phi)\cdot h =& \Big\{ \dfrac{1+y|\d_{\sigma} \eta|^2}{\eta}\d_y (\dif_{\eta} v \cdot h) - \d_{\sigma} \eta \d_{\sigma} (\dif_{\eta} v \cdot h)\Big\}\Big\vert_{y=1}\nonumber\\
&+ \Big\{ \d_{\sigma} h \dfrac{2y\d_{\sigma} \eta}{\eta} \d_y v - \d_{\sigma} h \d_{\sigma} v - h\dfrac{1+y|\d_{\sigma} \eta|^2}{\eta^2} \d_y v \Big\}\Big\vert_{y=1}.
\end{align*}
\end{proposition}
\begin{proof}
	Applying the product and chain rules for shape derivative, one can verify that
	\begin{align*}
		&\dif_{\eta} \big(\d_{\sigma} \eta\big)\cdot h = \d_{\sigma} h, \quad &&\dif_{\eta} \Big\{ \dfrac{1+y|\d_{\sigma}\eta|^2}{\eta} \Big\}\cdot h = -\dfrac{1+y|\d_{\sigma}\eta|^2}{\eta^2} h + \dfrac{ 2y \d_\sigma \eta  }{\eta} \d_{\sigma} h.
	\end{align*}
	Applying the above calculations and (\ref{detah}) on $\mG[\eta](\cdot)$ given in (\ref{mG-flat}), we obtain the desired expression.
\end{proof}
With few lines of calculation, one can also verify the following proposition:
\begin{proposition}\label{prop:dAh}
Let $\sinc(\theta) = \frac{\sin\theta}{\theta} $ and $\omega(\theta)\vcentcolon= \frac{\theta \cos\theta - \sin\theta}{\theta^2} = \d_{\theta} \sinc(\theta)$. Then
\begin{gather*}
\dif_{\eta} \gamma \cdot h = h \frac{\sin(y\eta)+y\eta \cos(y\eta)}{4},\\
\dif_{\eta} A \cdot h = y^2 h \omega(y\eta) \begin{pmatrix}
\eta^2 & -y \eta \d_\sigma \eta  \\
-y \eta \d_\sigma \eta  & 1+ y^2 |\d_{\sigma} \eta|^2 
\end{pmatrix} + y\sinc(y\eta) \begin{pmatrix}
2\eta h & -y \d_{\sigma}(\eta h)  \\
-y \d_{\sigma}(\eta h)  & 2y^2 \d_{\sigma} \eta \d_{\sigma} h
\end{pmatrix}.
\end{gather*}
\end{proposition}

\subsection{Conical harmonic extension of \texorpdfstring{$h\mB+V$}{hB+V}}
\begin{lemma}\label{lemma:varpi}
Let $v(\sigma,y)$ be solution to (\ref{vflat}). Define $\varpi_h(\sigma,y)$ to be
\begin{equation}\label{varpih}
    \varpi_h\vcentcolon= 
    h\dfrac{y\d_{y} v}{\eta} + \d_{\sigma} v - \d_{\sigma} \eta \dfrac{y \d_{y} v}{\eta}.
\end{equation}
Then $\varpi_h(\sigma,y)$ solves the following problem
\begin{equation}\label{varpihEq}
\begin{aligned}
	&-\mL_{\eta} \varpi_h = \textnormal{\div}\big( (\dif_{\eta} A\cdot h) \cdot \nabla v \big) - (\dif_{\eta} \gamma \cdot h) v && \text{for } \ (\sigma,y)\in \mS,\\
	&\varpi_h(\sigma,1) = h\mB + V, \qquad \d_y \varpi_h(\sigma,0) = 0 && \text{for } \ \sigma\in\R.
\end{aligned}
\end{equation}
\end{lemma}  
\begin{proof}
Let $a=a(y,h,\eta,\d_{\sigma}\eta)\in \R$ be a scalar function, and denote $\d$ as either $\d_{\sigma}$ or $\d_{y}$. It then follows from $\mL_{\eta} v = 0$ that
\begin{align}\label{adv}
	&-\mL_{\eta}(a \d v) = - \div \big( A \nabla (a \d v) \big) + \gamma a \d v\\ 
	=& -\div \big( A [\nabla, a] \d v \big) - \div \big( [a A, \d] \nabla v \big) - \d [\div,  a] ( A \nabla v ) - \d \{ a \cdot \mL_{\eta}v \} - \d(a \gamma) v \nonumber \\
	=& -\div \big(\d v A \nabla a \big) + \div \big( \d(aA) \cdot \nabla v \big) - \d\big\{ (\nabla a)^{\top} A \nabla v \big\} - \d(a\gamma) v.\nonumber
\end{align}
Suppose two scalar functions $\alpha$, $\beta$, and set $\varpi_h = \alpha \d_\sigma v + \beta \d_y v$. By the previous calculation (\ref{dvh}), one has
\begin{gather}\label{LVarpi}
	-\mL_{\eta} \varpi_h = \div(\mQ \nabla v ) - q v, \qquad \text{where} \\
	\mQ \vcentcolon= \d_{\sigma}(\alpha A) + \d_{y}(\beta A) - \bigg(
		A\nabla \alpha, A\nabla \beta \bigg) - 
	\begin{pmatrix}
		(A\nabla \alpha )^{\top}\\
		(A\nabla \beta)^{\top}	
	\end{pmatrix}, \quad q \vcentcolon= \div \bigg(\gamma \begin{pmatrix}
	\alpha\\ \beta
\end{pmatrix}\bigg).\nonumber
\end{gather} 
Our aim is to determine $\alpha$, $\beta$ such that 
\begin{equation}\label{Qq}
	\mQ = \dif_{\eta} A \cdot h, \qquad q = \dif_{\eta} \gamma \cdot h.
\end{equation}
Note that since $A$ is symmetric, so is $\mQ$. Thus the matrix equality $\mQ = \dif_{\eta} A \cdot h$ consists 3 equations. First, $q=\dif_{\eta}\gamma \cdot h$  implies that
\begin{equation}\label{divAB}
	\d_{\sigma} \big( \alpha \eta \sin(y\eta)  \big) + \d_{y} \big( \beta \eta \sin(y\eta)  \big) = h \big\{ \sin(y\eta) + y\eta \cos(y\eta) \big\}.
\end{equation} 
Substituting this into the $(1,1)$ entry of the equation $\mQ=\dif_{\eta}A\cdot h$, it follows that $\alpha$ solves the differential equation:
\begin{equation*}
	\dfrac{\eta}{\d_{\sigma}\eta} \dfrac{\d_{\sigma}\alpha}{\alpha} = y \dfrac{\d_y \alpha}{\alpha}.
\end{equation*}
By separation of variable, $\alpha = (y\eta)^r$ for some $r\in\R$, which will be determined later. Putting $\alpha=(y\eta)^r$ back into (\ref{divAB}) then integrating with respect to $y$, we also get
\begin{equation}\label{beta}
	\beta = \dfrac{y h}{\eta} - y^{r+1} \eta^{r-1} \d_{\sigma} \eta + \dfrac{f(\sigma)}{\eta \sin(y\eta)}, 
\end{equation}
where $f(\sigma)$ is a function of $\sigma$ only, which is to be determined. Putting $\alpha=(y\eta)^r$ and (\ref{beta}) into the $(1,2)$ and $(2,2)$ entries of $\mQ=\dif_{\eta}A\cdot h$, we get $2$ differential equations: 
\begin{align*}
	\d_{\sigma} f = - r y^{r-1} \eta^{r-1} \sin(y\eta), \quad y \d_{\sigma} \eta \d_{\sigma} f + \cot(y\eta) f + r y^{r} \eta^{r-1} \d_{\sigma} \eta \sin(y\eta) = 0.
\end{align*} 
To satisfy the above, we set $r=0$ and $f(\sigma)\equiv 0$. Then $\alpha=1$, $\beta = \frac{y}{\eta}(h-\d_{\sigma}\eta)$ and
\begin{equation*}
	\varpi_h = \alpha \d_\sigma v + \beta \d_y v = h\dfrac{y\d_y v}{\eta} + \d_{\sigma} v - \d_{\sigma} \eta \dfrac{y\d_y v}{\eta}.
\end{equation*}
With this construction, we get $-\mL_{\eta}\varpi_h = \div \big( (\dif_\eta v \cdot h) \nabla v \big) - (\dif_{\eta} \gamma \cdot h) v$ by (\ref{LVarpi}) and (\ref{Qq}). Since $\d_{y} v \vert_{y=0}=0$, it can be verified that $\d_y \varpi_h \vert_{y=0} = 0$. Finally, from the definition (\ref{BV}), $\varpi_h\vert_{y=1} = h \mB + V$. This concludes the proof.
\end{proof}
Define $\Upsilon_h\vcentcolon= \varpi_h-\dif_{\eta} v\cdot h$. Then combining Lemma \ref{lemma:varpi}, and the fact that $\dif_{\eta} v\cdot h$ solves the problem (\ref{dvh}), we obtain the following:
\begin{corollary}\label{corol:hBext} $\Upsilon_h\vcentcolon= \varpi_h-\dif_{\eta} v\cdot h$ solves the problem
\begin{subequations}\label{varthetah}
\begin{align}
&-\mL_{\eta} \Upsilon_{h} \equiv -\div\big( A \cdot\nabla \Upsilon_h\big) + \gamma \Upsilon_h = 0 && \text{for } \ (\sigma,y)\in \mS,\\
&\Upsilon_h(\sigma,1) = h\mB + V, \qquad \d_y \Upsilon_h(\sigma,0)=0 && \text{for } \ \sigma\in\R. 
\end{align}
\end{subequations}
\end{corollary}
\subsection{Shape derivative of \texorpdfstring{$\mG[\eta]$}{G[eta]}, proof of Theorem \ref{thm:shape}.}
In this subsection, we use Proposition \ref{prop:dG} and Corollary \ref{corol:hBext} to conclude the proof for Theorem \ref{thm:shape}. 
\begin{proof}[Proof of Theorem \ref{thm:shape}]
Since $\Upsilon_h= \varpi_h-\dif_{\eta} v\cdot h$ solves (\ref{varthetah}), it follows from the definition of auxiliary DN operator (\ref{mG-flat}) that
\begin{align*}
\mG[\eta](h\mathcal{B}+V)
=& \Big\{ \dfrac{1+y|\d_{\sigma} \eta|^2}{\eta}\d_y \varpi_h - \d_{\sigma} \eta \d_{\sigma} \varpi_h \Big\}\Big\vert_{y=1}\\
&- \Big\{ \dfrac{1+y|\d_{\sigma}\eta|^2}{\eta}\d_y (\dif_{\eta} v\cdot h) - \d_{\sigma} \eta \d_{\sigma} (\dif_{\eta} v\cdot h) \Big\}\Big\vert_{y=1}.
\end{align*}
Combining the above with Proposition \ref{prop:dG}, we obtain that
\begin{align}\label{dGhtemp}
\dif_{\eta} \mG[\eta](\phi)\cdot h =& - G[\eta](h \mathcal{B} + V) + \Big\{ \dfrac{1+y|\d_{\sigma}\eta|^2}{\eta}\d_y \varpi_h - \d_{\sigma} \eta \d_{\sigma} \varpi_h \Big\}\Big\vert_{y=1} \\
&+ \Big\{ \d_{\sigma} h \dfrac{2y\d_{\sigma} \eta}{\eta} \d_y v - \d_{\sigma} h \d_{\sigma} v - h\dfrac{1+y|\d_{\sigma} \eta|^2}{\eta^2} \d_y v \Big\}\Big\vert_{y=1} \nonumber\\
=\vcentcolon& - \mG[\eta](h \mathcal{B}+V) + \big\{ (\star) + (\mathrm{I}) \big\} \vert_{y=1}.\nonumber
\end{align}
By the definition of $\varpi_h$ in (\ref{varpih}), we split the term $(\star)$ into two parts:
\begin{align*}
	(\star) =& \bigg\{ \dfrac{1+y|\d_{\sigma}\eta|^2}{\eta}\d_y \Big( \d_\sigma v -  \d_\sigma \eta \frac{y\d_y v}{\eta}\Big) - \d_{\sigma} \eta \d_{\sigma} \Big( \d_\sigma v -  \d_\sigma \eta \frac{y\d_y v}{\eta}\Big)  \bigg\}\\
	&+\Big\{ \dfrac{1+y|\d_{\sigma}\eta|^2}{\eta}\d_y \big( h \frac{y\d_y v}{\eta}\big) - \d_{\sigma} \eta \d_{\sigma} \big( h \frac{y\d_y v}{\eta}\big) \Big\} = \vcentcolon (\mathrm{II}) + (\mathrm{III}).
\end{align*}
Expanding the term $(\mathrm{III})$ and using equation (\ref{Leta}), we get
\begin{align*}
	(\mathrm{III}) 
	=& h\dfrac{v}{4} - h\d_\sigma^2 v + h \dfrac{y \d_\sigma \eta}{\eta} \d_\sigma \d_y v + \dfrac{h \d_y v}{\eta^2} + h \dfrac{y \d_\sigma^2 \eta}{\eta} \d_y v  -  h\dfrac{\cot(y\eta) }{\eta} \d_y v \\ & - \d_\sigma h \dfrac{y \d_\sigma \eta}{\eta} \d_y v + (y-1) \dfrac{h\d_y^2 v}{\eta^2}.
\end{align*}
Combing this with $(\mathrm{I})$ we obtain that
\begin{align*}
	 (\mathrm{III}) + (\mathrm{I}) 
	 =& - \d_\sigma \bigg\{ h \Big( \d_\sigma v - \d_\sigma \eta\dfrac{y \d_y v}{\eta} \Big) \bigg\}  + h\Big( \dfrac{v}{4} - \dfrac{\d_y v}{\eta}\cot(y\eta) \Big) + (y-1) \dfrac{h\d_y^2 v}{\eta^2}.
\end{align*}
Moreover, expanding $(\mathrm{II})$, and using equation (\ref{Leta}), we also get
\begin{align*}
	(\mathrm{II}) 
	= & \d_\sigma \Big( \dfrac{\d_y v}{\eta} \Big) - \d_\sigma \eta \Big( \dfrac{v}{4}  - \dfrac{\d_y v}{\eta} \cot(y\eta) \Big) - (y-1) \dfrac{\d_\sigma \eta}{\eta^2} \d_y^2 v.
\end{align*}
From Remark \ref{rem:BV'}, we have $\mB = \frac{\d_y v}{\eta} \big\vert_{y=1}$ and $ V = \big\{ \d_\sigma v - \d_\sigma \eta \frac{\d_y v}{\eta} \big\}\big\vert_{y=1}$. Using these, and substituting $(\mathrm{I})$--$(\mathrm{III})$ into (\ref{dGhtemp}), it follows that
\begin{align*}
	\dif_{\eta}\mG[\eta](\phi)\cdot h =& - \mG[\eta](h\mB + V) + \big\{ (\mathrm{I}) + (\mathrm{II}) + (\mathrm{III}) \big\}\big\vert_{y=1}\\
	=& - \mG[\eta](h\mB + V) - \d_\sigma (h V - \mB) + (h-\d_\sigma \eta) \Big\{ \frac{\phi}{4} - \mB \cot(\eta) \Big\}.
\end{align*}
This concludes the proof.
\end{proof}
\begin{corollary}\label{corol:RdGh}
Set $\tilde{\mG}[\eta](\phi) \vcentcolon= \sin(\eta)\cdot \mG[\eta](\phi)$. Then for $ h\in\mC_{c}^{\infty}(\R)$,
\begin{equation*}
\dif_\eta \tilde{\mG}[\eta](\phi)\cdot h = - \tilde{\mG}[\eta](h\mB + V) - \d_{\sigma} \big\{ (hV-\mB) \sin(\eta) \big\} +  \dfrac{\sin(\eta)}{4}  (h-\d_\sigma\eta) \phi .
\end{equation*}
\end{corollary}
\begin{proof}
Using Theorem \ref{thm:shape}, and the product rule, we have that
\begin{align*}
\dif_{\eta} \tilde{\mG}[\eta](\phi) \cdot h =& h \cos(\eta) \mG[\eta](\phi) + \sin(\eta)  \dif_{\eta} \mG[\eta](\phi) \cdot h\\
 =& - \tilde{\mG}[\eta]\big( h\mB + V \big) - \d_{\sigma} \big\{ (hV-\mB) \sin(\eta) \big\} + (h-\d_{\sigma}\eta) \dfrac{\phi}{4} \sin(\eta)\\
 &+ h \cos(\eta) \big\{ \mG[\eta](\phi) + V \d_\sigma \eta - \mB \big\}.
\end{align*}
By definition (\ref{BV}), we get $\mG[\eta](\phi)+ V \d_\sigma \eta - \mB =0$, which concludes the proof.
\end{proof}

\begin{corollary}[Perturbation with respect to cone]
Let the surface $\eta\equiv \ath\in (0,\pi)$ be the cone with slope $\tan \ath$. Then for $h\in\mC_c^{\infty}(\R)$,
\begin{align*}
\dif_{\eta} \mG[\eta](\phi)\cdot h \vert_{\eta=\ath} = - \mG[\ath]\big( h \mG[\ath](\phi) \big) - \d_\sigma \big( h \d_\sigma \phi \big) + h \Big\{ \frac{\phi}{4} - \cot(\ath)\cdot\mG[\ath](\phi)  \Big\} .
\end{align*}
\end{corollary}
\begin{proof}
For $\eta\equiv \ath$, the terms reduces to $\mB = \mG[\ath](\phi)$ and $V=\d_\sigma \phi$ by (\ref{BV}). Moreover, we also have $\d_\sigma \mG[\ath](\phi) = \mG[\ath](\d_\sigma \phi)$. The statement follows by combining these with Theorem \ref{thm:shape}.
\end{proof}

\subsection{Cancellation in \texorpdfstring{$\mG[\eta](\mB+V)+\d_{\sigma}(V-\mB)$}{G[eta](B+V)+D(V-B)}}
Suppose $v$ is a solution to (\ref{vflat}), and let $\mB$, $V$ be given in (\ref{BV}). Recall the extension operator $E_{\eta}[\cdot]$ in Definition \ref{def:solOp}. Set 
\begin{equation}\label{diffW}
	\varpi \vcentcolon= \dfrac{y}{\eta} \d_y v +\d_\sigma v - \dfrac{y\d_\sigma \eta}{\eta}\d_y v, \qquad \mathcal{W} \vcentcolon= \varpi - E_{\eta}[\mB+V]. 
\end{equation} 
Then repeating the same procedure for the proofs of Lemma \ref{lemma:varpi} and Corollary \ref{corol:hBext}, it can be verified that $\mathcal{W}$ solves:
\begin{gather}
\left\{\begin{aligned}
	&-\mL_{\eta} \mathcal{W} = \div\big( M \nabla v \big) - \dfrac{\sin(y\eta)+y\eta \cos(y\eta)}{4} v && \text{in } \ (\sigma,y)\in\mS,\\
	& \mathcal{W}(\sigma,1)=0, \qquad \d_y \mathcal{W}(\sigma,0)=0 && \text{for } \ \sigma \in\R,
\end{aligned}\right.\\
\text{where } \ \ M \vcentcolon= y^2 h \omega(y\eta) \begin{pmatrix}
	\eta^2 & -y \eta \d_\sigma \eta  \\
	-y \eta \d_\sigma \eta  & 1+ y^2 |\d_{\sigma} \eta|^2 
\end{pmatrix} + y\sinc(y\eta) \begin{pmatrix}
	2\eta & -y \d_{\sigma}\eta  \\
	-y \d_{\sigma}\eta  & 0
\end{pmatrix}.\nonumber
\end{gather}

\begin{proposition}\label{prop:cancelEllip}
Suppose $\teta\in H^{s+\frac{1}{2}}(\R)$ with $s>\frac{5}{2}$, and let $v$ be solution to (\ref{vflat}). Then $\mathcal{W}$ defined in (\ref{diffW}) satisfies the regularity:
\begin{equation*}
	\mathcal{W} \in \mC_y^0\big((0,1]; H_\sigma^{s}(\R)\big), \qquad \d_y \mathcal{W} \in \mC_y^0\big((0,1]; H_\sigma^{s-1}(\R)\big).
\end{equation*}
\end{proposition}
\begin{proof}
	One observes that $\mathcal{W}$ satisfies the same equation as $w$ in (\ref{w-D}), except for source term at the right hand side. Thus we can apply the exact same argument in the proof of Lemma \ref{lemma:ellip} to obtain:
	\begin{align*}
		&\int_{0}^{1}\!\! \Big\{ \| \mathcal{W}(\cdot,y) \|_{H^{s+\frac{1}{2}}}^2 + \|\d_y \mathcal{W}(\cdot,y)\|_{H^{s-\frac{1}{2}}}^2 \Big\} \, y \dif y\\
		\le& C \Big|\dfrac{\fU_s(\teta)}{\fl(\teta)}\Big|^{2s+1}\!\!\! \int_{0}^{1}\!\! \| \nabla v(\cdot,t) \|_{H^{s-\frac{1}{2}}}^2 \, y \dif y  \le C \Big|\dfrac{\fU_s(\teta)}{\fl(\teta)}\Big|^{4s+2} \|\phi\|_{H^{s}}^2,
	\end{align*}
	where we also used Corollary \ref{corol:vH1Est}. Moreover, similarly to Lemma \ref{lemma:vHigh}, one can also obtain using the standard elliptic boundary estimate argument and Lemma \ref{lemma:vHigh} that
	\begin{equation*}
		\int_{0}^{1}\!\!\! y^3 \|\d_y^2 \mathcal{W}(\cdot, y)\|_{H^{s-\frac{3}{2}}} \dif y \le C\Big|\dfrac{\fU_s(\teta)}{\fl(\teta)}\Big|^{2s-1}\!\!\!\! \int_{0}^{1}\!\! y^3 \|\nabla \d_y v(\cdot, y)\|_{H^{s-\frac{3}{2}}} \dif y \le C \big(\|\teta\|_{H^{s+\frac{1}{2}}}\big) \|\phi\|_{H^s}^2.
	\end{equation*}
	Thus we can apply the interpolation embedding theorem, Proposition \ref{prop:bochner} to obtain
	\begin{equation*}
		\mathcal{W} \in \mC_{y}^0 \big( (0,1] ; H_{\sigma}^{s}(\R) \big), \qquad \d_y \mathcal{W} \in \mC_{y}^0 \big( (0,1] ; H_{\sigma}^{s-1}(\R) \big).
	\end{equation*} 
	This concludes the proof.
\end{proof}

\begin{theorem}\label{lemma:cancel}
	Suppose $\teta\in H^{s+\frac{1}{2}}(\R)$ with $s>\frac{5}{2}$. Let $v$ be solution to (\ref{vflat}), and let $\mB$, $V$ be given in (\ref{BV}). Then
	\begin{equation*}
		\mG[\eta](\mB+V) + \d_{\sigma}(V-\mB) \in H_{\sigma}^{s-1}(\R).
	\end{equation*}  
\end{theorem}
\begin{proof}
	By the definition of auxiliary DN operator and $E_{\eta}[\cdot]$, it follows that
	\begin{align*}
		\mG[\eta](\mB + V) = \Big\{ \dfrac{1+y |\d_{\sigma}\eta|^2}{\eta} \d_y E_{\eta}[\mB+V] - \d_\sigma \eta \d_\sigma E_{\eta}[\mB+V] \Big\}\Big\vert_{y=1}. 
	\end{align*}
	Substituting $\mathcal{W}=\varpi-E_{\eta}[\mB+V]$ in to the above, we have:
	\begin{align*}
		\mG[\eta](\mB+V) =& \Big\{ \dfrac{1+y |\d_{\sigma}\eta|^2}{\eta} \d_y \varpi - \d_\sigma \eta \d_\sigma \varpi \Big\}\Big\vert_{y=1} -\Big\{ \dfrac{1+y |\d_{\sigma}\eta|^2}{\eta} \d_y \mathcal{W} - \d_\sigma \eta \d_\sigma \mathcal{W} \Big\}\Big\vert_{y=1}. 
	\end{align*}
	Recall that $\varpi = \frac{y}{\eta} \d_y v +\d_\sigma v - \frac{y\d_\sigma \eta}{\eta}\d_y v$. Moreover, from Remark \ref{rem:BV'}, one has $\mB = \frac{\d_y v}{\eta}\vert_{y=1}$ and $V=\big(\d_\sigma v - \d_\sigma \eta \frac{\d_y v}{\eta}\big)\big\vert_{y=1}$. Using this, and the fact that $-\mL_{\eta}v=0$, we derive:
	\begin{align*}
		&\Big\{ \dfrac{1+y |\d_{\sigma}\eta|^2}{\eta} \d_y \varpi - \d_\sigma \eta \d_\sigma \varpi \Big\}\Big\vert_{y=1}\\
		=& \bigg\{\d_{\sigma} \Big( \dfrac{\d_y v}{\eta} \!-\! \d_{\sigma} v \!+\! \d_{\sigma}\eta \dfrac{y \d_y v}{\eta} \Big) \!+\! \dfrac{1\!+\! y |\d_{\sigma}\eta|^2}{\eta^2}\d_y v  \!+\! (1\!-\!\d_{\sigma}\eta) \Big( \dfrac{v}{4} \!-\! \dfrac{\cot(y\eta)}{\eta} \d_y v\Big) \bigg\}\bigg\vert_{y=1}\\
		=& \d_{\sigma}(\mB- V) + \dfrac{1+y |\d_{\sigma}\eta|^2}{\eta^2} \d_y v \Big\vert_{y=1} + (1-\d_{\sigma}\eta) \Big\{ \dfrac{v}{4} - \dfrac{\d_y v}{\eta} \cot(y\eta) \Big\}\Big\vert_{y=1}.
	\end{align*}
	Putting this into the expression for $\mG[\eta](\mB+V)$, it follows that
	\begin{align*}
		\mG[\eta](\mB+V) + \d_{\sigma} (V-\mB) =& \dfrac{1+y |\d_{\sigma}\eta|^2}{\eta^2} \d_y v \Big\vert_{y=1} + (1-\d_{\sigma}\eta) \Big\{ \dfrac{v}{4} - \dfrac{\d_y v}{\eta} \cot(y\eta) \Big\}\Big\vert_{y=1}\\
		&-\Big\{ \dfrac{1+y |\d_{\sigma}\eta|^2}{\eta} \d_y \mathcal{W} - \d_\sigma \eta \d_\sigma \mathcal{W} \Big\}\Big\vert_{y=1}.
	\end{align*}
	By the regularities of $v$ and $\mathcal{W}$ given in Lemmas \ref{lemma:vHigh} and Proposition \ref{prop:cancelEllip}, we conclude that the term $\mG[\eta](\mB+V) + \d_{\sigma} (V-\mB) \in H^{s-1}(\R)$.  
\end{proof}


\section{Stokes expansion for small perturbations}\label{sec:Stokes}
In this section we present the Taylor expansion of DN operator for a small data following \cite{Craig}.
In general, the analyticity of a functional 
$G[a]\vcentcolon= \mG[\ath+a]$ mapping complex Banach spaces $ B_1$, to $B_2$,
is equivalent to the existence of the Taylor expansion
\[
G[a]=\sum_{k=0}^\infty \Lambda_k[a],\quad \Lambda_k(a)=(a_1, \dots, a_k)|_{a_i=a},
\]
where $a$ is in a sufficiently 
small ball in $B_1$, and $\Lambda_k[a]:B_1^k\to B_2$
is a multilinear map (in $a_i$) satisfying the estimate
\[
\| \Lambda_k(a_1, \dots, a_k)\|_{B_2}\le C^k\prod_{i=1}^k\|a_i\|_{B_1},
\]
see \cite{Coifman}. Let $\Phi_m(\sigma, \theta)=e^{i\sigma m}P_{im-\frac12}(\cos\theta)$.
It is clear that $\Phi_m$ solves \eqref{ellip-cf-a}.
Note that $\d_\theta\Phi_m(\sigma, 0)=0$, and 
\[
P_s'(0)=\frac{-2\sqrt\pi}{\Gamma(-\frac s2)\Gamma(\frac12+\frac s2)}.
\]
We expand $\theta=\eta=\theta_*+\tilde{\eta}$, where 
$\tilde{\eta}, \tilde{\eta}_\sigma$ are assumed to be small.
Then if we assume that $G[\tilde{\eta}]=\sum_{j=0}^\infty G_j[\tilde{\eta}]$, where $G_j[\tilde{\eta}]$ is homogeneous of degree $j$ in $\tilde{\eta}$
we get from 
\eqref{mG}
\begin{align}\label{Taylor-exp-small}
& e^{i\sigma m} P'_{im-\frac12}(\cos(\theta_*+\tilde{\eta}))\sin(\theta_*+\tilde{\eta})
-im e^{i\sigma m}P_{im-\frac12}(\cos(\theta_*+\tilde{\eta}))\tilde{\eta}_\sigma\\
=&
(\sum_{j=0}^\infty G_j[\tilde{\eta}])\left(e^{i\sigma m} P_{im-\frac12}(\cos(\theta_*+\tilde{\eta}))\right)
=
(\sum_{j=0}^\infty G_j[\tilde{\eta}])
\left(e^{i\sigma m} \sum_{k=0}^\infty \frac{a_k(m, \theta_*)}{k!}\tilde{\eta}^k\right), \nonumber
\end{align}
where we assumed that $P_{im-\frac12}(\cos(\theta_*+\tilde{\eta}))$
has the following expansion
\[
P_{im-\frac12}(\cos(\theta_*+\tilde{\eta}))
=
\sum_{k=0}^\infty \frac{a_k(m, \theta_*)}{k!}\tilde{\eta}^k.
\]
Then by Taylor's theorem, 
\begin{equation*}
P'_{im-\frac12}(\cos(\theta_*+\tilde{\eta}))\sin(\theta_*+\tilde{\eta})=
\frac \d{\d\tilde{\eta}}
\sum_{k=0}^\infty \frac{a_k(m, \theta_*)}{k!}\tilde{\eta}^k
=
\sum_{k=1}^\infty \frac{a_k(m, \theta_*)}{(k-1)!}\tilde{\eta}^{k-1}.
\end{equation*}
Consequently, substituting these into \eqref{Taylor-exp-small}
we obtain 
\begin{align*}
&e^{i\sigma m}\left(\sum_{k=1}^\infty \frac{a_k(m, \theta_*)}{(k-1)!}\tilde{\eta}^{k-1}
-
im\frac{\d \tilde{\eta}}{\d \sigma} \sum_{k=0}^\infty \frac{a_k(m, \theta_*)}{k!}\tilde{\eta}^k
 \right)\\
 &=
(\sum_{j=0}^\infty G_j[\tilde{\eta}])
\left(e^{i\sigma m} \sum_{k=0}^\infty \frac{a_k(m, \theta_*)}{k!}\tilde{\eta}^k\right) =
\sum_{\ell=0}^\infty\sum_{j=0}^\ell G_j[\tilde{\eta}]\frac{a_{\ell-j}(m, \theta_*)}{(\ell-j)!}\tilde{\eta}^{\ell-j}
.
\end{align*}
Equating the terms homogenous of degree $k$ in $\tilde{\eta}$ on both sides we obtain
\begin{equation*}
G_\ell[\tilde{\eta}]=\frac1{a_0(m, \theta_*)}\left[\left(\frac{a_{\ell+1}(m, \theta_*)}{\ell!}-im\tilde{\eta}_\sigma \frac{a_\ell(m, \theta_*)}{\ell!}\right)\tilde{\eta}^\ell
-
\sum_{j=0}^{\ell-1} G_j[\tilde{\eta}]\frac{a_{\ell-j}(m, \theta_*)}{(\ell-j)!}\tilde{\eta}^{\ell-j}\right].
\end{equation*}
From here we get the symbol of $G_\ell$ as 
\begin{equation*}
\frac1{a_0(D, \theta_*)}\left[\left(\frac{a_{\ell+1}(D, \theta_*)}{\ell!}-iD\tilde{\eta}_\sigma \frac{a_\ell(D, \theta_*)}{\ell!}\right)\tilde{\eta}^\ell
-
\sum_{j=0}^{\ell-1} G_j[\tilde{\eta}]\frac{a_{\ell-j}(D, \theta_*)}{(\ell-j)!}\tilde{\eta}^{\ell-j}\right].
\end{equation*}
If $\tilde{\eta}=0$, then from above expression, we get that 
\begin{equation*}
G[0](e^{i\sigma m})=e^{i\sigma m}\frac{a_1(m, \theta_*)}{a_0(m, \theta_*)}.
\end{equation*}
Then, expanding $\phi$ in Fourier  series, we recover the symbol of the DN operator for general $\phi$, 
\[
G[0](\phi)= \frac{a_1(D, \theta_*)}{a_0(D, \theta_*)}.
\]
In order to compute the other terms $G_j[\tilde{\eta}]$ we need to compute 
the coefficients $a_k(m, \theta_*)$.
It is well known that $P_\ell(z)=\frac1{\Gamma(1)}{}_2F_1(-\ell, \ell+1; 1, \frac{1-z}2)$,
where ${}_2F_1(-\ell, \ell+1; 1, \frac{1-z}2)$ is the hypergeometric 
series. Taking $\ell=im-\frac12$ we get 
\begin{eqnarray*}
P_{im-\frac12}(z)
&=&
\frac1{\Gamma(1)}{}_2F_1(-\ell, \ell+1; 1, \frac{1-z}2)\\
&=&
\frac1{\Gamma(-\ell)\Gamma(\ell+1)}\sum_{n=0}^\infty\frac{\Gamma(n-\ell)\Gamma(n+\ell+1)}{\Gamma(n+1)n!}\left(\frac{1-z}2\right)^n.
\end{eqnarray*}
First notice that from  \eqref{GammaF} 
$$
\Gamma(-\ell)\Gamma(\ell+1)=\frac\pi{\sin(\pi(-\ell))}=\frac\pi{\sin(\pi(\frac12-im))}=\frac\pi{\cosh (\pi m)}.
$$
Next we use $\Gamma(z)\Gamma(\bar z)=|\Gamma(z)|^2$, and hence 
\[
\Gamma(n-\ell)\Gamma(n+\ell+1)
=
\Gamma(n+\tfrac12-im)\Gamma(n+\tfrac12+im)\in \R, 
\]
and the identity (see \eqref{GaHalf})
\[
\left|\Gamma(n+\tfrac12\pm im)\right|=\frac\pi{\cosh(\pi m)}\prod_{k=1}^n\left(\left( k+\tfrac12 \right)^2+m^2\right).
\]
Combining we find that 
\begin{equation*}
P_{im-\frac12}(z)
=
\frac\pi{\cosh(\pi m)}\sum_{n=0}^\infty\frac{1}{(n!)^2}
\prod_{k=1}^n\left(\left( k+\tfrac12 \right)^2+m^2\right)^2
\left(\frac{1-z}2\right)^n.
\end{equation*}
To finish the computation it remains to expand the $z$ term in $\tilde{\eta}$,
where $z=\cos(\theta_*+\tilde{\eta})$.
We have 
\begin{gather*}
	\begin{aligned}
		&\frac{1-z}2
		=
		\sin^2\frac{\theta_*+\tilde{\eta}}2
		=
		\cos^2\frac{\theta_*}2\left(\tan\frac{\theta_*}2\cos\frac{\tilde{\eta}}{2}+\sin\frac{\tilde{\eta}}{2}\right)^2\\
		=&
		\cos^2\frac{\theta_*}2
		\left(
		\tan\frac{\theta_*}2\sum_{p=0}^\infty\frac{(-1)^p}{(2p)!}\left(\frac{\tilde{\eta}}{2}\right)^{2p}+\sum_{p=0}^\infty\frac{(-1)^p}{(2p+1)!}\left(\frac{\tilde{\eta}}{2}\right)^{2p+1}
		\right)^2 = \big(\mathscr{S}[\ath,\tilde{\eta}]\big)^2 \cos^2\frac{\theta_*}2,
	\end{aligned}\\
	\text{where } \quad  \mathscr{S}[\ath,\tilde{\eta}] \vcentcolon= \tan\frac{\theta_*}2+\frac{\tilde{\eta}}{2}-\tan\frac{\theta_*}2\frac1{2!}\left(\frac{\tilde{\eta}}{2}\right)^{2}
	+\cdots .
\end{gather*}
Hence we get
\begin{align*}
P_{im-\frac12}(\cos(\theta_*+\tilde{\eta}))
=\frac\pi{\cosh(\pi m)}\sum_{n=0}^\infty\frac{1}{(n!)^2}
\prod_{k=1}^n\left[\left[ k+\tfrac{1}{2} \right]^2+m^2\right]^2 \big(\mathscr{S}[\ath,\tilde{\eta}]\big)^{2n}
\cos^{2n}\frac{\theta_*}{2} .
\end{align*}
For the first three coefficients we have the explicit forms 
\begin{gather*}
a_0(m, \theta_*)
=
\frac\pi{\cosh(\pi m)}\sum_{n=0}^\infty\frac{\mathfrak{Q}(n,m)}{(n!)^2}
 \sin^{2n}\frac{\theta_*}2,\\
a_1(m, \theta_*)
=
\frac\pi{\cosh(\pi m)}\sum_{n=0}^\infty\frac{\mathfrak{Q}(n,m)}{(n!)^2}
n \cos^{2n}\frac{\theta_*}2  \tan^{2n-1}\frac{\theta_*}2,\\
a_2(m, \theta_*)
=
\frac\pi{\cosh(\pi m)}\sum_{n=0}^\infty\frac{\mathfrak{Q}(n,m)}{(n!)^2}
\frac n2 \cos^{2n}\frac{\theta_*}2 \left[(2n-1) \tan^{2n-2}\frac{\theta_*}2 
-\frac n2 \tan^{2n}\frac{\theta_*}2\right],\\
\text{where } \quad \mathfrak{Q}(n,m)\vcentcolon= \prod_{k=1}^n\left[\left[ k+\tfrac12 \right]^2+m^2\right]^2.
\end{gather*}

\appendix




\section{Sobolev Inequalities}\label{append:sobo}
The following Sobolev embedding theorems can be found in the book \cite{Tools}.
\subsection{Sobolev inequalities}
\begin{proposition}\label{prop:bochner}
	Let $d\ge 1$, and let $I\subseteq \R$ be a connected interval. Suppose the function $u(x,y)\vcentcolon \R^d \times [0,1] \to \R$ is such that
	\begin{equation*}
		u \in \mathrm{X}_2^{s+\frac{1}{2}}(I;\R^d), \qquad \d_y f \in \mathrm{X}_2^{s-\frac{1}{2}}(I;\R^d), \quad \text{where}  \quad \mathrm{X}_p^{k}(I;\R^d) \vcentcolon= L^p_{y}\big(I;H_x^{k}(\R^d)\big). 
	\end{equation*}
	Then $u\in \mC_y^{0}\big(I;H_x^s(\R^d)\big)$ and $ u \in L^{\infty}_y\big(I;H_x^s(\R^d)\big)$, and there exists a constant $C>0$ independent of $u$ such that
	\begin{equation*}
		\sup\limits_{y\in I} \| u(\cdot, y) \|_{H^s(\R^d)} \le C \| u \|_{\mathrm{X}_2^{s-\frac{1}{2}}(I;\R^d)}^{\frac{1}{2}} \| \d_y u \|_{\mathrm{X}_2^{s-\frac{1}{2}}(I;\R^d)}^{\frac{1}{2}}.
	\end{equation*}
\end{proposition}

\begin{proposition}[Sobolev Embedding for composite]\label{prop:Sobcomp}
	Assume $u\in L^{\infty}\cap H^{s}(\R^d)$ is real valued for $s\ge0$. If $F\in\mC^{\infty}(\R)$ and $F(0)=0$, then $F(u)\in L^{\infty}\cap H^{s}(\R^d)$, and
	\begin{equation*}
		\|F(u)\|_{H^{s}(\R^d)} \le \sup\limits_{x\in\R} |F^{\prime}(x)| \| u \|_{H^{s}(\R^d)}.
	\end{equation*}
\end{proposition}

\begin{proposition}[Sobolev Embedding for classical product]\label{prop:clprod}
	Assume $u_1\in H^{\alpha}(\R^d)$ and $u_2\in H^{\beta}(\R^d)$ for $\alpha+\beta\ge 0$. If $ s\le \min\{\alpha,\beta\}$ and $s\le \alpha+\beta - \frac{d}{2}$, then $u_1u_2\in H^{s}(\R^d)$ and there exists a constant $C>0$ independent of $u_1$, $u_2$ such that
	\begin{equation*}
		\|u_1 u_2\|_{H^{s}(\R^d)} \le C \|u_1\|_{H^{\alpha}(\R^d)} \|u_2\|_{H^{\beta}(\R^d)}.
	\end{equation*}
\end{proposition}

\subsection{Sobolev spaces on cone and strip domains}\label{append:C2S}
Consider the coordinate transformation law:
\begin{equation}\label{rhoSigma}
	r =  \mathcal{P}(\sigma) \vcentcolon= e^{-\sigma}, \quad \sigma = \mathcal{P}^{-1}(r) \vcentcolon= -\ln r \quad \text{for } \ r\in(0,\infty) \ \text{ and } \ \sigma\in \R.
\end{equation}
Then $r\to 0 $ as $ \sigma \to \infty$ and $r \to \infty $ as $ \sigma \to -\infty$. For functions $F(r)\vcentcolon(0,\infty)\to \R$ and $f(\sigma)\vcentcolon \R \to (0,\infty)$, we define the adjoint maps:
\begin{equation}\label{Pstar}
	\big(\mathcal{P}_{\ast}F\big)(\sigma) \vcentcolon= F\big(\mathcal{P}(\sigma)\big), \qquad \big(\mathcal{P}_{\ast}^{-1}f\big)(r) \vcentcolon= f\big( \mathcal{P}^{-1}(r) \big).
\end{equation}
Set $f(\sigma)\vcentcolon= F\big(\mathcal{P}(\sigma)\big) =F\big(e^{-\sigma}\big)$. Then under the transformation (\ref{rhoSigma}), one has
\begin{equation*}
	\int_{\R} |f(\sigma)|^2\, \dif \sigma = 	\int_{0}^{\infty} |F(r)|^2 \,\dfrac{\dif r}{r}. 
\end{equation*}
Moreover, by Fa\`a di Bruno's formula:
\begin{align*}
	\d_\sigma^m f(\sigma) 
	=& (-1)^m \sum_{k=1}^m c_{m,k} \cdot e^{-k\sigma} (\d_r^k F)\vert_{r=e^{-\sigma}}, \ \text{ where } \
	c_{m,k} \vcentcolon= B_{m,k}(\!\!\underbrace{1,1,\dotsc,1}_{(m-k+1)\textnormal{-times}}\!\!).   
\end{align*}
Here $B_{m,k}$ is the $(m,k)$-th Bell's polynomial given by
\begin{equation*}
	B_{m,k}(x_1,\dotsc,x_{m-k+1}) \vcentcolon= \sum \dfrac{m!}{j_1!\cdots j_{m-k+1}!} \Big(\dfrac{x_1}{1!}\Big)^{j_1} \!\! \cdots  \Big(\dfrac{x_{m-k+1}}{(m-k+1)!}\Big)^{j_{m-k+1}},
\end{equation*}
where the sum is taken over all sequences $j_1,j_2,\dotsc,j_{m-k+1}$ of non-negative integers satisfying the constraints:
\begin{gather*}
	j_1+j_2 + \cdots + j_{m-k+1} = k, \quad \text{ and } \quad j_1 + 2 j_2 + \cdots + (m-k+1) j_{m-k+1} = m.
\end{gather*}
Similarly, differentiating $F(r)=f(-\ln r)$, we also get 
\begin{equation*}
	(\d_r^m F) (r) = \frac{(-1)^m}{r^m} \sum_{k=1}^m  \tilde{c}_{m,k} (\d_\sigma^k f)\vert_{\sigma=-\ln r}, \ \text{ where } \ \tilde{c}_{m,k}\vcentcolon= B_{m,k}\big(0!,1!,\dotsc, (m-k)!\big).
\end{equation*} 
These calculation implies that there exists some constant $C=C(m)>0$ such that
\begin{equation}\label{dRdS}
	C^{-1} \sum_{k=1}^m | \d_\sigma^k f(\sigma) |^2  \le \sum_{k=1}^{m} r^{2k} |\d_r^k F (r)|^2 \big\vert_{r=\mathcal{P}(\sigma)}
	\le C \sum_{k=1}^m | \d_\sigma^k f(\sigma) |^2. 
\end{equation}
Therefore, we have the following equivalence of Sobolev spaces:
\begin{proposition}\label{prop:normEquiv}
	Let $\mathcal{P}_{\ast}$ and $\mathcal{P}_{\ast}^{-1}$ be the mappings defined in (\ref{rhoSigma})--(\ref{Pstar}). For $m\in\mathbb{N}$, define the Sobolev spaces:
	\begin{gather*}
		H_{\sigma}^{m} \vcentcolon= \bigg\{ f\in H_{\text{loc}}^m(\R) \ \bigg\vert \ \|f\|_{H_{\sigma}^m}^2\vcentcolon= \sum_{k=1}^m \int_{\R}\! |\d_\sigma^k f|^2 \, \dif \sigma <\infty  \bigg\},\\
		\tilde{H}_r^{m} \vcentcolon= \bigg\{ F\in H_{\text{loc}}^{m}(0,\infty) \ \bigg\vert \ \| F \|_{\tilde{H}_{r}^m}^2 \vcentcolon= \sum_{k=1}^m \int_{0}^{\infty}\!\!\!  r^{2k-1}|\d_r^k F (r)|^2 \, \dif r <\infty \bigg\}.
	\end{gather*}
	Then $\mathcal{P}_{\ast} \vcentcolon \tilde{H}_r^m  \to H_{\sigma}^{m}$ is one-to-one and onto, with the inverse map $\mathcal{P}_{\ast}^{-1}$.
\end{proposition}

\section{Fourier Multiplier and Commutator Estimates}
The commutator estimates presented below can be found in Appendix B.2.1 of \cite{lannes2}. First we define the following class of Fourier multiplier symbols:
\begin{definition}\label{def:FM}
	A Fourier multiplier $\sigma(D)$ is said to belong to the symbol class $\fS^k(\R^d)$ for some $k\in\R$ if $\zeta\mapsto \sigma(\zeta)\in\mathbb{C}$ is smooth and satisfies
	\begin{equation*}
		\sup\limits_{\zeta\in\R^d} \absm{\zeta}^{|\beta|-s}|\d_\zeta^{\beta}\sigma(\zeta)|<\infty \qquad \text{for all } \ (\zeta,\beta)\in \R^d\times \mathbb{N}^d.
	\end{equation*}
	Moreover, one also defines the following semi-norm for $\sigma\in\fS^k$:
	\begin{equation*}
		\mathcal{N}^k(\sigma) \vcentcolon= \sup\Big\{ \absm{\zeta}^{|\beta|-s} |\d_\zeta^\beta \sigma(\zeta)| \ \Big\vert\ \beta\in\mathbb{N} \ \text{ and } \ |\beta|\le 2+d+\lceil \tfrac{d}{2} \rceil \Big\},
	\end{equation*}
	where $\lceil \alpha \rceil$ denotes the ceiling function for $\alpha\in R$.
\end{definition}
\begin{proposition}\label{prop:commu}
	Let $t_0>\frac{d}{2}$, $k\ge 0$ and $\sigma\in\fS^k$. If $ 0 \le k \le t_0+1$ and $f\in H^{t_0+1}(\R^d)$, then for all $g\in H^{k-1}(\R^d)$,
	\begin{equation}
		\big\| [\sigma(D),f] g \big\|_{L^2(\R^d)} \le C \mathcal{N}^k(\sigma) \|\nabla f\|_{H^{t_0}(\R^d)} \|g\|_{H^{k-1}(\R^d)},
	\end{equation} 
	for some constant $C>0$ independent of $\sigma$, $f$, and $g$.
\end{proposition}
If the symbol takes the special form of $\sigma(\zeta)=\absm{\zeta}^k$, then the following improvements on Proposition \ref{prop:commu} holds
\begin{proposition}\label{prop:commuLam}
	Let $t_0>\frac{d}{2}$. Assume $-\frac{d}{2}<k\le t_0+\delta$. Then for all $f\in H^{t_0+\delta}(\R^d)$ and $g\in H^{k-\delta}(\R^d)$, one has
	\begin{equation*}
		\|[\absm{D}^k,f]g\|_{L^2(\R^d)} \le C \|f\|_{H^{t_0+\delta}(\R^d)} \|g\|_{H^{k-\delta}(\R^d)},
	\end{equation*}
	for some constant $C>0$ independent of $f$, and $g$.
\end{proposition}



\section{Equations for Taylor's Cone in Spherical Coordinate}\label{append:AS-Zakh}

 
\subsection{Equations in spherical coordinate \texorpdfstring{$(r,\theta,\alpha)$}{(rho,theta,alpha)}}\label{subsec:AS}
Consider $n=3$. Let $(r,\theta,\alpha)$ be the spherical coordinate system, with $r\in(0,\infty)$ being radial; $\theta\in[0,\pi]$ being polar; $\alpha\in[0,2\pi)$ being azimuthal. For the axially symmetric flow,  $\Psi=\Psi(t,r,\theta)$, hence $$\vv(t,r,\theta)=
\nabla \Psi =\hr \d_{r} \Psi + \hth \dfrac{\d_\theta\Psi}{r}.$$ 
The axially symmetric conical surface can be parametrised by $(r,\alpha)\in(0,\infty)\times[0,2\pi)$, and $\theta=\Theta(t,r)$ is a function of $(t,r)$. It is then described implicitly by the equation:
\begin{equation*}
F(t,x,y,z)\vcentcolon=\arccos\Big(\dfrac{z}{\sqrt{x^2+y^2+z^2}}\Big)-\Theta\big(t,\sqrt{x^2+y^2+z^2}\big) = \theta - \Theta(t,r) = 0,
\end{equation*}
for some function $\Theta(t,r)\vcentcolon [0,\infty)\times(0,\infty) \to (0,\pi)$. The position vector of surface is described by $\X =r \hr(t,r,\alpha)$, where the unit vector $\hrho$ is a function of $(t,r,\alpha)$ given by:
\begin{equation*}
\hr(t,r,\alpha) = \hx \cos\alpha \sin\big(\Theta(t,r)\big) + \hy \sin\alpha \sin\big(\Theta(t,r)\big) + \hz \cos\big(\Theta(t,r)\big).
\end{equation*}
It follows by chain rule that
\begin{equation*}
\d_r \X = \hr + r \d_\theta \hr \d_r \Theta = \hr + \hth r \d_r \Theta, \qquad \d_t \X = r \d_\theta \hr \d_t \Theta = \hth r \d_t \Theta.
\end{equation*}
The normal vector is given by
\begin{align*}
\normal =& \d_\alpha \X \times \d_r \X = r \d_\alpha \hr \times (\hr + \hth r \d_r \Theta) = (r \sin \Theta ) \hat{\alpha} \times (\hr + \hth r \d_r \Theta)\\
=& r \sin \Theta ( \hth - \hr r \d_r \Theta ).
\end{align*}
The unit normal $\un(t,r)$ is given by
\begin{equation*}
\un(t,r) = \dfrac{\normal}{|\normal|} = \dfrac{\hth - \hr r \d_{r}\Theta}{\sqrt{1+r^2 |\d_r \Theta|^2}}.
\end{equation*}
Thus the kinematic equation $F_t+\nabla \Psi\cdot\nabla F =0$ and the Bernoulli equation can be written as follows, for all $(t,r)\in[0,T]\times[0,\infty)$,
\begin{subequations}\label{thetaEq}
\begin{gather}
	\d_t \Theta(t,r) - \big\{ \dfrac{\d_{\theta} \Psi}{r^2} - \d_r \Theta \d_r \Psi \big\} \Big\vert_{\theta=\Theta(t,r)} = 0,\label{thetaEq-a}\\
	\big\{ \partial_t \Psi + \dfrac{|\d_{r} \Psi|^2}{2}  + \dfrac{|\d_\theta \Psi|^2}{2r^2} + p - \dfrac{\epsilon}{2\rho}|\nabla \varphi|^2 \big\}\Big\vert_{\theta=\Theta(t,r)} = 0.\label{thetaEq-b}
\end{gather}
\end{subequations}
Moreover, we define $\psi(r,t)\vcentcolon= \Psi(t,r,\Theta(t,r))$. Then by the chain rule it follows that for $(r,t)\in(0,\infty)\times[0,T]$,
\begin{subequations}\label{Rzeta}
\begin{align}
&\d_t \psi (t,r) = \big(\d_t \Psi + \d_t \Theta \d_\theta \Psi \big)\big\vert_{\theta=\Theta(t,r)},\label{Rzeta-a}\\
&\d_r \psi (t,r) = \big(\d_r \Psi + \d_r \Theta \d_\theta \Psi \big)\big\vert_{\theta=\Theta(t,r)}.\label{Rzeta-b}
\end{align}
\end{subequations}

\subsection{Reformulation using Dirichlet-Neumann operator}
The Dirichlet-Neuman operator associated with $\d\Omega(t)=\{\theta-\Theta(t,r)=0\}$ acting on $\psi$ is given by
\begin{equation}\label{AS-DN}
G[\Theta](\psi) \vcentcolon= \big(\nabla F \cdot \nabla \Psi\big) \big\vert_{F(t,\x)=0} = \Big\{ \dfrac{\d_{\theta} \Psi}{r^2} - \d_{r}\Theta \d_r \Psi \Big\}\Big\vert_{\theta=\Theta(t,r)}.
\end{equation}
By (\ref{Rzeta-b}) and (\ref{AS-DN}), it follows that
\begin{subequations}\label{thetaPhi}
\begin{align}
\d_\theta \Psi \big(t,r, \Theta(t,r)\big) =& \dfrac{r^2 \{\d_r \Theta \d_r \psi + G[\Theta](\psi)\}}{1+r^2 |\d_r \Theta|^2},\label{thetaPhi-a}\\
\d_r \Psi \big(t,r, \Theta(t,r)\big) =& \dfrac{\d_r \psi - r^2 \d_r \Theta \cdot G[\Theta](\psi)}{1+r^2 |\d_r \Theta|^2}.\label{thetaPhi-b}
\end{align}
\end{subequations}
Substituting (\ref{Rzeta-a}) into (\ref{thetaEq-b}), then using (\ref{thetaEq-a}), it follows that
\begin{equation*}
	\d_t \psi + \Big\{ \d_r \Theta \d_r \Psi \d_\theta \Psi + \dfrac{|\d_r \Psi|^2}{2} - \dfrac{|\d_\theta \Psi|^2}{2 r^2} + p - \dfrac{\epsilon}{2\rho}|\nabla \varphi|^2 \Big\}\Big\vert_{\theta=\Theta(t,r)} = 0.
\end{equation*}
Completing the square and using (\ref{Rzeta-b}), we have
\begin{align*}
\d_r \Theta \d_r \Psi \d_\theta \Psi =& \dfrac{1}{2} \big( \d_r \Psi + \d_r \Theta \d_\theta \Psi \big)^2 - \dfrac{|\d_r \Psi|^2}{2} -\dfrac{|\d_r \Theta \d_\theta \Psi|^2}{2} \\
=& \dfrac{|\d_r \psi|^2}{2}  - \dfrac{|\d_r \Psi|^2}{2} -\dfrac{|\d_r \Theta \d_\theta \Psi|^2}{2}.
\end{align*}
Thus we have
\begin{align*}
	\d_t\psi + \dfrac{|\d_r \psi|^2}{2} - \dfrac{1+r^2 |\d_r \Theta|^2}{2r^2}|\d_\theta \Psi|^2\vert_{\theta=\Theta(t,r)} + \Big\{ p - \dfrac{\epsilon}{2\rho}|\nabla \varphi|^2 \Big\}\Big\vert_{\theta=\Theta(t,r)} =0.
\end{align*}
Substituting (\ref{thetaPhi-a}) into the above, we obtain the equations
\begin{subequations}\label{Zak-cone}
\begin{gather}
\d_t \Theta - G[\Theta](\psi) = 0,\label{Kin-cone}\\
\d_t\psi + \dfrac{|\d_{r} \psi|^2}{2} - \dfrac{\big| r \d_{r} \Theta \d_{r} \psi + r G[\Theta](\psi) \big|^2}{2(1+r^2|\d_{r}\Theta|^2)} + \Big\{ p - \dfrac{\epsilon}{2\rho}|\nabla \varphi|^2 \Big\}\Big\vert_{\theta=\Theta(t,r)} = 0. \label{Dyn-cone}
\end{gather}
\end{subequations}

\subsection{Surface tension and Electric field}\label{append:STEF}
If the surface tension is taken into account, then according to the Young-Laplace law, the jump in pressure across the free boundary, $[p]= p-p_0$ is proportional to the mean curvature, denoted as $H(\Theta)$. Thus one has
\begin{equation*}
	p\vert_{\theta=\Theta(t,r)} = p_0-\dfrac{\kappa}{\rho} H(\Theta), 
\end{equation*} 
where $\kappa>0$ is the constant coefficient of surface tension, $\rho>0$ is the constant density, and $p_0$ is the ambient pressure, with $p_0=0$ if the fluid is surrounded by vacuum. For surfaces in $3$-dimensional space, $H$ is given by the divergence of unit normal vector $\un$ restricted on the surface. Using the divergence for spherical coordinates, one gets:
\begin{align*}
	H(\Theta)&=-\dfrac{1}{2}\div \un\Big\vert_{\x\in\d\Omega} = -\dfrac{1}{2} \Big\{ \dfrac{1}{r^2}\d_r\big(r^2 (\hat{r}\cdot\un)\big) + \dfrac{1}{r\sin\theta}\d_\theta\big( (\hth\cdot \un)\sin\theta\big) \Big\}\Big\vert_{\theta=\Theta(r)}\\
	&= \d_r\Big( \dfrac{r \d_r \Theta}{2\sqrt{1+r^2 |\d_r\Theta|^2}} \Big) + \dfrac{\d_r \Theta}{\sqrt{1+r^2 |\d_r\Theta|^2}} - \dfrac{\cot\Theta}{2r \sqrt{1+r^2 |\d_r\Theta|^2}}.
\end{align*}
Next, we consider the electric field term in (\ref{Dyn-cone}). Let $\varphi(r,\theta)$ be the solution to
\begin{equation*}
	\left\{\begin{aligned}
		&\d_r (r^2 \d_r \varphi \sin\theta) + \d_\theta(\d_\theta \varphi \sin\theta)  = 0  &&\text{for } \   \Theta(r) \le \theta <  \pi,\\
		&\varphi\big( r, \Theta(r) \big) = 0, \quad \d_\theta \varphi\big(r,\pi\big) = 0 &&\text{for } \ r\in(0,\infty),\\
		&|\d_r\varphi| \to 0 &&\text{as } \ r\to \infty.
	\end{aligned}
	 \right.
\end{equation*}
For some arbitrary constant $C\neq 0$, we define for $\theta\in [\Theta(r),\pi)$,
\begin{gather*}
	\varphi_{\textrm{T}}(r,\theta)\vcentcolon= C \sqrt{r} P_{1/2}(-\cos\theta), \qquad \Xi(r,\theta)\vcentcolon=\varphi(r,\theta)-\varphi_{\textrm{T}}(r,\theta),\\
	 \xi(r)\vcentcolon= -C \sqrt{r} P_{1/2}\big(-\cos \Theta(r) \big) = -\varphi_{\textrm{T}}(r,\theta)\vert_{\theta=\Theta(r)}.
\end{gather*}
Note that if $(\Theta,\Psi,\varphi)\equiv(\ath,0,\varphi_{\textrm{T}})$ is a stationary solution to (\ref{Zak-cone}), then $\ath$ must be the unique root of $P_{1/2}(-\cos\ath)=0$, and the constant $C$ in $\varphi_{\textrm{T}}$ takes the form:
\begin{equation*}
	C = C_{\ast}\vcentcolon= -\dfrac{\sqrt{\kappa\cot\ath}}{\sqrt{\epsilon}P_{1/2}^1(-\cos\ath)} \quad \text{ and } \quad \ath \approx 0.2738\, \pi.
\end{equation*}
By construction $\Xi(r,\theta)$ solves the Dirichlet boundary value problem:
\begin{equation*}\left\{
	\begin{aligned}
		&\d_r(r^2 \d_r \Xi \sin\theta ) +  \d_\theta (\d_\theta \Xi \sin\theta ) = 0 && \text{for } \ \Theta(r) \le \theta < \pi,\\
		& \Xi\big(r,\Theta(r)\big) = \xi(r), \quad \d_\theta \Xi(r,\pi)=0 && \text{for } \ r\in(0,\infty).
	\end{aligned}\right.
\end{equation*}
With few lines of calculation, one sees that
\begin{equation*}
	\d_r \varphi_{\mathrm{T}} \vert_{\theta=\Theta(r)} = \dfrac{\varphi_{T}}{2r}\Big\vert_{\theta=\Theta(r)}= -\dfrac{\xi}{2r}, \qquad \d_\theta \varphi_{\mathrm{T}}\vert_{\theta=\Theta(r)} = \dfrac{\xi - 2r \d_r \xi}{2r \d_r \Theta}. 
\end{equation*}
Moreover, denote the function $\mathfrak{P}(\theta) \vcentcolon= C\frac{\dif}{\dif \theta}\big(P_{1/2}(-\cos\theta)\big)$, then
\begin{equation*}
	\mathfrak{P}(\theta)= -C P_{1/2}^{1}(-\cos\theta) \qquad \text{and} \qquad \dfrac{\xi-2r\d_r \xi}{2r\d_r \Theta} = \sqrt{r} \mathfrak{P}(\theta). 
\end{equation*}
Using the above identities, and the relations (\ref{thetaPhi-a})--(\ref{thetaPhi-b}), it follows that
\begin{gather*}
	|\nabla \Xi|^2\vert_{\theta=\Theta(r)}
	= \dfrac{|\d_r \xi|^2 + r^2 \big(G[\Theta](\xi)\big)^2}{1+r^2 |\d_r \Theta|^2}, \qquad |\nabla \varphi_{\textrm{T}}|^2\vert_{\theta=\Theta(r)} 
	= \dfrac{\xi^2}{4r^2}+ \dfrac{|\mathfrak{P}(\Theta)|^2}{r},\\
	2\nabla \Xi \cdot \nabla \varphi_{\textrm{T}}\vert_{\theta=\Theta(r)}
	= \dfrac{2r^2 \d_r \Theta \d_r \xi G[\Theta](\xi)}{1+r^2|\d_r\Theta|^2} -2 \sqrt{r}\mathfrak{P}(\Theta) G[\Theta](\xi)  - \dfrac{2|\d_r \xi|^2}{1+r^2|\d_r\Theta|^2}.
\end{gather*}
Combining these together, we obtain:
\begin{align*}
	|\nabla \varphi|^2\vert_{\theta=\Theta(r)} =& \big\{ |\nabla \Xi|^2 + 2 \nabla \Xi \cdot \nabla \varphi_{\textrm{T}} + |\nabla\varphi_{\textrm{T}}|^2 \big\}\big\vert_{\theta=\Theta(r)}\\
	=& \dfrac{\xi^2}{4r^2} + \Big| r G[\Theta](\xi) - \dfrac{\mathfrak{P}(\Theta)}{\sqrt{r}} \Big|^2 - \dfrac{\big| \d_r \xi - r^2 \d_r \Theta G[\Theta](\xi) \big|^2}{1+ r^2 |\d_r\Theta|^2}.
\end{align*}
Observe that $\xi(r) = -C \sqrt{r} P_{1/2}\big(-\cos\Theta(r)\big)$ is solely determined by $\Theta(r)$. Thus the force due to electric field $\frac{\epsilon}{2\rho}|\nabla \varphi|^2\vert_{\theta=\Theta}$ is considered as a functional acting on $\Theta(r)$.


\begin{bibdiv}
\begin{biblist}

\bib{Coifman}{article}{
   author={Coifman, R. R.},
   author={Meyer, Yves},
   title={Nonlinear harmonic analysis and analytic dependence},
   conference={
      title={Pseudodifferential operators and applications},
      address={Notre Dame, Ind.},
      date={1984},
   },
   book={
      series={Proc. Sympos. Pure Math.},
      volume={43},
      publisher={Amer. Math. Soc., Providence, RI},
   },
   date={1985},
   pages={71--78},
   review={\MR{812284}},
   doi={10.1090/pspum/043/812284},
}



\bib{Tools}{book}{
   author={Alazard, Thomas},
   author={Zuily, Claude},
   title={Tools and problems in partial differential equations},
   series={Universitext},
   publisher={Springer, Cham},
   date={2020},
   pages={xii+357},
   review={\MR{4174398}},
   doi={10.1007/978-3-030-50284-3},
}

\bib{Craig}{article}{
   author={Craig, Walter},
   author={Groves, Mark D.},
   title={Hamiltonian long-wave approximations to the water-wave problem},
   journal={Wave Motion},
   volume={19},
   date={1994},
   number={4},
   pages={367--389},
   issn={0165-2125},
   review={\MR{1285131}},
   doi={10.1016/0165-2125(94)90003-5},
}



\bib{GR2007}{book}{
	author={Gradshteyn, I. S.},
	author={Ryzhik, I. M.},
	title={Table of integrals, series, and products},
	edition={7},
	note={Translated from the Russian;
		Translation edited and with a preface by Alan Jeffrey and Daniel
		Zwillinger;
		With one CD-ROM (Windows, Macintosh and UNIX)},
	publisher={Elsevier/Academic Press, Amsterdam},
	date={2007},
	pages={xlviii+1171},
	isbn={978-0-12-373637-6},
	isbn={0-12-373637-4},
	review={\MR{2360010}},
}










\bib{lannes2}{book}{
	author={Lannes, David},
	title={The water waves problem},
	series={Mathematical Surveys and Monographs},
	volume={188},
	note={Mathematical analysis and asymptotics},
	publisher={American Mathematical Society, Providence, RI},
	date={2013},
	pages={xx+321},
	isbn={978-0-8218-9470-5},
	review={\MR{3060183}},
	doi={10.1090/surv/188},
}




\bib{MOS1966}{book}{
   author={Magnus, Wilhelm},
   author={Oberhettinger, Fritz},
   author={Soni, Raj Pal},
   title={Formulas and theorems for the special functions of mathematical
   physics},
   series={Die Grundlehren der mathematischen Wissenschaften, Band 52},
   edition={Third enlarged edition},
   publisher={Springer-Verlag New York, Inc., New York},
   date={1966},
   pages={viii+508},
   review={\MR{232968}},
}





\bib{Olver1997}{book}{
	AUTHOR = {Olver, Frank W. J.},
	TITLE = {Asymptotics and special functions},
	SERIES = {AKP Classics},
	NOTE = {Reprint of the 1974 original [Academic Press, New York;
		MR0435697 (55 \#8655)]},
	PUBLISHER = {A K Peters, Ltd., Wellesley, MA},
	YEAR = {1997},
	PAGES = {xviii+572},
	ISBN = {1-56881-069-5},
	REVIEW = {\MR{1429619}},
}



\bib{Taylor}{article}{
author = {Geoffrey Ingram Taylor},  
journal = {Proc. R. Soc. Lond.},
number = {A. 280},
pages = {383--397},
title = {Disintegration of water drops in an electric field},
year = {1964},
}


\bib{Zak1968}{article}{
	author={Zakharov, V.~E.},
	title={Stability of periodic waves of finite amplitude on the surface of a deep fluid},
	journal={J. Appl. Mech. Tech. Phys.},
	volume={9},
	date={1968},
	pages={190--194},
	issn={1573-8620},
	doi={10.1007/BF00913182},
}


\bib{ZK66}{book}{
	AUTHOR = {Zhurina, M.~I. and Karmazina, L.},
	TITLE = {Tables and formulae for the spherical functions
		{$P\sp{m}\sb{-1/2+i\tau }\,(z)$}},
	NOTE = {Translated by E. L. Albasiny},
	PUBLISHER = {Pergamon Press, Oxford-New York-Paris},
	YEAR = {1966},
	PAGES = {vii+108},
	REVIEW = {\MR{0203097}},
}



\end{biblist}
\end{bibdiv}
\end{document}